\documentclass[12pt]{amsart}
\usepackage{amsmath,amsthm}
\usepackage{amssymb}
\usepackage[all]{xypic}

\SelectTips{cm}{12}
\UseTips{}
\usepackage{euscript}
\usepackage{hyperref} 
\begin{document}
%
%  Types of ``theorems''
%
\theoremstyle{plain}
\swapnumbers
    \newtheorem{thm}{Theorem}[section]
    \newtheorem{prop}[thm]{Proposition}
    \newtheorem{lemma}[thm]{Lemma}
    \newtheorem{cor}[thm]{Corollary}
    \newtheorem{subsec}[thm]{}
    \newtheorem*{thma}{Theorem A}
    \newtheorem*{thmb}{Theorem B}
    \newtheorem*{propc}{Proposition C}
\theoremstyle{definition}
    \newtheorem{defn}[thm]{Definition}
    \newtheorem{example}[thm]{Example}
    \newtheorem{examples}[thm]{Examples}
    \newtheorem{notn}[thm]{Notation}
\theoremstyle{remark}
    \newtheorem{remark}[thm]{Remark}
    \newtheorem{ack}[thm]{Acknowledgements}
%
%   Special diagram/display environments (Mark Johnson)
%
\newenvironment{myeq}[1][]
{\stepcounter{thm}\begin{equation}\tag{\thethm}{#1}}
{\end{equation}}
\newcommand{\mydiag}[2][]{\myeq[#1]\xymatrix{#2}}
\newcommand{\mydiagram}[2][]
{\stepcounter{thm}\begin{equation}
     \tag{\thethm}{#1}\vcenter{\xymatrix{#2}}\end{equation}}
%use:  \mydiagram[\label{``label''}]{``xy-pic syntax''}
%
\newenvironment{mysubsection}[2][]
{\begin{subsec}\begin{upshape}\begin{bfseries}{#2.}
\end{bfseries}{#1}}
{\end{upshape}\end{subsec}}
\newenvironment{mysubsect}[2][]
{\begin{subsec}\begin{upshape}\begin{bfseries}{#2\vsn.}
\end{bfseries}{#1}}
{\end{upshape}\end{subsec}}
\newcommand{\sect}{\setcounter{thm}{0}\section}
\newcommand{\wh}{\ -- \ }
\newcommand{\wwh}{-- \ }
\newcommand{\w}[2][ ]{\ \ensuremath{#2}{#1}\ }
\newcommand{\ww}[1]{\ \ensuremath{#1}}
\newcommand{\wwb}[1]{\ \ensuremath{(#1)}-}
\newcommand{\wb}[2][ ]{\ (\ensuremath{#2}){#1}\ }
\newcommand{\wref}[2][ ]{\ \eqref{#2}{#1}\ }
\newcommand{\wwref}[2][ ]{\ \eqref{#2}{#1}}
\newcommand{\wbref}[2][ ]{\eqref{#2}{#1}}
%
%             arrows
%
\newcommand{\xra}[1]{\xrightarrow{#1}}
\newcommand{\xla}[1]{\xleftarrow{#1}}
\newcommand{\xsim}{\xrightarrow{\sim}}
\newcommand{\hra}{\hookrightarrow}
\newcommand{\epic}{\to\hspace{-5 mm}\to}
\newcommand{\adj}[2]{\substack{{#1}\\ \rightleftharpoons \\ {#2}}}
\newcommand{\ccsub}[1]{\circ_{#1}}
\newcommand{\DEF}{:=}
\newcommand{\EQUIV}{\Leftrightarrow}
\newcommand{\hsp}{\hspace{10 mm}}
\newcommand{\hs}{\hspace{5 mm}}
\newcommand{\hsm}{\hspace{3 mm}}
\newcommand{\vsm}{\vspace{3 mm}}
\newcommand{\vsn}{\vspace{1 mm}}
\newcommand{\vs}{\vspace{4 mm}}
\newcommand{\rest}[1]{\lvert_{#1}}
\newcommand{\lra}[1]{\langle{#1}\rangle}
\newcommand{\lin}[1]{\{{#1}\}}
\newcommand{\llrr}[1]{\langle\!\langle{#1}\rangle\!\rangle}
\newcommand{\llrrL}[1]{\llrr{#1}_{\Lambda}}
\newcommand{\q}[1]{^{({#1})}}
\newcommand{\li}[1]{_{({#1})}}
\newcommand{\var}{\varepsilon}
\newcommand{\hotimes}{\hat{\otimes}}
\newcommand{\vrp}{\varphi}
%
%             Caligraphic
%
\newcommand{\A}{{\EuScript A}}
\newcommand{\C}{{\mathcal C}}
\newcommand{\D}{{\mathcal D}}
\newcommand{\E}{{\EuScript E}}
\newcommand{\F}{{\EuScript F}}
\newcommand{\G}{{\mathcal G}}
\newcommand{\K}{{\mathcal K}}
\newcommand{\LL}{{\mathcal L}}
\newcommand{\LQ}{\LL_{\bQ}}
\newcommand{\M}{{\mathcal M}}
\newcommand{\MA}{\M_{\A}}
\newcommand{\OA}{\OO^{\A}}
\newcommand{\OAp}{\OA_{+}}
\newcommand{\OAh}{\widehat{\OO}^{\A}}
\newcommand{\N}{{\EuScript N}}
\newcommand{\OO}{{\EuScript O}}
\newcommand{\Op}{\OO_{+}}
\newcommand{\Oh}{\widehat{\OO}}
\newcommand{\MO}{\M_{\OO}}
\newcommand{\MOp}{\M_{\Op}}
\newcommand{\PP}{{\mathcal P}}
\newcommand{\Pe}[1]{{\EuScript Pe}\sp{#1}}
\newcommand{\PO}{{\mathcal PO}}
\newcommand{\QQ}{{\mathcal Q}}
\newcommand{\QQp}{{\mathcal Q}_{+}}
\newcommand{\Ss}{{\mathcal S}}
\newcommand{\Sa}{\Ss_{\ast}}
\newcommand{\Sr}{\Ss_{\ast}^{\red}}
\newcommand{\TT}{{\mathcal T}}
\newcommand{\Ta}{\TT_{\ast}}
\newcommand{\U}{{\mathcal U}}
\newcommand{\V}{{\mathcal V}}
\newcommand{\eV}{{\EuScript V}}
\newcommand{\eVn}[1]{\eV\lra{#1}}
\newcommand{\W}{{\mathcal W}}
\newcommand{\X}{{\mathcal X}}
\newcommand{\eX}{{\EuScript X}}
\newcommand{\Y}{{\mathcal Y}}
\newcommand{\eY}{{\EuScript Y}}
\newcommand{\Z}{{\mathcal Z}}
\newcommand{\eZ}{{\EuScript Z}}
%
%             categories and functors
%
\newcommand{\hy}[2]{{#1}\text{-}{#2}}
%
%    Names of categories and functors
%
\newcommand{\Alg}[1]{{#1}\text{-}{\EuScript Alg}}
\newcommand{\Ab}{{\EuScript Ab}}
\newcommand{\Abgp}{{\Ab\Grp}}
\newcommand{\AbL}{\Ab_{\Lambda}}
\newcommand{\Cat}{{\EuScript Cat}}
\newcommand{\Ch}{{\EuScript Ch}}
\newcommand{\DiGa}{{\EuScript D}i{\EuScript G}\sb{\ast}}
\newcommand{\Grp}{{\EuScript Gp}}
\newcommand{\Gpd}{{\EuScript Gpd}}
\newcommand{\RM}[1]{\hy{{#1}}{\EuScript Mod}}
\newcommand{\RL}{\RM{\Lambda}}
\newcommand{\Set}{{\EuScript Set}}
\newcommand{\Sets}{\Set_{\ast}}
\newcommand{\SC}{\hy{\Ss}{\Cat}}
\newcommand{\VC}{\hy{\V}{\Cat}}
%
%            Algebras
%
\newcommand{\Tal}[1][ ]{$\Theta$-algebra{#1}}
\newcommand{\TAlg}{\Alg{\Theta}}
\newcommand{\Pa}[1][ ]{$\Pi$-algebra{#1}}
\newcommand{\PAa}[1][ ]{$\PiA$-algebra{#1}}
\newcommand{\PAMa}[1][ ]{$\PiAM$-algebra{#1}}
\newcommand{\PAMas}[1][ ]{$\PiAM$-algebras{#1}}
\newcommand{\PiA}{\Pi\sb{\A}}
\newcommand{\PiAM}{\Pi\sb{\A}\sp{\M}}
\newcommand{\PAlg}{\Alg{\Pi}}
\newcommand{\PAAlg}{\Alg{\PiA}}
\newcommand{\PAMAlg}{\Alg{\PiAM}}
%
%             Enriched categories
%
\newcommand{\COC}{\hy{\CO}{\Cat}}
\newcommand{\GO}{(\Gpd,\OO)}
\newcommand{\GOC}{\hy{\GO}{\Cat}}
\newcommand{\GOp}{(\Gpd,\Op)}
\newcommand{\GOpC}{\hy{\GOp}{\Cat}}
\newcommand{\OC}{\hy{\OO}{\Cat}}
\newcommand{\SO}{(\Ss,\OO)}
\newcommand{\SaO}{(\Sa,\OO)}
\newcommand{\SOC}{\hy{\SO}{\Cat}}
\newcommand{\SaOC}{\hy{\SaO}{\Cat}}
\newcommand{\SOp}{(\Ss,\Op)}
\newcommand{\SOpC}{\hy{\SOp}{\Cat}}
\newcommand{\SaOp}{(\Sa,\Op)}
\newcommand{\SaOpC}{\hy{\SaOp}{\Cat}}
\newcommand{\HSO}[1]{H^{#1}\sb{\operatorname{SO}}}
\newcommand{\VO}{(\V,\OO)}
\newcommand{\VOC}{\hy{\VO}{\Cat}}
%
%    simplicial objects and categories
%
\newcommand{\co}[1]{c({#1})}
\newcommand{\Ad}{A_{\bullet}}
\newcommand{\Bd}{B_{\bullet}}
\newcommand{\tBL}{B\Lambda}
\newcommand{\Ed}{\E_{\bullet}}
\newcommand{\Fd}{F_{\bullet}}
\newcommand{\Gd}{G_{\bullet}}
\newcommand{\Kd}{\K_{\bullet}}
\newcommand{\Vd}{V_{\bullet}}
\newcommand{\whV}{\widehat{V}}
\newcommand{\hVd}[1]{\whV\sp{\lra{#1}}\sb{\bullet}}
\newcommand{\cVd}[1]{\breve{V}^{\lra{#1}}\sb{\bullet}}
\newcommand{\tVd}[1]{\tilde{V}^{\lra{#1}}\sb{\bullet}}
\newcommand{\qV}[2]{V^{\lra{#1}}_{#2}}
\newcommand{\qVd}[1]{\qV{#1}{\bullet}}
\newcommand{\qW}[2]{W^{({#1})}_{#2}}
\newcommand{\qWd}[1]{\qW{#1}{\bullet}}
\newcommand{\Wd}{W_{\bullet}}
\newcommand{\tWd}{\tilde{W}_{\bullet}}
\newcommand{\whW}{\widehat{W}}
\newcommand{\whWd}{\whW\sb{\bullet}}
\newcommand{\Xd}{X_{\bullet}}
\newcommand{\Yd}{Y_{\bullet}}
\newcommand{\sps}{semi-Postnikov section}
\newcommand{\qps}{quasi-Postnikov section}
%
%            B bold face
%
\newcommand{\bF}{\mathbb F}
\newcommand{\Fp}{\bF_{p}}
\newcommand{\bN}{{\mathbb N}}
\newcommand{\bQ}{{\mathbb Q}}
\newcommand{\bR}{{\mathbb R}}
\newcommand{\bZ}{{\mathbb Z}}
%
%         Lattices and diagrams
%
\newcommand{\Fs}{F_{s}}
\newcommand{\Gp}{\Gamma_{+}}
\newcommand{\GpG}{\Gp,\Gamma}
\newcommand{\cGp}{\check{\Gamma}_{+}}
\newcommand{\baK}{\overline{\K}}
\newcommand{\bM}{\overline{M}}
\newcommand{\bX}{\bar{X}}
\newcommand{\bY}[1]{\overline{Y}\sb{#1}}
\newcommand{\hX}[1]{\widehat{X}\sp{#1}}
\newcommand{\cX}{\overline{X}}
\newcommand{\tX}{\widetilde{X}}
%
%             homotopy & homology
%
\newcommand{\pis}{\pi_{\ast}}
\newcommand{\piul}[2]{\pi\sp{#1}\sb{#2}}
\newcommand{\piAn}[1]{\piul{\A}{#1}}
\newcommand{\piA}{\piAn{\ast}}
\newcommand{\hpi}{\piul{\A}{0}}
\newcommand{\pinat}[1]{\operatorname{\pi^{\natural}_{#1}}}
\newcommand{\Po}[1]{\mathbf{P}^{#1}}
\newcommand{\bE}[2]{\mathbf{E}({#1},{#2})}
\newcommand{\tE}[2]{E({#1},{#2})}
\newcommand{\tEL}[2]{E\sb{\Lambda}({#1},{#2})}
\newcommand{\tPo}[1]{\Po{#1}}
%
%          operators
%
\newcommand{\ab}{\sb{\operatorname{ab}}}
\newcommand{\arr}{\sb{\operatorname{arr}}}
\newcommand{\Arr}{\operatorname{Arr}}
\newcommand{\HAQ}[1]{H^{#1}}
\newcommand{\HL}[1]{H^{#1}\sb{\Lambda}}
\newcommand{\csk}[1]{\operatorname{csk}_{#1}}
\newcommand{\Coef}{\operatorname{Coef}}
\newcommand{\Cok}{\operatorname{Cok}}
\newcommand{\colim}{\operatorname{colim}}
\newcommand{\hocolim}{\operatorname{hocolim}}
\newcommand{\cone}[1]{\operatorname{Co}(#1)}
\newcommand{\sk}[1]{\operatorname{sk}_{#1}}
\newcommand{\cskc}[1]{\operatorname{cosk}^{c}_{#1}}
\newcommand{\fin}{\operatorname{fin}}
\newcommand{\hc}[1]{\operatorname{hc}_{#1}}
\newcommand{\ho}{\operatorname{ho}}
\newcommand{\holim}{\operatorname{holim}}
\newcommand{\Hom}{\operatorname{Hom}}
\newcommand{\uHom}{\underline{\Hom}}
\newcommand{\Id}{\operatorname{Id}}
\newcommand{\Image}{\operatorname{Im}}
\newcommand{\inc}{\operatornarme{inc}}
\newcommand{\init}{\operatorname{init}}
\newcommand{\Ker}{\operatorname{Ker}}
\newcommand{\vf}{v_{\fin}}
\newcommand{\vfi}[1]{v{#1}_{\fin}}
\newcommand{\vi}{v_{\init}}
\newcommand{\vin}[1]{v{#1}_{\init}}
\newcommand{\bstar}{\mbox{\large $\star$}}
\newcommand{\Mor}{\operatorname{Mor}}
\newcommand{\Obj}{\operatorname{Obj}\,}
\newcommand{\op}{\sp{\operatorname{op}}}
\newcommand{\pt}{\operatorname{pt}}
\newcommand{\red}{\operatorname{red}}
\newcommand{\wPh}[1]{\widetilde{\Phi}\sb{#1}}
%
%            Mapping spaces
%
\newcommand{\map}{\operatorname{map}}
\newcommand{\mapp}{\map\,}
\newcommand{\mapa}{\map_{\ast}}
%
%        spaces
%
\newcommand{\tg}[1]{\widetilde{\gamma}_{#1}}
\newcommand{\bdz}[1]{\bar{d}^{#1}_{0}}
\newcommand{\bbd}{\mathbf{d}}
\newcommand{\bd}[1]{\bbd^{#1}_{0}}
\newcommand{\tbd}[1]{\widetilde{\bbd}^{#1}_{0}}
\newcommand{\odel}{\overline{\delta}}
\newcommand{\bDelta}{\mathbf{\Delta}}
\newcommand{\hD}[1]{\hat{\Delta}^{#1}}
\newcommand{\tD}[1]{\tilde{\Delta}^{#1}}
\newcommand{\tDl}[2]{\tilde{\Delta}^{#1}_{#2}}
\newcommand{\tDp}[1]{\tD{#1}_{+}}
\newcommand{\hDp}[1]{\hD{#1}_{+}}
\newcommand{\bG}[1]{\overline{G}\sb{#1}}
\newcommand{\bS}[1]{\mathbf{S}^{#1}}
\newcommand{\bSp}[2]{\bS{#1}_{({#2})}}
\newcommand{\bV}[1]{\overline{V}_{#1}}
\newcommand{\bW}{\overline{W}}
\newcommand{\bmo}{\mathbf{-1}}
\newcommand{\bze}{\mathbf{0}}
\newcommand{\bo}{\mathbf{1}}
\newcommand{\bt}{\mathbf{2}}
\newcommand{\bth}{\mathbf{3}}
\newcommand{\bfo}{\mathbf{4}}
\newcommand{\bj}{\mathbf{j}}
\newcommand{\bjp}{\mathbf{j+1}}
\newcommand{\bk}{\mathbf{k}}
\newcommand{\bjkp}{\mathbf{j+k+1}}
\newcommand{\bkm}{\mathbf{k-1}}
\newcommand{\bm}{\mathbf{m}}
\newcommand{\bn}{\mathbf{n}}
\newcommand{\bnmk}{\mathbf{n-k}}
\newcommand{\bnmko}{\mathbf{n-k-1}}
\newcommand{\bnmkpo}{\mathbf{n-k+1}}
\newcommand{\bnm}{\mathbf{n-1}}
\newcommand{\bnmt}{\mathbf{n-2}}
\newcommand{\bnp}{\mathbf{n+1}}
\newcommand{\bnpp}{\mathbf{n+2}}
%
%         Fraktur
%
\newcommand{\fG}{\mathfrak{G}}
\newcommand{\cF}[1]{{\mathcal K}\sb{#1}}
\newcommand{\parz}{\partial_{0}}
\newcommand{\cfbase}[1]{\parz\cF{#1}}
\newcommand{\cftop}[1]{\widetilde{\partial}\cF{#1}}
\newcommand{\mC}[1]{\mathbf{C}\sb{#1}}
\newcommand{\baC}[1]{\overline{C}\sb{#1}}
\newcommand{\bbE}{\overline{E}'\sb{\ast}}
\newcommand{\buD}[1]{\overline{D}\sp{#1}}
\newcommand{\baD}[1]{\overline{D}\sb{#1}}
\newcommand{\baE}[1]{\overline{E}\sb{#1}}
\newcommand{\baP}[1]{\overline{P}\sb{#1}}
\newcommand{\mZ}[1]{\mathbf{Z}\sb{#1}}
\newcommand{\mZu}[1]{\mathbf{Z}\sp{#1}}
\newcommand{\latch}[1]{\mathbf{L}_{#1}}
\newcommand{\match}[1]{\mathbf{M}_{#1}}
\newcommand{\norm}[1]{\mathbf{N}_{#1}}
%
%           Titlepage
%
\title{Higher homotopy operations and Andr\'{e}-Quillen cohomology}
%
%Author information
%
\author[D.~Blanc]{David Blanc}
\address{Department of Mathematics\\ University of Haifa\\ 31905 Haifa\\ Israel}
\email{blanc@math.haifa.ac.il}
\author [M.W.~Johnson]{Mark W.~Johnson}
\address{Department of Mathematics\\ Penn State Altoona\\
                  Altoona, PA 16601\\ USA}
\email{mwj3@psu.edu}
\author[J.M.~Turner]{James M.~Turner}
\address{Department of Mathematics\\ Calvin College\\
         Grand Rapids, MI 49546\\ USA}
\email{jturner@calvin.edu}
\date{\today}
\subjclass{Primary: 55Q35; \ secondary: 18G55, 55N99}
\keywords{Higher homotopy operations, Andr\'{e}-Quillen cohomology, 
homotopy-commutative diagram, $k$-invariants, obstruction theory}

\begin{abstract}
There are two main approaches to the problem of realizing a \Pa (a
graded group $\Lambda$ equipped with an action of the primary homotopy
operations) as the homotopy groups of a space $X$. Both involve trying
to realize an algebraic free simplicial resolution \w{\Gd} of
$\Lambda$ by a simplicial space \w[,]{\Wd} and proceed by induction on
the simplicial dimension.  
The first provides a sequence of Andr\'{e}-Quillen cohomology classes
in \w{\HAQ{n+2}(\Lambda;\Omega^{n}\Lambda)} \wb{n\geq 1} as
obstructions to the existence of successive Postnikov sections for
\w{\Wd} (cf.\ \cite{DKStB}). The second gives a sequence of
geometrically defined higher homotopy operations as the obstructions
(cf.\ \cite{BlaHH}); these were identified in \cite{BJTurH} with the
obstruction theory of \cite{DKSmH}. There are also (algebraic and
geometric) obstructions for distinguishing between different
realizations of $\Lambda$. 

In this paper we
\begin{enumerate}
\renewcommand{\labelenumi}{(\alph{enumi})~}
\item provide an explicit construction of the cocycles
representing the cohomology obstructions;
\item provide a similar explicit construction of certain minimal
  values of the higher homotopy operations (which reduce to
  ``long Toda brackets''), and
\item show that these two constructions correspond under an evident map.
\end{enumerate}
\end{abstract}
\maketitle

\setcounter{section}{0}

%
%c0   Introduction
%
\section*{Introduction}
\label{cint}

Secondary and higher order operations are often used in homotopy theory
either as obstructions to resolving existence problems, or as 
computational tools.  In the 1950's, Adams
used secondary cohomology operations in \cite{AdamsN} to show the
non-existence of elements of Hopf invariant one in the stable homotopy
groups of spheres; at the same time, Toda employed his secondary
compositions (Toda brackets) to calculate some of these homotopy
groups in \cite{TodaC}. Higher homotopy and cohomology operations have
since been applied in many areas, including $H$-spaces, rational
homotopy, and stable homotopy theory (cf.\ \cite{BauA,HarpS,SpaS,TanrH}). 

To make sense of the general notion of a higher order operation, note
that many of the homotopy invariants of algebraic topology, such as
homotopy or (co)homology groups, carry a further \emph{primary}
structure, definable in the homotopy category itself. For example, the
homotopy groups of a pointed space $X$ have Whitehead products
and composition operations, which together make \w{\pis X} into a
\emph{\Pa} (see \S \ref{dssmc} below). Similarly, the mod $p$ cohomology
of $X$ has the structure of an unstable algebra over the Steenrod
algebra (cf.\ \cite[\S 1.4]{SchwU}), and the stable homotopy groups of
a commutative ring spectra form a graded commutative ring. 

The appropriate higher order operations (such as Massey products
or Toda brackets) form a higher structure superimposed on the
primary one: they are usually defined only when certain lower-order
operations vanish. General higher homotopy operations were 
defined in \cite{BMarkH,BChachP} to be certain obstructions to
rectifying homotopy-commutative diagrams \w[.]{\tX:\Gamma\to\ho\M}
Here $\M$ is a pointed simplicial model category and $\Gamma$ is a
finite directed indexing category called a \textit{lattice} (cf. \S
\ref{dlat}). In \cite{BJTurH}, we show how this obstruction
theory may be identified with that of Dwyer, Kan, and Smith 
(cf.\ \cite{DKSmH}). 

\begin{mysubsection}{Realization problems}\label{srealiz}
One natural question which arises in this context is whether a given
abstract (primary) algebraic structure \wh such as a \Pa[,] an
unstable algebra, or a graded ring \wh is in fact associated
to some topological space  or spectrum, and in how many ways. Such
realization questions have a long history in algebraic topology (see
\cite{HopfT,KuhnTR,SteCA} for the case of cohomology).  

Here we consider the problem of realizing an abstract \Pa $\Lambda$ as
the homotopy groups of a space $X$. To do so, we start with an
(algebraic) free simplicial resolution \w{\Gd} of $\Lambda$. This can
always be realized by a ``lax'' simplicial space \w[,]{\whWd} with each
\w{\whW\sb{n}} homotopy equivalent to a wedge of spheres, where the
simplicial identities hold only up to homotopy. If \w{\whWd} can be
rectified to a strict simplicial space \w[,]{\Wd} then its geometric
realization \w{X:=\|\Wd\|} has \w[,]{\pis X\cong\Lambda} as required.

There are two known approaches to solving this rectification problem
(and thus the original realization problem):
\begin{enumerate}
\renewcommand{\labelenumi}{\Roman{enumi}.~}
\item The ``geometric'' approach proceeds by induction over the
  skeleta of \w[,]{\whWd} yielding obstructions to the successive
  rectification problems in the form of higher homotopy operations 
  (see \cite{BlaHH,BlaAI}). 
\item The ``algebraic'' approach of Dwyer, Kan, and Stover 
constructs inductively two sequences of Andr\'{e}-Quillen cohomology
obstructions: one sequence for the realization of $\Lambda$, and the 
other for two such realizations $X$ and $Y$ to be homotopy equivalent.
\end{enumerate}

More explicitly, the second approach uses successive Postnikov
approximations \w{\qWd{n}} to the putative simplicial space \w{\Wd} to
define two Andr\'{e}-Quillen cohomology classes:

\begin{enumerate}
\renewcommand{\labelenumi}{(\alph{enumi})~}
\item An \emph{existence} obstruction
\w[,]{\beta_{n}\in\HAQ{n+2}(\Lambda;\Omega^{n}\Lambda)} which vanishes
if and only if \w{\qWd{n}} extends to an \wwb{n+1}Postnikov section
\w[.]{\qWd{n+1}}
\item A \emph{difference} obstruction
\w[,]{\delta_{n}\in\HAQ{n+1}(\Lambda;\Omega^{n}\Lambda)} for
distinguishing between possible extensions \w[.]{\qWd{n+1}}
\end{enumerate}
\noindent See \cite{DKStB,BDGoeR} for further details, with additional 
variants in \cite{BJTurR}.  

A different version of this theory allows one to determine 
whether a  graded commutative ring \w{R_{\ast}} is isomorphic to
\w{\pis S} for some commutative ring spectrum $S$ (see \cite{GHopM}). 
We note, however, that the main application of this theory (see \cite{Goe}) 
relies on a large scale vanishing of relevant Andr\'e-Quillen cohomology
groups, which of course guarantees vanishing of the obstructions.
The approach we describe here characterizes the obstructions directly, 
and more explicitly. We hope that this will open the door to addressing 
a broader range of realization questions using these techniques. 
For a simple example, see \S\ref{crht}.
\end{mysubsection}

\begin{mysubsection}{Main results}\label{smainr}
The aim of this paper is to make explicit the close connection between
these two approaches, by showing that the higher homotopy operations
correspond in a systematic way to the Andr\'e-Quillen obstruction classes.

For this purpose, we first study Andr\'{e}-Quillen cohomology for
general universal algebras, showing how it can be calculated using a
cochain complex, and providing an explicit description of the
$k$-invariants of a simplicial algebra (see Proposition \ref{pmoore}
and Corollary \ref{cco} below).

The $n$-th existence obstruction \w{\beta_{n}} mentioned above is in
fact the $k$-invariant for the simplicial \Pa \w[,]{\pis\qWd{n}} so
we can use this description to analyze \w[.]{\beta_{n}}
We then explain how essentially the same inductive process for
realizing $\Lambda$ (now using simply a truncated simplicial object 
\w[,]{\qVd{n+1}} rather than a Postnikov section) has another
obstruction theory in terms of higher homotopy operations, which are  
subsets \w{\llrr{\Psi^{n+2}_0}} of
\w[.]{[\bigvee_{\Psi^{n+2}_0}\,\Sigma^{n}\bV{n+2},~V_{0}]}
After defining a natural \emph{correspondence homomorphism} 
\w{\wPh{n}:[\bigvee_{\Psi^{n+2}_0}\,\Sigma^{n}\bV{n+2},~V_{0}]\to
\HAQ{n+2}(\Lambda;\Omega^{n+2}\Lambda)} in \S \ref{sch},
we construct certain natural \emph{minimal} values in
\w{\llrr{\Psi^{n+2}_0}} and prove:

%
%      Theorem: HH and AQ Obstructions correspond
%
\begin{thma}
The homomorphism \w{\wPh{n}} maps each minimal value of 
\w{\llrr{\Psi^{n+2}_0}} to the corresponding Andr\'{e}-Quillen
obstruction \w{\beta_{n}} to realizing $\Lambda$, so if the minimal
value vanishes, so does \w[.]{\beta_{n}} Conversely, if the cohomology
obstruction \w{\beta_{n+1}} associated to the next step vanishes, so does 
\w[.]{\llrr{\Psi^{n+2}_0}} 
\end{thma}
\noindent [See Theorem \ref{thhaq} and Corollary \ref{cvanish}].

Next, we also provide an explicit description of the cohomological
difference obstructions \w{\delta_{n}} for distinguishing between
inequivalent \wwb{n+1}Postnikov sections \w{\qVd{a}} and \w{\qVd{b}}
of resolutions of $\Lambda$. We use the same formalism employed in
constructing \w{\llrr{\Psi^{n+2}_0}} to define a corresponding higher
homotopy operation difference obstruction \w{\llrr{\bdz{a},\bdz{b}}} 
in \w[,]{[\Sigma^{n+1}\bV{n+2},V_{0}]} with its own minimal values, and
show:
%
%      Theorem: SO and AQ cohomology related.
%
\begin{thmb}
The correspondence homomorphism maps a minimal value of 
\w{\llrr{\bdz{a},\bdz{b}}} to the Andr\'{e}-Quillen
obstruction \w[.]{\delta_{n}}
\end{thmb}
\noindent [See Theorem \ref{tdiffobst}]\vsm.

As an application, we use this theory to show that for any connected
graded Lie algebra $\Lambda$ over $\bQ$, there is a branch of the
obstruction theory for which all cycles representing the 
Andr\'{e}-Quillen obstructions \w{\beta_{n}} to realizing $\Lambda$ vanish
(see Proposition \ref{pratpa} below). From this we recover
the well-known fact, due to Quillen, that any simply-connected
rational \Pa is realizable.
\end{mysubsection}

\begin{mysubsection}{Notation and conventions}\label{snac}
The category of topological spaces is denoted by $\TT$, and that of
pointed connected spaces by \w[.]{\Ta} 
For any category $\C$, \w{s\C:=\C^{\bDelta\op}} is the category of
simplicial objects over $\C$. We abbreviate \w{s\Set} to \w{\Ss} and
\w{s\Sets} to \w[;]{\Sa} \w{\Sr} denotes the category of
\emph{reduced} simplicial sets. The constant simplicial object on an
object \w{X\in\C} is written \w[,]{\co{X}\in s\C} and the $n$-\emph{truncation}
of \w{\Gd\in s\C} (forgetting \w{G_{i}} for \w[)]{i>n} is denoted
by \w[.]{\tau_{n}\Gd} The \emph{$n$-skeleton} functor is left adjoint
to the truncation functor \w[.]{\tau_{n}} However, we reserve the
notation \w{\sk{n}:s\C\to s\C} for the composite of the $n$-skeleton
functor with \w[.]{\tau_{n}} The \emph{$n$-coskeleton} functor
\w{\csk{n}:s\C\to s\C} is right adjoint to \w[.]{\sk{n}}

We denote by \w{\Grp} the category of groups, by \w{\Abgp} that of abelian
groups, and by \w{\Gpd} that of groupoids. When $\A$ is an abelian
category, \w{\Ch(\A)} denotes the category of (non-negatively
graded) chain complexes over $\A$. 

If \w{\lra{\V,\otimes}} is a monoidal category, \w{\VC}
is the collection of all (not necessarily small) categories enriched over $\V$
(see \cite[\S 6.2]{BorcH2}). For any set $\OO$, denote by \w{\OC} the
category of all small categories $\D$ with \w[.]{\Obj\D=\OO} A
\ww{\VO}-\emph{category} is a category \w{\D\in\OC} enriched over
$\V$, with mapping objects \w[.]{\map_{\D}(-,-)\in\V}
The category of all small \ww{\VO}-categories will be denoted by \w[.]{\VOC}
The main examples of \w{\lra{\V,\otimes}} we have in mind are
\w[,]{\lra{\Ss,\times}} \w[,]{\lra{\Sa,\wedge}} and \w[.]{\lra{\Gpd,\times}}

Note that because the Cartesian product on $\Ss$  (and the smash
product on \w[)]{\Sa} are defined levelwise, we can think of an
\ww{\SO}-\ or \ww{\SaO}-category as a simplicial object
over \ww{\OC} \wwh that is, a simplicial category with fixed object
set $\OO$ in each dimension, and all face and degeneracy functors the
identity on objects.
\end{mysubsection}

\begin{mysubsection}{Organization}
\label{sorg}
In Section \ref{cetmc} we provide some background on \PAa[s] and the
related resolution model categories used in this paper. In Section
\ref{cgac} we prove some basic facts about the cohomology of (graded)
universal algebras.  Section \ref{chaho} discusses rectification of
homotopy-commutative diagrams, and the higher homotopy operations
which appear as the obstructions to such rectification. Section
\ref{caqco} analyzes the Andr\'{e}-Quillen cohomology existence
obstructions to realizing a \PAa $\Lambda$, and Section \ref{chhoeo}
defines the higher homotopy operation version of these obstructions.
In Section \ref{cmhho} we describe certain minimal values of these
higher homotopy operations, and prove Theorem \ref{thhaq}.  Section
\ref{cdo} shows how the difference obstructions (both cohomological
and higher homotopy operation versions) may be treated analogously,
yielding  Theorem \ref{tdiffobst}. Finally, Section \ref{crht} briefly
discusses rational homotopy theory.
\end{mysubsection}

\begin{ack}
This research was supported by BSF grant 2006039. The third author was
supported by a Calvin Research Fellowship (SDG).
\end{ack}

%
%c1   Model categories
%
\sect{Model categories}
\label{cetmc}

This paper deals primarily with homotopy theory of pointed connected
topological spaces, with the usual homotopy groups. However, some of
the results hold more generally, so we now introduce axiomatic
descriptions of some more general settings.

\begin{defn}\label{dssmc}
Let $\M$ be a pointed simplicial model category (cf.\ \cite[II, \S 1]{QuiH}),
so that for every simplicial set $K$ and \w{X\in\M} we
have a object \w{K\hotimes X} in $\M$. In particular, we call 
\w{\bS{1}\hotimes X} the \emph{half-suspension} of $X$, and by choosing
a basepoint \w{\pt} in \w[,]{S^{1}} we define the  
$n$-\emph{fold suspension} \w{\Sigma^{n} X} to be 
\w[.]{(\bS{n}\hotimes X)/(\{\pt\}\hotimes X)} Thus for 
\w[,]{X\in\Ta} we have 
\w[,]{K\hotimes X=|K|\ltimes X:=(|K|\times X)/(|K|\times\{\ast\})} 
so \w[.]{\Sigma^{n}X:=|\bS{n}|\wedge X}

Now let $\A$ be a collection of homotopy cogroup objects in $\M$
(sometimes called \emph{spherical objects}). For any object
\w[,]{Y\in\M} its $\A$-\emph{homotopy groups}  
are \w[,]{\piA Y=(\piul{A}{n}Y)_{A\in\A,n\in\bN}} where
\w[.]{\piul{A}{n}Y:=[\Sigma^{n}A,Y]_{\M}}
A map \w{f:Y\to Z} in $\M$ is called an \emph{$A$-equivalence} if it
induces an isomorphism in \w[.]{\piA}
We denote by \w{\PiA} the full subcategory of \w{\M} whose objects are
finite coproducts of suspensions of elements of $\A$. A
product-preserving functor \w{\Lambda:\PiA\op\to\Sets} is
called a \emph{\PAa[,]} and the category of such is denoted by
\w[.]{\PAAlg} When we wish to emphasize the dependence on $\M$, we
call these \emph{\PAMas[,]} and denote the category by \w[.]{\PAMAlg}
We write \w{\Lambda\lin{B}} for the value of $\Lambda$
at an object \w[.]{B\in\PiA} We denote by \w{\MA} the
smallest full subcategory of $\M$ containing $\A$ and closed under
suspensions, arbitrary coproducts, and weak equivalences.

When \w{\M=\Sr} (or \w[)]{\Ta} and \w{\A:=\{S^{1}\}}
these are called simply \Pa[s,] and the category is denoted by
\w{\PAlg} (cf.\ \cite{StoV}).
\end{defn}

\begin{example}\label{egpa}
The canonical example of a \PAa is a \emph{realizable} one,
denoted by \w[,]{\piA X} and defined for fixed \w{X \in\M} by
\w{A\mapsto[A,X]_{\ho\M}} for all \w[.]{A\in\PiA} This defines a
functor \w[.]{\piA:\ho\M\to\PAAlg}  Thus when
\w[,]{\A=\{S^{1}\}} so \w[,]{\PAAlg=\PAlg} a realizable
\Pa consists of the sequence of groups \w[,]{\pis X} equipped with
the action of the primary homotopy operations on them (compositions,
Whitehead products, and action of the fundamental group). This is
called the \emph{homotopy \Pa[]} of $X$.
\end{example}

\begin{defn}\label{dgalg}
Let $\Theta$ be an \emph{FP-sketch}, in the sense of
Ehresmann (cf.\ \cite{EhreET}) \wh that is, a small category with a
distinguished collection $\PP$ of products.  A \emph{\Tal} is a functor
\w{\Lambda:\Theta\to\Set} which preserves the products in $\PP$. We think 
of a map \w{\phi:\prod_{i=1}^{n}\,a_{i}\to\prod_{j=1}^{m}\,b_{j}} 
in $\Theta$ as representing an $m$-valued $n$-ary operation on \Tal[s,]
with gradings indexed by \w{(a_{i})_{i=1}^{n}} and
\w[,]{(b_{j})_{j=1}^{m}} respectively.

The category of \Tal[s] is denoted by \w[.]{\TAlg}
If \w{\Obj(\Theta)} is generated under the products in $\PP$ by a set
$\OO$, there is a forgetful functor \w{U:\TAlg\to\Set^{\OO}} into
the category of $\OO$-graded sets, with left adjoint the \emph{free \Tal}
functor \w[.]{F:\Set^{\OO}\to\TAlg}
\end{defn}

\begin{example}\label{eggalg}
The main type of \Tal[s] considered in this paper are those with 
\w[,]{\Theta=\PiA\op} \w[,]{\OO=\A} and $\PP$ the set of finite 
coproducts of objects in $\A$. Note that these are products in \w[.]{\PiA\op} 

In particular, let \w{\Pi_{1}} denote the homotopy category of finite
wedges of circles: that is, the full subcategory of \w{\ho\Ta} with object
set \w[.]{\{\bigvee_{i=1}^{n}\,S\sp{1}\}_{n=0}^{\infty}} Then
\w[,]{\fG:=\Pi_{1}\op} with \w[,]{\PP=\Obj(\fG)} is the theory
representing groups \wh that is, \w{\Alg{\fG}} is naturally equivalent
to \w[.]{\Grp} Similarly, \w{\fG^{\bN}} represents $\bN$-graded groups.

We define a $\fG$-\emph{theory} to be an FP-sketch $\Theta$
equipped with an embedding of sketches \w[,]{\fG^{\OO}\hra\Theta} for
$\OO$ as in \S \ref{dgalg}. In this case, any \Tal $X$ has a
natural underlying $\OO$-graded group structure. We do not
require the operations of a $\fG$-theory to be homomorphisms (that
is, commute with the $\fG$-structure).
\end{example}

\begin{defn}\label{dmod}
As in \cite{BecT} or \cite[\S 1]{QuiC}, we define a \emph{module} over a
\Tal $\Lambda$ to be an abelian group object \w{p:\K\to\Lambda} in 
\w[.]{\TAlg/\Lambda} Note that 
\w{0\in\Hom_{\TAlg/\Lambda}(\Lambda,M)} then provides a section for the
structure map $p$. If for each \w{v\in\OO} we let
\w[,]{\baK_{v}:=\{x\in\K(v)~: \ p(x)=0\}} then \w{\baK_{v}} has an abelian group
structure induced from that of \w[,]{\K\to\Lambda} and for each \w{f:u\to v}
in \w{\Theta} we have a homomorphism of abelian groups
\w[.]{f^{\ast}:\baK_{v}\to\baK_{u}} Thus $\K$ is completely determined by
the \emph{restricted module} functor \w[.]{\baK:\Theta\to\Abgp} We denote
the category of restricted modules by \w[,]{\RL} with \w{\K\mapsto\baK}
defining an equivalence of categories \w[.]{(\TAlg/\Lambda)\ab\to\RL} 
\end{defn}

\begin{remark}\label{rga}
The identification of the half-suspension
\w{\bS{n}\hotimes A} with \w{\Sigma^{n}A\vee A} for a
homotopy cogroup object $A$ (cf.\ \cite{BJiblSL}) makes
\w{[\bS{n}\hotimes A,Y]} into a module over \w{\hpi Y} in the sense of
\S \ref{dmod}. This allows us to think of
\w{\piAn{n}Y:=[\Sigma^{n}A,Y]} itself as a restricted
\ww{\piAn{0}Y}-module.

Moreover, for any \PAa $\Lambda$, we may define an abelian \PAa\
\w{\Omega\Lambda} by setting
\w[.]{(\Omega\Lambda)\lin{A}:=\Lambda\lin{\Sigma A}} This has a
natural structure of a restricted $\Lambda$-module (see
\cite[\S 9.4]{DKStB}, and compare \cite[\S 1.11]{BBlaC}).
\end{remark}

\begin{mysubsection}{Resolution model categories}
\label{srmc}
Let $\M$ be a pointed, cofibrantly generated, right proper simplicial
model category, equipped with a collection $\A$ of homotopy cogroup objects.
The category \w{s\M} of simplicial objects over $\M$ has a
\emph{resolution model category} structure, 
in which a map \w{f:\Wd\to\Vd} in \w{s\M} is a weak equivalence
if and only if the induced map of simplicial groups
\w{f_{\#}:\piul{A}{n}\Wd\to\piul{A}{n}\Vd}
is a weak equivalence for each \w{A\in\A} and \w[.]{n\geq 0}
See \cite{BouC} and \cite{JardB} for further details.

Moreover, \w{s\M} has its own simplicial structure \w{(s\M,\otimes)}
(cf.\ \cite[II, \S 1]{QuiH}), and thus has a set of spherical objects:
\w{(\bS{n}\otimes\co{\Sigma^{i}A})/(\{\pt\}\otimes\co{\Sigma^{i}A})} 
for \w[.]{A\in\A; i,n\in\bN} 
The \emph{natural} homotopy groups of a simplicial object \w{\Xd\in s\M}
are defined by 
\w{\pinat{n,i,A}\Xd:=[\bS{n}\otimes\Sigma^{i}A,\Xd]_{s\M}} for
\w[.]{A\in\A} Setting \w[,]{\pinat{n}\Xd:=\{\pinat{n,i,A}\Xd\}_{A\in\A,i\in\bN}}
it may be shown that \w{\pinat{n}\Xd} has a natural \PAa structure (see
\cite[\S 5]{DKStB}).

To describe some basic constructions in \w[,]{s\M} recall that the
$n$-th \emph{Moore chains} object of a Reedy fibrant simplicial object
\w{\Xd} is defined:
\begin{myeq}\label{eqmoor}
\mC{n}\Xd~:=~\cap_{i=1}^{n}\Ker\{d_{i}:X_{n}\to X_{n-1}\}~,
\end{myeq}
\noindent with differential \w[.]{\partial_{n}^{\Xd}=\partial_{n}:=
(d_{0})\rest{\mC{n}\Xd}:\mC{n}\Xd\to\mC{n-1}\Xd}
The $n$-th \emph{Moore cycles} object is 
\w[.]{\mZ{n}\Xd:=\Ker(\partial_{n}^{\Xd})}

It turns out that under mild assumptions on $\M$
the inclusion \w{\iota:\mC{n+1}\Xd\hra X_{n+1}} induces an isomorphism
\begin{myeq}\label{eqcommmoor}
\iota_{\star}~:~\piA\mC{n+1}\Xd~\to~\mC{n+1}\piA \Xd
\end{myeq}
\noindent (see \cite[Lemma 2.7]{StoV} or \cite[Prop.\ 2.7]{BlaCW}),
which fits into a commuting diagram of \PAa[s] with exact rows:
\mydiagram[\label{eqhurewicz}]{
\piA \mC{n+1}\Xd \ar[rr]^{(\partial_{n+1}^{\Xd})_{\#}}
\ar[d]_{\iota_{\star}}^{\cong} &&
\piA \mZ{n}X \ar@{->>}[r]^{\hat{\vartheta}_{n}} \ar[d]^{\hat{\iota}_{\star}} &
\pinat{n}\Xd \ar@{.>}[d]^{h_{n}} \\
\mC{n+1}(\pis\Xd) \ar[rr]^{\partial_{n+1}^{\piA\Xd}} &&
\mZ{n}(\piA\Xd) \ar@{->>}[r]^{\vartheta_{n}} & \pi_{n}(\piA\Xd)~.
}
\noindent This defines the \emph{Hurewicz map}
\w{h_{n}:\pinat{n}\Xd\to \pi_{n}\piA\Xd} in the
\emph{spiral long exact sequence} of $\A$-graded groups:
\begin{myeq}[\label{eqspiral}]
\begin{split}
\dotsc~\to~  \Omega\pinat{n-1}\Xd & ~\xra{s_{n}}~
\pinat{n}\Xd~\xra{h_{n}}~\pi_{n}\piA\Xd~\xra{\partial_{n}} \\
\to~& \Omega\pinat{n-2}\Xd~\to~\dotsc~\to ~\pinat{1}\Xd~\to~
\pi_{1}\piA\Xd
\end{split}
\end{myeq}
\noindent (cf.\ \cite[8.1]{DKStB}), where \w{s_{n}} is induced by the
connecting homomorphism in \w{\piA} for the fibration sequence
in $\M$: 
\begin{myeq}[\label{eqsn}]
\mZ{n}\Xd~\xra{j_{n}}~\mC{n}\Xd~\xra{d_{0}}~\mZ{n-1}\Xd~.
\end{myeq}
\end{mysubsection}

\begin{example}\label{egrmc}
When \w{\M=\Ta} and \w[,]{\A=\{S^{1}\}} the
resolution model category of simplicial spaces is the original
\ww{E^{2}}-model category of \cite{DKStE}.
\end{example}

\begin{remark}\label{rrmc}
If $\Theta$ is a $\fG$-theory (\S \ref{eggalg}), the monogenic free
\Tal[s] constitute a collection $\A$ of (strict) cogroup objects in
\w{\M=\TAlg} (with the trivial model category structure). Since maps
between free \Tal[s] represent the operations in $\Theta$, 
in this case a \PAMa\ may be identified with a \Tal[,] so there is a
canonical equivalence of categories \w[.]{\PAMAlg\approx\M}
This applies in particular to \w{\Theta=\PiA} itself, so that 
\w[.]{\Alg{\PiA\sp{\PAAlg}}\approx\PAAlg}

In this case, the resolution model category structure on
\w{s\TAlg} is Quillen's  model category for simplicial universal
algebras (cf.\  \cite[II, \S 4]{QuiH}), and \w{\pinat{n}\Gd} is 
the graded group \w[,]{\pi_{n}\Gd} equipped with a natural \Tal
structure. Any \w{\Gd\in s\TAlg} for which each \w{G_{n}} is free, and
the degeneracy maps take generators to generators, is cofibrant.
\end{remark}

\begin{mysubsection}{$E^{2}$-model categories}\label{setmc}
If $\M$ is a pointed model category as in \S \ref{srmc} with a
collection of spherical objects $\A$ (\S \ref{dssmc}), the resolution model
category \w{s\M} is called an \ww{E^{2}}-\emph{model category} if
it is equipped with:  

\begin{enumerate}
\renewcommand{\labelenumi}{(\alph{enumi})\ }
\item A functorial \emph{Postnikov tower} of fibrations (in \w[)]{s\M}
  for each \w[:]{\Wd\in s\M}
$$
\dotsc \to \Po{n}\Wd\xra{p\q{n}}\Po{n-1}\Wd\xra{p\q{n-1}}\dots\to\Po{0}\Wd~,
$$
\noindent equipped with a weak equivalence
\w[,]{r:Z\to\Po{\infty}\Wd:=\lim_{n}\Po{n}\Wd} as well as fibrations
\w[,]{r\q{n}:\Po{\infty}\Wd\to\Po{n}\Wd} such that \w{p\q{n}} and
\w{r\q{n}} induce ismorphisms in \w{\pinat{i}} for \w[,]{i\leq n}
and \w{\pinat{i}\Po{n}\Wd=0} for \w[.]{i>n}
\item For every \PAMa\ $\Lambda$, there is a functorial
  \emph{classifying object} \w[,]{B\Lambda\in s\M} unique up to homotopy, with
  \w{B\Lambda\simeq\Po{0}B\Lambda} and \w[.]{\pinat{0}B\Lambda\cong\Lambda}
\item Given a \PAMa\ $\Lambda$ and a $\Lambda$-module $K$, for each \w{n\geq 1}
  there is a functorial fibrant \emph{Eilenberg-Mac~Lane object}
  \w{E=\tEL{K}{n}} in \w[,]{s\M/B\Lambda} unique up to homotopy,
  equipped with a section $s$ for
  \w[,]{(r\q{0}\circ r):E\to\Po{0}E\simeq B\Lambda} such that
  \w{\pinat{n}E\cong K} as a $\Lambda$-module, and \w{\pinat{i}E=0} for
  \w[.]{1\leq i\neq n}
\item For every \w[,]{n\geq 0} there is a functor that assigns to each
  \w{\Wd\in s\M} with \w{\pinat{0}\Wd=\Lambda} a homotopy pull-back square:
\mydiagram[\label{eqkinv}]{
\ar @{} [dr] |<<<{\framebox{\scriptsize{PB}}}
\Po{n+1}\Wd \ar[r]^{p\q{n+1}} \ar[d] &
\Po{n}\Wd \ar[d]^{k_{n}}\\ B\Lambda \ar[r] & \tEL{\pinat{n+1}\Wd}{n+2}
}
\noindent (in \w[)]{s\M} with \w{k_{n}} the $n$-th
$k$-\emph{invariant} for \w[.]{\Wd} 
\item A \emph{realization} functor \w[,]{J:s\M\to\M} such that, for
\w{\Lambda\in\PAMAlg} and cofibrant \w[,]{\Xd\in s\M} if
\w{\piA\Xd\stackrel{\sim}{\to}\tBL} is a weak
equivalence in \w[,]{s\PAMAlg} then there is an isomorphism:
\begin{myeq}[\label{eqrealwe}]
[A,J\Xd]_{\M} \stackrel{\cong}{\rightarrow}\Hom_{\PAMAlg}(\piA A,\Lambda)~,
\end{myeq}
\noindent natural in $\Lambda$ and \w[.]{A\in\A}
\end{enumerate}
\end{mysubsection}

\begin{remark}\label{rsmc}
In all the cases we are interested in, the coskeleton \w{\csk{n+1}\Wd}
(\S \ref{snac}) provides the functorial Postnikov section
\w{\Po{n}\Wd} for Reedy fibrant simplicial objects \w{\Wd\in s\M}
(cf.\ \cite[\S 15]{PHirM}). 
\end{remark}

\begin{examples}\label{egetmc}
The two main $E^{2}$-model categories we have in mind are:
\begin{enumerate}
\renewcommand{\labelenumi}{(\arabic{enumi})}
\item \w{\M=\Ta} or \w[,]{\Sa} with \w[.]{\A=\{S^{1}\}} In this case
$J$ is the usual realization functor, and \wref{eqrealwe} follows from the
collapse of the Bousfield-Friedlander spectral sequence (cf.\
\cite[Theorem B.5]{BFrieH}).  
\item \w[,]{\M=\TAlg} the category of \Tal[s] for some $\fG$-theory
$\Theta$, as in \S \ref{rrmc}, and $\A$ the monogenic free \Tal[s.] 
Here \w{J\Xd:=\pi_{0}\Xd} (so \w{\piA\Xd\stackrel{\sim}{\to}\tBL} is a
weak equivalence in \w{s\PAMAlg\approx s\M} as above if and only if 
\w{\var:\Xd\to\Lambda} is a resolution), and the
isomorphism \wref{eqrealwe} is induced by $\var$.
\end{enumerate}
For further examples, see \cite[\S 3]{BJTurR}. 
\end{examples}

\begin{defn}\label{dcw}
Let $\M$ be a pointed model category. A simplicial object \w{\Gd\in s\M} is
called a \emph{CW object} if:

\begin{enumerate}
\renewcommand{\labelenumi}{(\alph{enumi})}
\item for each \w{n\geq 0} there is an object \w{\bG{n}\in\M} such that
\w[.]{G_{n}=\bG{n}\amalg \latch{n}\Gd} Here
\begin{myeq}\label{eqlatch}
\latch{n}\Gd~:=~
\coprod_{0\leq k\leq n}~\coprod_{0\leq i_{1}<\dotsc<i_{n-k-1}\leq n-1}~\bG{k}
\end{myeq}
\noindent is the $n$-th \emph{latching object} of \w[,]{\Gd} in which
the copy of \w{\bG{k}} indexed by \w{(i_{1},\dotsc,i_{n-k-1})} is in
the image of \w[.]{s_{i_{n-k-1}}\dotsc s_{i_{2}}s_{i_{1}}}
\item There is an \emph{attaching map} \w{\bdz{G_{n}}:\bG{n}\to G_{n-1}}
with \w{d_{i}\circ\bdz{G_{n}}=0} for \w[,]{0\leq i\leq n-1} or equivalently,
\w{\bdz{G_{n}}} factors through \w[.]{\mZ{n-1}\Gd \subset G_{n-1}}
\item The face maps of \w{\Gd} are determined by the simplicial
  identities and the requirement that
  \w{(d_{0})\rest{\bG{n}}=\bdz{G_{n}}} and
  \w{(d_{i})\rest{\bG{n}}=0} for \w[.]{1\leq i\leq n}
\end{enumerate}
The collection \w{(\bG{n})_{n=0}^{\infty}} is called a \emph{CW basis} for
\w[.]{\Gd}

When $\A$ is a collection of homotopy cogroup objects in $\M$,
\w{\Gd\to X} is a cofibrant replacement in the resolution model
category structure on \w{s\M} determined by $\A$, and each \w{\bG{n}}
in a CW basis for \w{\Gd} lies in \w{\MA} (\S \ref{dssmc}), we
call \w{\Gd} a \emph{CW resolution} of $X$.
\end{defn}

\begin{remark}\label{stepred}
The category of \wwb{n+2}truncated CW objects \w{\qVd{n+2}}in
$\M$ is equivalent to the category of pairs consisting of an
\wwb{n+1}truncated CW object \w{\qVd{n+1}} and a map 
\w[:]{\bdz{\bV{n+2}}:\bV{n+2}\to \mZ{n+1}\qVd{n+1}}
given such a pair \w[,]{(\Vd,\bdz{\bV{n+2}})}
we obtain a \wwb{n+2}truncated CW object by setting
\w[,]{V_{n+2}:=\bV{n+2}\coprod \latch{n+2}\Vd} with 
\w{(d_{0})\rest{\bV{n+2}}:=\bdz{\bV{n+2}}} and
\w{(d_{i})\rest{\bV{n+2}}=0} for \w[.]{i>0}  The degeneracies are
given by the obvious inclusions into \w[,]{\latch{n+2}\Vd} and the
face maps on \w{\latch{n+2}\Vd} are determined by the simplicial identities.

When \w[,]{\M:=\TAlg} one can use this method inductively to
construct a free CW-resolution of a \Tal $\Lambda$ (with each 
\w{\bV{n+2}} free).
\end{remark}

%
%c2    Cohomology of $\Theta$-algebras
%
\sect{Cohomology of \Tal[s]}
\label{cgac}

In this section we recall the definition of Andr\'{e}-Quillen
cohomology for simplicial \Tal[s,] provide a cochain description for their
cohomology, and give an explicit construction of their
$k$-invariants. Although most of the results are valid more generally,
for simplicity we restrict attention to the case where the \Tal[s] have
an underlying group structure. 

\begin{defn}\label{dcoh}
Let $\A$ be a collection of spherical objects in a pointed model
category $\M$, such that the resolution model category \w{s\M}
is an \ww{E^{2}}-model category (\S \ref{setmc}). Assume given a \PAMa\
$\Lambda$, a $\Lambda$-module $\K$, 
and  an object \w{\Wd\in s\M} equipped with a \emph{twisting map} 
\w[.]{t:\hpi\Wd\to\Lambda} We use $t$, along with the natural map
\w[,]{\Wd\to B\hpi\Wd} to think of \w{\Wd} as an object in
\w[.]{s\M/\tBL} Following \cite{AndrM,QuiC}, we define the $n$-th
\emph{cohomology group of \w{\Wd} with coefficients in $\K$} to be   
$$
\HL{n}(\Wd;\K)~:=~[\Wd,\tEL{\K}{n}]_{s\M/\tBL}~=~
\pi_{0}\map_{s\M/\tBL}(\Wd,\tEL{\K}{n})~,
$$
\noindent where the last mapping space is defined by the (homotopy) pullback:
$$
\xymatrix@R=25pt{
\map_{s\M/\tBL}(\Wd,\tEL{\K}{n}) \ar[d] \ar[r] &
\map_{s\M}(\Wd,\tEL{\K}{n}) \ar[d]^{\left(p\q{0}_{\tEL{\K}{n}}\right)_{\ast}}\\
\{Bt\circ p\q{0}_{\Wd}\}~  \ar@{>->}[r] & \map_{s\M}(\Wd,\tBL)~.
}
$$
\end{defn}

Typically, we have \w[,]{\Lambda=\hpi\Wd} with $t$ an isomorphism; if
in addition \w[,]{\Wd\simeq\tBL} we denote \w{\HL{n}(\Wd;\K)} simply
by \w[.]{\HAQ{n}(\Lambda;\K)} 

\begin{defn}\label{drcoh}
There is also a relative version, for a pair \w{(\Wd,\Yd)} \wh that
is, a cofibration \w{i:\Yd\hra\Wd} in \w{s\M} \wh with $\K$ a
$\Lambda$ module and \w{t:\hpi\Yd\to \Lambda} a twisting map as before.
Let \w{\PO(\Wd,\Yd)} denote the (homotopy) pushout in \w{s\M} of:
$$
\xymatrix@R=25pt{
\Yd \ar[d]_{r} \ar[r]^{i} &\Wd \ar[d] \\
B\hpi\Yd \ar[r] & \PO(\Wd,\Yd)~.
}
$$
\noindent We define \w{\HL{n}(\Wd,\Yd;\K)} to be the group of homotopy
classes of maps
\w{f:\PO(\Wd,\,\Yd)\,\to\,(\tEL{\K}{n},\,\tBL)} 
in \w{s\M/\tBL} fitting into the commutative diagram: 
$$
\xymatrix@R=25pt{
\Yd \ar[d]_{r} \ar[r]^{i} &\Wd \ar[d]^{j} \\
B\hpi\Yd \ar[r] \ar@/_2pc/[dr]_{Bt} & \PO(\Wd,\Yd)\ar[rd]^{f} & \\
& \tBL \ar[r]^{s} & \tEL{\K}{n}
}
$$
\noindent (see \cite[\S 2.1]{DKSmO}).
\end{defn}

\begin{remark}\label{rabel}
Let $\Theta$ be a $\fG$-theory and \w[,]{\M=\TAlg} so
\w{\PAMAlg\approx\M} (cf.\ \S \ref{rrmc}). The constant object
\w{\co{\Lambda}} is a fibrant model for \w[,]{\tBL} so
\w[.]{s\M/\tBL\cong s(\M/\Lambda)}  
In particular, this implies that \w{s((\M/\Lambda)\ab)} may be
identified with the category \w{(s\M/\tBL)\ab} of abelian group
objects in \w[,]{s\M/\tBL} with abelianization functor
\w{\AbL:(s\M/\tBL)\to (s\M/\tBL)\ab} defined dimensionwise
(cf.\ \cite[\S 3.20]{BlaQ}).
\end{remark}

%
%      Proposition:  cohomology calculated via Moore chains
%
\begin{prop}\label{pmoore}
Let $\Theta$ be a $\fG$-theory and \w[.]{\M=\TAlg} For any \w[,]{\Lambda\in\M} 
$\Lambda$-module $\K$, cofibrant \w{\Wd\in s\M}  and twisting map $t$,
we have maps
$$
\Hom_{s\M/\tBL}(\Wd,\tEL{\K}{n})~\xla{\zeta}~ 
\Hom_{\Ch(\RL)}(\baC{\ast}\Wd,\bbE)~\xra{\eta}~
\Hom_{\Lambda}(\mC{*}\AbL\Wd,\K)~,
$$
\noindent  which are natural in \w[,]{\Wd} and induce 
an isomorphism
$$
\HL{n}(\Wd;\K)\cong H\sp{n}\Hom_{\Lambda}(\mC{\ast}\AbL\Wd,\K)~
$$
\noindent for each \w[.]{n\geq 0}
\end{prop}

\begin{proof}
\textbf{Step I.} \ First, we show how \w{\HL{n}(\Wd;\K)} may be
described in terms of a mapping space of chain complexes:

Since \w{E:=\tEL{\K}{n}} can be chosen to be a \emph{strict} abelian
group object in \w{s\M/\tBL} (see \cite[\S 3.14]{BJTurR}), we have a natural
identification:  
$$
\map_{s\M/\tBL}(\Wd,\,E)~\cong~\map_{(s\M/\tBL)\ab}(\AbL\Wd,\,E)~.
$$
\noindent Recall that the Moore chain functor
\w{\mC{\ast}:s(\M/\Lambda)\ab\to\Ch((\M/\Lambda)\ab)} 
induces the Dold-Kan equivalence between simplicial objects and chain
complexes over an abelian category (cf.\ \cite[\S 1]{DolH}).
Composing with the equivalence
\w{(\M/\Lambda)\ab\xra{\approx}\RM{\Lambda}} of
\S \ref{dmod} defines \w[.]{\baC{\ast}:s(\M/\Lambda)\ab\to\Ch(\RL)} 

Applying \w{\baC{\ast}}  to the simplicial $\Lambda$-module
\w{\AbL\Wd} yields \w[,]{\baD{\ast}:=\baC{\ast}\AbL\Wd} with
\begin{myeq}\label{eqchains}
\map_{s\M/\tBL}(\Wd,\,E)~\cong~\map_{\Ch(\RL)}(\baD{\ast},\,\baC{\ast}E)~,
\end{myeq}
\noindent so applying \w{\pi_{0}} yields the required cohomology
group\vsm.

\noindent \textbf{Step II.} \ We now translate this into chain function complexes:

Both sides of \wref{eqchains} are simplicial abelian groups, so we can
replace the right hand side under the Dold-Kan equivalence with the usual mapping
chain complex \w[,]{F_{\ast}:=\uHom(\baD{\ast},\,\baC{\ast}E)} 
where
$$
F_{0}~:=~\Hom_{\Ch(\RL)}(\baD{\ast},\baC{\ast}E)
\hspace*{5mm}\text{and}\hspace*{5mm} F_{1}:=\Hom_{\Ch(\RL)}(\baD{\ast},P\baC{\ast}E)
$$
\noindent for \w{P\baC{\ast}E} the path object on \w[.]{\baC{\ast}E}
Moreover, the differential \w{\partial_{1}:\uHom_{1}\to\uHom_{0}} is
induced by the path fibration  \w[.]{p:P\baC{\ast}E\to\baC{\ast}E}

Since \wref{eqchains} induces an identification:
\w[,]{\pi_{0}\map_{s\M/\tBL}(\Wd,\,E)\cong H_{0}\uHom(\baD{\ast},\baC{\ast}E)}
we have a right-exact sequence:
\begin{myeq}\label{eqsesh}
\Hom_{\Ch(\RL)}(\baD{\ast},P\baC{\ast}E)~\xra{p_{\ast}}~
	\Hom_{\Ch(\RL)}(\baD{\ast},\baC{\ast}E)~
\epic~\HL{n}(\Wd;\K)~\to~0~.
\end{myeq}
\noindent Here \w{p:P\baC{\ast}E\to\baC{\ast}E} has
the obvious minimal model:
\begin{myeq}\label{eqobvious}
\begin{split}
& \xymatrix@R=25pt{
\baP{\ast} \ar[d]^{q}  & = & \dotsc 0 \ar[r] & \baK \ar[r]^{=} \ar[d]^{=} &
\baK \ar[r] \ar[d] & 0\dotsc\\
\baE{\ast}  & =  & \dotsc 0 \ar[r] & \baK \ar[r] & 0 \ar[r] & 0\dotsc}\\ 
& \dim \hspace*{26mm} n+1\hspace*{6mm} n
\hspace*{9mm} n-1\hspace*{4mm} n-2
\end{split}
\end{myeq}
\noindent since 
\w[,]{H_{i}\baC{\ast}E\cong\overline{H_{i}\mC{\ast}E}\cong\overline{\pi_{i}E}}
which is $\baK$ for \w{i=n} and $0$ otherwise (note that 
\w[!).]{\overline{\Lambda}=0} 

This means that we can choose a cofibrant model \w{\bbE} for
\w{\baC{\ast}E} fitting into a span \wh that is, a commuting diagram:
\begin{myeq}\label{eqspan}
\xymatrix@R=25pt{
P\baC{\ast}E \ar[d]^{p}  &&  P\bbE \ar[d]^{p'} \ar[ll]_{P\zeta}\ar[rr]^{P\eta} 
&& \baP{\ast} \ar[d]^{q}\\
\baC{\ast}E  && \bbE \ar[ll]_{\zeta}\ar[rr]^{\eta} && \baE{\ast}~.
}
\end{myeq}
\noindent The horizontal weak equivalences $\zeta$, $\eta$,
\w[,]{P\zeta} and  \w{P\eta} induce a span of quasi-isomorphisms from
\wref{eqsesh} to 
\begin{myeq}\label{eqminsesh}
\Hom_{\Ch(\RL)}(\baD{\ast},\baP{\ast})~\xra{q_{\ast}}~
\Hom_{\Ch(\RL)}(\baD{\ast},\baE{\ast})~\epic~\HL{n}(\Wd;\K)~\to~0~,
\end{myeq}
\noindent which are both the identity on \w[\vsm.]{\HL{n}(\Wd;\K)}

\noindent \textbf{Step III.} \ Finally, we produce a  
commuting diagram:
\mydiagram[\label{eqcochains}]{
\Hom_{\Ch(\RL)}(\baD{\ast},\baP{\ast}) \ar[d]_{\alpha} \ar[r]^{q_{\ast}} &
\Hom_{\Ch(\RL)}(\baD{\ast},\baE{\ast}) \ar[d]_{\beta} \\ 
\buD{n-1} \ar[r]^{\delta^{n-1}=(\partial_{n})^{\ast}} & \mZu{n}\buD{\ast}
}
\noindent where the dual cochain complex \w{\buD{\ast}} is defined by
applying \w{\Hom_{\Ch(\RL)}(-,\baK)} dimensionwise to
\w[,]{\baD{\ast}} and \w{\mZu{n}\buD{\ast}} are its $n$-cocycles:

To describe $\alpha$, note that a chain map $f$ in
\w{\Hom_{\Ch(\RL)}(\baD{\ast},\baP{\ast})} is given by a commuting diagram:
\begin{equation*}
\xymatrix@R=25pt{
\dotsc \baD{n+1} \ar[d] \ar[r]^{\partial_{n+1}} &
\baD{n} \ar[d]^{\psi\circ\partial_{n}}
\ar[r]^{\partial_{n}} &  \baD{n-1} \ar[d]^{\psi} \ar[r]^{\partial_{n-1}} &\dotsc \\
\dotsc 0 \ar[r] & \baK \ar[r]^{=} & \baK\ar[r] &0 \dotsc
}
\end{equation*}
\noindent so that \w[.]{\alpha(f):=\psi} 

Similarly, a chain map $g$ in
\w{\Hom_{\Ch(\RL)}(\baD{\ast},\baE{\ast})} 
is given by a commuting diagram:  
\begin{equation*}
\xymatrix@R=25pt{
\dotsc \baD{n+1} \ar[d] \ar[r]^{\partial_{n+1}} &
\baD{n} \ar[d]^{\phi} \ar[r]^{\partial_{n}} &
\baD{n-1} \ar[d] \ar[r]^{\partial_{n-1}} &\dotsc \\
\dotsc 0 \ar[r] & \baK \ar[r] & 0\ar[r] &\dotsc
}
\end{equation*}
\noindent with \w[;]{\beta(g):=\phi} indeed,
\w{\beta(g)\in\mZu{n}\buD{\ast}} since \w[.]{\phi\circ\partial_{n+1}=0} 

Since \wref{eqcochains} clearly commutes, and the cokernel of the
bottom map is \w[,]{\HL{n}\buD{\ast}} we conclude from \wref{eqminsesh} that: 
$$
\HL{n}(\Wd;\K)~=~H^{n}(\Hom_{\RL}(\baC{\ast}\AbL\Wd,\baK))~\cong~
H^{n}(\Hom_{\M/\Lambda}(\mC{\ast}\AbL\Wd,\K))~,
$$
\noindent again using the equivalence of categories 
\w{\RL\approx(\M/\Lambda)\ab} of \S \ref{dmod}, and the fact that $\K$
is an abelian group object in \w[.]{\M/\Lambda}
\end{proof}

\begin{lemma}\label{lcocw}
Under the assumptions of Proposition \ref{pmoore}, if \w{\Wd} has a
CW-basis \w{(\bW_{n})_{n=0}^{\infty}} (\S \ref{dcw}),
then for each \w{n>0} the natural map \w{\bW_{n}\to\mC{n}\AbL\Wd} induces an
isomorphism 
$$
\Hom_{\Lambda} (\mC{n}\AbL\Wd,\K)~\to~\Hom_{\M}(\bW_{n},\baK)~.
$$
\end{lemma}

\begin{proof}
Let \w{s_{I^{(n)}}:W_{0}\to W_{n}} be the unique iterated degeneracy
map, making \w{W_{0}} a coproduct summand in 
\w[,]{\latch{n}\Wd\cong\latch{n}'\Wd\amalg W_{0}} where by \wref[,]{eqlatch}
\w{\latch{n}'\Wd} is a coproduct of the images under various iterated
degeneracies of (copies of) the basis objects \w[.]{(\bW_{i})_{i=1}^{n-1}}

 Since the structure map \w{W_{n}\to\Lambda}
for \w{\Wd\to\Lambda} factors though the retract \w{W_{n}\to W_{0}}
for \w[,]{s_{I^{(n)}}} in fact \w{W_{n}\to\Lambda} is a coproduct in
\w{\M/\Lambda} of \w{0:\bW_{n}\amalg\latch{n}'\Wd\to\Lambda} and
\w[,]{\var:W_{0}\to\Lambda} where the first summand further splits as
a coporoduct of objects of the form \w[.]{0:\bW_{i}\to\Lambda}

Because the abelianization functor
\w{\AbL:\M/\Lambda\to(\M/\Lambda)\ab} is a left adjoint, it commutes
with coproducts, and when applied to \w{0:A\to\Lambda} yields 
\w[,]{0:\Ab A\to\Lambda} where \w{\Ab:\M\to \M\ab} is the usual
abelianization of \Tal[s.] 
\begin{equation*}
\begin{split}
\AbL W_{n}=&\AbL(\bW_{n}\amalg\latch{n}'\Wd)\amalg_{\Lambda}\AbL W_{0}\cong
(\Ab\bW_{n}\amalg \Ab\latch{n}'(\Wd))\amalg_{\Lambda}\AbL W_{0}\\
\cong&~(\Ab\bW_{n}\amalg \latch{n}'(\Ab\Wd))\amalg_{\Lambda}\AbL W_{0}~\cong~
\Ab\bW_{n}~\amalg_{\Lambda}~\latch{n}(\AbL\Wd)
\end{split}
\end{equation*}
\noindent where the last coproduct is actually the direct sum
\begin{myeq}\label{eqnorm}
\Ab\bW_{n}~\oplus~\latch{n}(\AbL\Wd)
\end{myeq}
\noindent  in the abelian category \w[,]{(\M/\Lambda)\ab} and all the
abelianizations are  applied dimensionwise (cf.\ \S \ref{rabel}).

This implies that 
\begin{myeq}\label{eqnormbase}
\mC{n}\AbL\Wd~\cong~\norm{n}(\AbL\Wd)~=~\AbL\bW_{n}=\Ab\bW_{n}~,
\end{myeq}
\noindent since over the abelian category \w{(\M/\Lambda)\ab} the
Moore chains \w{\mC{\ast}} can be identified with the normalized chains
\w[,]{\norm{\ast}:s(\M/\Lambda)\ab\to\Ch((\M/\Lambda)\ab)} where by
definition, \w{\norm{n}\Ad=A_{n}/\latch{n}\Ad} for 
\w{\Ad\in s(\M/\Lambda)\ab} (cf.\ \wref[]{eqnorm} and \cite[(1.12)]{DolH}).

The abelianization \w{\AbL\bW_{n}} of \w[,]{0:\bW_{n}\to\Lambda} is just
\w[,]{0:\Ab\bW_{n}\to\Lambda} where \w{\Ab\bW_{n}} has a trivial
restricted $\Lambda$-module structure (cf.\ \S \ref{dmod}).
Therefore maps into \w{(p:\K\to\Lambda)} (over $\Lambda$) factor
through \w[.]{\Ker p=:\baK}  Applying \w{\Hom_{\Lambda}(-,\K)} to
\wref{eqnormbase} thus yields: 
\begin{equation*}
\begin{split}
\Hom_{\Lambda}(\mC{n}\AbL\Wd,\K)~\cong&~\Hom_{\Lambda}(\AbL\bW_{n},\K)~
\cong~\Hom_{\RL}(\Ab\bW_{n},\baK)\\
~\cong&~\Hom_{\M\ab}(\Ab\bW_{n},\baK)~\cong~\Hom_{\M}(\bW_{n},\baK)~,
\end{split}
\end{equation*}
\noindent as required.
\end{proof}

Combining Lemma \ref{lcocw} and Proposition \ref{pmoore}, we have:

%
%      Corollary:  representing cohomology classes by module maps
%
\begin{cor}\label{cco}
For \w{\Wd} as above, every cohomology class in \w{\HL{n}(\Wd;\K)}
is determined by a map of \Tal[s] \w[.]{\phi:\bW_{n}\to\baK} 
\end{cor}

Notice the map $\phi:\bW_{n} \to \baK$ represents zero in
\w{\HL{n}(\Wd;\K)} if and only if there is a commuting diagram in $\M$:
$$
\xymatrix@R=25pt{
\bW_{n} \ar[d]_{\phi} \ar[rr]^{\bdz{\bW_{n}}} && W_{n-1}
\ar[d]_{\psi}\ar[rd] & \\
\baK \ar[rr]^{i} && \K \ar[r]^{p} & \Lambda~.
}
$$

\begin{mysubsection}{Description of $k$-invariants}\label{sdki}
Let \w{\Kd\in s\M} for \w{\M=\TAlg} as above, with \w{\tPo{n}\Kd} its
$n$-th Postnikov section. Recall from \wref{eqkinv}
that the functorial $n$-th $k$ invariant
\w{k_{n}\in\HAQ{n+2}(\tPo{n}\Kd;\,\pi_{n+1}\Kd)} fits into a homotopy
pullback square:
\mydiagram[\label{eqkinvpa}]{
\ar @{} [dr] |<<<<<<{\framebox{\scriptsize{PB}}}
\tPo{n+1}\Kd \ar[r]^{p\q{n+1}} \ar[d] & \tPo{n}\Kd \ar[d]^{k_{n}}\\
\tBL \ar[r] & \tEL{\pi_{n+1}\Kd}{n+2}
}
\noindent for \w[.]{\Lambda =\pi_{0}\Kd}  This is constructed as in
\cite[\S 6]{BDGoeR} by first taking the homotopy pushout $Z$ of the
upper left corner of \wref[,]{eqkinvpa} and then noting that
\w[.]{\tPo{n+2}Z\simeq\tEL{\pi_{n+1}\Kd}{n+2}}

We can represent  \w{k_{n}} for \w{\Kd}  by the map in \w[:]{\M/\Lambda}
\begin{myeq}\label{eqbcocycle}
b:(\tPo{n}\Kd)_{n+2}\to\pi_{n+1}\Kd
\end{myeq}
\noindent which sends any \wwb{n+2}simplex
\w{\sigma\in(\csk{n+1}\Kd)_{n+2}} to the class in
\w{\pi_{n+1}\Kd} represented by the matching collection of
\wwb{n+1}faces
$$
(d^{n+2}_{0}\sigma,\dotsc,d^{n+2}_{n+2}\sigma)~\subseteq~
(\csk{n+1}\Kd)_{n+1}=\K_{n+1}~.
$$
\noindent Note that this collection need no longer have an \wwb{n+2}dimensional
fill-in in \w[,]{\Kd} but it does represent a map from the boundary of an
\wwb{n+2}simplex into \w[,]{\Kd} and so an element of
\w[.]{\pi_{n+1}\Kd}

In forming the homotopy pushout, we actually replace the map
\w{p\q{n+1}} which is the identity up through dimension \w[,]{n+1}
by a cofibration, by adding a $k$-simplex to \w{\Po{n}\Kd} for each
non-unique filler of a matching collection of \wwb{n+1}faces
in \w{\Po{n+1}\Kd} for \w[.]{k \geq n+2} However, any matching
collection in \w{\Po{n+1}\Kd} with a filler represents zero in
\w{\pi_{n+1}\Kd} so this does not affect \w[,]{[b]} while the
above description remains valid otherwise.

Since \w{\pi_{n+1}\Kd} is an abelian group object in
\w{\M/\Lambda} for \w[,]{\Lambda=\pi_{0}\Kd} the map $b$ factors
through the abelianization  \w[,]{\AbL(\tPo{n}\Kd)_{n+2}} so we can
  apply Proposition \ref{pmoore} to produce a cohomology class.
\end{mysubsection}

%
%c3   Rectification of diagrams and higher homotopy operations
%
\sect{Rectification of diagrams and higher homotopy operations}
\label{chaho}

Higher homotopy or cohomology operations have been studied extensively
since they were first discovered over fifty years ago, but there is
still no completely satisfactory theory that adequately covers all
known examples and explains their properties. In the approach we take
here,  based on that of \cite{BMarkH} (as modified in \cite{BChachP}),
they appear as ``geometric'' obstructions to realizing
homotopy-commutative diagrams. 

\begin{mysubsect}{The rectification problem}\label{srp}

Let $\M$ be a model category, and $\D$ a small category. We start with
a functor \w[,]{\tX:\D\to\ho\M} which we would like to \emph{rectify}
\wh that is, lift to a functor \w[.]{X:\D\to\M} 

By definition, we can choose a \emph{function} \w{X\arr:\Arr\D\to\Arr\M} 
which assigns to each arrow \w{\phi:a\to b} in $\D$ a map
\w{X\arr(\phi):\tX(a)\to\tX(b)} representing \w[.]{\tX(\phi)}
Moreover, for each two composable arrows \w{a\xra{\phi}b\xra{\psi}c}
in $\D$ we can choose a homotopy
\w[.]{H_{(\psi,\phi)}:X\arr(\psi\circ\psi)\sim X\arr(\psi)\circ
  X\arr(\phi)} The idea is that if we can choose these homotopies
\emph{compatibly}, in an appropriate sense, then the diagram $\tX$ can
be rectified; and that higher homotopy operations arise as the
obstructions to making such a choice. 

To make this precise, we need suitable function complexes for $\M$, in
which to house the higher homotopies: we work here
with the more familiar simplicial enrichment of $\M$ (although the
cubical version is more economical) \wh more precisely, with the
pointed version, enriched in \w[.]{\Sa} 

In fact, we only need the mapping spaces between the objects of $\M$
which are in the image of $\tX$, so let \w{\MO} denote the
\ww{\SaO}-category with object set \w{\OO:=\Obj\D} and
\w[.]{\map_{\MO}(a,b):=\map_{\M}(\tX a,\tX b)} We always assume 
\w{\tX a} is cofibrant and \w{\tX b} is fibrant. Thus the 
\emph{homotopy category} \w{\pi_{0}\MO} of \w{\MO} may be thought of
as a subcategory of the original \w[,]{\ho\M} and we can replace $\tX$ by a
functor \w[.]{\bX:\D\to\pi_{0}\MO} If we think of $\D$ as the constant
\ww{\SaO}-category \w[,]{\co{\D}} rectifying $\tX$ is equivalent to
lifting $\bX$ to an \ww{\SaO}-functor \w[.]{\hX{}:\co{\D}\to\MO}

In \cite[\S 1]{DKanS}, Dwyer and Kan define a simplicial model
category structure on \w[,]{\SOC} also valid for \w{\SaOC} (cf.\
\cite[Prop.~1.1.8]{HovM}), in which the fibrations and weak
equivalence are defined objectwise (that is, on each mapping space
\w[).]{\Xd(a,b)} Thus the above lifting problem can be stated in a
homotopically meaningful way if we use a cofibrant replacement for
\w[,]{\co{\D}} and require \w{\MO} to be fibrant (which just means
that each mapping space of \w{\MO} is a Kan complex).

The cofibrant objects in \w{\SaOC} are not easy to describe, in
general. However, a canonical cofibrant replacement for \w{\co{\D}} is
given by the simplicial category \w{\Fs\D} obtained by iterating the
comonad \w[,]{FU:\OC\to\OC} where the \emph{free category} functor
\w{F:\DiGa\to\Cat} is left adjoint to the forgetful functor
\w{U:\Cat\to\DiGa} to the category of directed pointed graphs.

A lift of \w{\bX:\D\to\pi_{0}\MO} to an \ww{\SaO}-functor
\w{\hX{}:\Fs\D\to\MO} is called an \emph{$\infty$-homotopy commuting}
version of the original $\tX$, and by \cite[Theorem~IV.4.37]{BVogHI} or
\cite[Theorem~2.4]{DKSmH}, its existence is equivalent (in a
homotopy-invariant sense) to solving the original rectification
problem, in the case where \w{\M=\Ta} or \w[.]{\Sa}
\end{mysubsect}

\begin{mysubsection}{The inductive process of rectification}
\label{sipr}
Following \cite{DKSmH}, we wish to construct such an $\infty$-homotopy
commuting lift $\hX{}$ of $\tX$ by induction over the Postnikov
sections (applied to each simplicial mapping space of the target
category $\M$). For our purposes we need the relative version, in
which $\C$ is a subcategory of the given indexing category $\D$, and
\w{\tX\rest{\C}} has already been rectified. 

Let \w{\OO=\Obj\C} and \w[,]{\Op:=\Obj\D} and note that the inclusion 
\w{i:\C\hra\D} induces a cofibration \w{\Fs i:\Fs\C\hra\Fs\D} in \w{\SaOpC}
(where we think of an \ww{\SaO}-category as an \ww{\SaOp}-category by
extending trivially). Thus the following pushout in \w[:]{\SaOpC}  
\mydiagram[\label{eqhpo}]{
\Fs\C \ar[d]_{r} \ar[r]^{\Fs i} &\Fs\D \ar[d]^{s} \\
\co{\C} \ar[r] & \Fs(\D,\C)
}
\noindent is in fact a homotopy pushout (and the vertical maps are
weak equivalences).

We begin the induction with an \ww{\SaO}-functor \w[,]{X:\co{\C}\to\MO}
(a rectification of \w[),]{\tX\rest{\C}} and a compatible
\ww{\SaOp}-functor \w{X_{0}:\Fs\D\to \Po{0}\MOp} lifting $\tX$, which
together induce \w[.]{\hX{0}:\Fs(\D,\C)\to \Po{0}\MOp} 

Now assume by induction on \w{n>0} that we can lift \w{\hX{0}}
to \w[,]{\hX{n-1}:\Fs(\D,\C)\to \Po{n-1}\MOp} making the following
diagram commute:
\mydiagram[\label{eqdks}]{
\Fs\C \ar[d]_{r}^{\simeq} \ar[r]^{\Fs i} &\Fs\D \ar[d]^{s}_{\simeq} & \\
\co{\C} \ar[r] \ar[d]_{X} & \Fs(\D,\C)\ar[rd]^{\hX{n-1}} & \\
\MO \ar[r]_{j} & \MOp \ar[r]_{q^{n}} & \Po{n-1}\MOp
}
\noindent and our goal is to identify the obstruction to lifting 
\w{\hX{n-1}} to \w[.]{\hX{n}:\Fs(\D,\C)\to \Po{n}\MOp}
\end{mysubsection}

\begin{mysubsection}{The Dwyer-Kan-Smith obstruction theory}
\label{sdks}
Although \w{\SaOC} (the category of simplicially enriched categories
with fixed object set $\OO$) is not quite a resolution model category
as defined here, Dwyer and Kan have shown that it has a notion of
\ww{\SaO}-cohomology, represented by Eilenberg-Mac~Lane objects 
\w{\E_{\D}(\K,n)} in \w[,]{\SaOC} defined much as in  \S \ref{dcoh}
(if we replace \PAa[s] by \ww{\GO}-categories as the target of
\w{\hpi} throughout), including a relative version  \w[.]{\HSO{\ast}(\D,\C;\K)}

We can thus use \w{\hX{n-1}} to pull back the \wwb{n-1}st $k$-invariant
\w{k_{n-1}:\Po{n-1}\MOp\to\E_{\D}(\pi_{n}\MOp,n+1)} for \w[,]{\MOp} as in
\wref[,]{eqkinv} to a map
\w[,]{h_{n-1}:=k_{n-1}\circ\hX{n-1}:\Fs(\D,\C)\to\E_{\G}(\pi_{n}\MOp,n+1)}
and deduce that the map \w{\hX{n-1}} lifts to \w{\hX{n}} if and only
if \w{[h_{n-1}]} vanishes in \w{\HSO{n+1}(\D,\C;\pi_{n}\MOp)} (see
\cite[Proposition 4.8]{DKSmH}). 
\end{mysubsection}

Our goal is to replace these \ww{\SaO}-cohomology obstructions by
geometrically defined higher homotopy operations, which we can then
identify in cases of interest with the Andr\'{e}-Quillen cohomology
obstructions of \cite{BDGoeR}. For this purpose, we must restrict
attention to a more limited class of indexing categories $\D$, defined
as follows: 

\begin{defn}\label{dlat}
A finite non-unital category $\Gamma$ will be called a \emph{lattice}
if it has no self-maps and is equipped with a \emph{(weakly) initial}
object \w{\vi} and a  \emph{(weakly) final} object \w[,]{\vf} such that
there is a unique \w[.]{\phi_{\max}:\vi\to\vf} 
We say that $\Gamma$ is \emph{pointed} if it is enriched in pointed
sets. In this case necessarily \w[.]{\phi_{\max}=\ast}

A composable sequence of $n$ arrows in $\Gamma$ will be called an
$n$-\emph{chain}. For any finite category $\Gamma$, the maximal
occuring $n$ is its \emph{length}. When $\Gamma$ is a lattice, this is
necessarily for a chain from \w{\vi} to \w[,]{\vf} factorizing
\w[.]{\phi_{\max}}
\end{defn}

\begin{example}\label{egtoda}
The simplest pointed lattice of interest to us is the \emph{Toda lattice}
of length $3$:
$$
\xymatrix@R=25pt{
\vi \ar[r]^{f} \ar@/^{1.5pc}/[rr]^{\ast} & u \ar[r]^{g}
\ar@/_{1pc}/[rr]_{\ast} & w \ar[r]^{h} & \vf~.
}
$$
\end{example}

\begin{mysubsection}{Higher homotopy operations}\label{shho}
Now let \w{\D=\Gp} be a (pointed) lattice, with objects
\w[,]{\Op:=\Obj\Gp} and \w{\C=\Gamma} a subcategory of \w{\Gp} with
object set \w[.]{\OO\subseteq\Op} We assume given a (pointed) diagram
\w{\tX:\Gp\to\ho\M} for an \ww{\Sa}-category $\M$, which we wish to 
rectify as above. In the cases of interest to us
\w{\Op\setminus\OO} will consist of a single object (either \w{\vi} or
\w[),]{\vf} and the maps in \w{\Gp} which are not in $\Gamma$ will be
(essentially) only zero maps. 

In the approach of \cite{BMarkH,BChachP}, we try to extend an
\ww{\SaO}-functor \w{X:\co{\C}\to\MO} to \w{\hX{}:\Fs(\GpG)\to\MOp} 
(see \S \ref{srp}), by induction over the skeleta of \w[.]{\Fs(\GpG)}
When \w{\MOp} is fibrant, this is essentially the same as the
induction in \S \ref{sipr}, since the $k$-coskeleton functor is a
\wwb{k-1}Postnikov section for a fibrant \ww{\SaO}-category. The
obstruction to extending to the \wwb{k+1}stage thus lies in a set of
relative homotopy classes of \ww{\SaOp}-functors  from
\w{(\sk{k+1}\Fs(\GpG),\sk{k}\Fs(\GpG))} to \w[,]{\MOp} which are in
general hard to describe. However, if \w{n+1} is the length of
\w[,]{\Gp} then \w[,]{\Fs\Gp} and thus \w[,]{\Fs(\GpG)} is
$n$-dimensional, and in this case the last obstruction is adjoint to a
wedge of maps \w{\Sigma^{n-1}\tX(\vi)\to\tX(\vf)} in $\M$ (see
\cite[Proposition 3.21]{BJTurH}). We then define the associated $n$-th
order higher homotopy operation associated to $\tX$ to be the
collection of elements: 
$$
\llrr{\tX}~\subseteq~[\Sigma^{n-1}\tX(\vi),\,~\tX(\vf)]_{\ho\M}~,
$$
\noindent obtained in this way from all possible extensions of $X$ to 
\w[.]{\sk{n-1}\Fs(\GpG)} Each such element is called a called a
\emph{value} of \w[.]{\llrr{\tX}}

 In the cases of interest to us here, $\M$ is equipped with a
 collection $\A$ of homotopy cogroup objects and \w{\tX(\vi)} is in
 \w{\MA} (\S \ref{dssmc}) so \w{\llrr{\tX}} takes 
values in the $\A$-homotopy groups of \w[.]{\tX(\vf)}

In \cite[Theorem~4.14]{BJTurH} the set \w{\llrr{\tX}}
was shown to be equivalent (for \w{\D=\Gp} a lattice) to
the set of \ww{\SaO}-cohomology classes appearing as the final
obstructions to rectification in \cite[Theorem~2.4]{DKSmH} \wh with a
particular cohomology class associated to each value of
\w{\llrr{\tX}} in such a way that they vanish simultaneously.

Note that the obstruction theory of \cite{BMarkH,BChachP} is
actually defined for $\M$ enriched in (pointed) \emph{cubical} sets 
(rather than in \w[),]{\Sa} using Boardman and Vogt's $W$-construction 
instead of \w{\Fs\D} (which is in fact a canonical triangulation thereof).
\end{mysubsection}

%
%c4   The Andr\'{e}-Quillen Cohomology Obstructions
%
\sect{The Andr\'{e}-Quillen Cohomology Obstructions}
\label{caqco}

In this section, we study the Andr\'{e}-Quillen cohomological
existence obstructions to realizing an abstract \PAa $\Lambda$, with
a view to comparing them to the higher homotopy operation
obstructions described in the next section.

\begin{mysubsection}{The general setting}\label{assumptions}
Given an \ww{E^{2}}-model category $\M$ (\S \ref{setmc}) with a
 collection of spherical objects $\A$, and an abstract \PAa $\Lambda$,
 one can try to \emph{realize} it by finding an object \w{X\in\M} with
 \w[.]{\piA X\cong\Lambda} 

For this purpose, we try to construct a cofibrant object
\w{\Vd} in \w{s\M} realizing a free simplicial resolution
\w{\Gd\to\Lambda} in \w[,]{s\PAAlg} (i.e. \w[)]{\piA\Vd = \Gd}
and \wref{eqrealwe} then implies that one can choose \w[.]{X=J\Vd}  On
the other hand, we know simplicial resolutions exist in both
\w{s\PAAlg} and \w[,]{s\M} and it follows from 
\cite[Proposition 3.13]{BlaCW}) that if $\Lambda$ is  
realizable, any choice of \w{\Gd} must also be realizable in the sense that
there is a CW object \w{\Vd\in s\M} with \w[.]{\piA\Vd\cong\Gd} Thus, the
realization problem for $\Lambda$ is reduced to one of realizing
CW objects in \w[.]{s\PAAlg}

Therefore, assume given a free simplicial \PAa resolution \w{\Gd} of a \PAa\
$\Lambda$ with CW-basis \w{\{\bG{n}\}_{n=0}^{\infty}} and 
\wwb{n+2}attaching map 
\w[.]{\bdz{\bG{n+2}}:\bG{n+2} \to \mZ{n+1}\Gd \subset \mC{n+1}\Gd} 	
In our inductive approach, we also assume given
an \wwb{n+1}truncated realization \w{\qVd{n+1}} of \w[:]{\Gd} that is,
the \wwb{n+1}truncations \w{\tau_{n+1}\piA\qVd{n+1}} and
\w{\tau_{n+1}\Gd} are isomorphic. We also choose a map  	
\w{\bdz{\bV{n+2}}:\bV{n+2}\to \mC{n+1}\qVd{n+1}} realizing the
attaching map \w[.]{\bdz{\bG{n+2}}:\bG{n+2} \to \mC{n+1}\Gd}
	 	
We may assume \w{\qVd{n+1}} is Reedy fibrant (as a truncated
simplicial object).  If we apply the \wwb{n-1}Postnikov section
functor to \w[,]{\qVd{n+1}} we obtain a (full) simplicial object
\w{\qWd{n}:=\csk{n}\qVd{n+1}} (cf.\ \S \ref{rsmc}). 

To start the induction, note that each basis object \w[,]{\bG{n}} and
thus each \w[,]{G_{n}} is free, so if we choose any realization for
\w[,]{\bdz{\bG{1}}:\bG{1}\to G_{0}} we obtain a strict $1$-truncated
realization \w{\qVd{1}} of \w[.]{\Gd} We can further choose a
realization \w{\bdz{\bV{2}}:\bV{2}\to C_{1}\qVd{1}} of
\w[,]{\bdz{\bG{2}}:\bG{2}\to Z_{1}\Gd\subseteq C_{1}\Gd} again by
\wref[.]{eqcommmoor} Finally, we can use \cite[Corollary 4.2]{BJTurR}
to guarantee that \w{d_{0}^{1}\circ\bdz{\bV{2}}=0} on the nose, thus
extending \w{\qVd{1}} to \w[,]{\qVd{2}} so there are no obstructions
for \w[.]{n=0} 

At the $n$-th stage, we deduce from \wref{eqspiral} that:
\begin{myeq}[\label{eqnqnat}]
\pinat{k}{\qWd{n}}\cong\begin{cases}
		\Omega^{k}\Lambda &~~\text{for \ } 0\leq k<n,\\
		 0     &~~\text{otherwise}\end{cases}
\end{myeq}
\noindent as well as:
\begin{myeq}[\label{eqnqpia}]
\pi_{k}\piA\qWd{n}\cong\begin{cases}
		\Lambda  &~~\text{for \ } k=0,\\
		\Omega^{n}\Lambda &~~\text{for \ } k=n+1,\\
		 0       &~~\text{otherwise}~.\end{cases}
\end{myeq}
\noindent Any simplicial object over $\M$ satisfying \wref{eqnqnat}
and \wref{eqnqpia} is called an \wwb{n-1}\emph{\sps}\ for $\Lambda$.

\end{mysubsection}

\begin{mysubsection}{The cohomology existence obstruction}
\label{sceo}
There are several variant descriptions of the Andr\'{e}-Quillen
cohomological existence obstructions, given by 
\cite[Theorem 4.15]{BlaAI}, \cite[\S 9]{BDGoeR}, and 
\cite[Theorem 6.4(b)]{BJTurR}, respectively. We recall the third version:

Under the setting of \S \ref{assumptions}, \w{\qWd{n}} need not extend
to a full resolution \w[,]{\Wd} with \w[,]{\pi_{0}\piA\Wd\cong\Lambda} 
and \w{\pi_{i}\piA\Wd=0} for \w[.]{i>0}  If it does so extend, then
the structure map \w{r\q{n-1}:\Wd\to\Po{n-1}\Wd=\qWd{n}} induces a map
of simplicial \PAa[s] \w{r\q{n-1}_{\#}:\tBL\to\piA\qWd{n}}
over $\Lambda$, which serves as a zero section, and implies that 
\w[.]{\piA\qWd{n}\cong \tEL{\Omega^{n}\Lambda}{n+1}}
In other words, if  \w{\qWd{n}}  extends to a full resolution, then
the simplicial \PAa \w{\piA\qWd{n}} has trivial
$k$-invariants.  Thus, we distinguish those \wwb{n-1}\sps s where 
the $n$-th $k$-invariant for \w{\piA\qWd{n}} vanishes by calling
them \wwb{n-1}\emph{\qps s}.

It turns out that if \w{\qWd{n}} is an \wwb{n-1}\sps, the $n$-th
$k$-invariant for the simplicial \PAa\ \w[,]{\piA\qWd{n}} which we
denote by:   
\begin{myeq}\label{eqaqcohcl}
\beta_{n}~\in~\HAQ{n+2}(\Lambda,\Omega^{n}\Lambda)~,
\end{myeq}
\noindent is precisely the obstruction to extending \w{\qWd{n}} to an
$n$-\sps\ \w{\qWd{n+1}} with \w{\csk{n}\qWd{n+1}\simeq\qWd{n}}
(see \cite[\S 9]{BDGoeR}). 
In other words, an \wwb{n-1}\sps\ extends one step further toward a
full resolution if and only if it  is an \wwb{n-1}\qps. Here we use
\wref{eqnqpia} to identify \w{\Po{n}\piA\qWd{n}} with \w{\tBL} (a
simplicial \PAa classifying object for $\Lambda$).  

Of course, there may be many ways to choose the extension
\w{\qWd{n+1}} once it is known to exist; these are classified by 
difference obstructions (see Section \ref{cdo} below). 
\end{mysubsection}

\begin{mysubsection}{Representing the $k$-invariant explicitly}
\label{speckinv}
Suppose that the map \w{\bdz{\bV{n+2}}:\bV{n+2} \to \mC{n+1}\qVd{n+1}}
actually lands in the \wwb{n+1}cycle object \w[.]{\mZ{n+1}\qVd{n+1}}
This would yield an extension of \w{\qVd{n+1}} to an \wwb{n+2}truncated
realization of \w{\Gd}  (cf.\ \S \ref{stepred}).  Since
\w{d_{i}\bdz{\bV{n+2}}=0} for \w{i\geq 1} by \wref[,]{eqmoor} the
composite map \w{d_{0} \bdz{\bV{n+2}}} is truly the only obstruction to extending
the realization. Moreover, since at the \PAa level \w{\bdz{\bG{n+2}}} lands
in the cycles, rather than just the chains, from \wref{eqcommmoor} we
know that  \w[.]{\left[d_{0} \bdz{\bV{n+2}}\right]=0} Thus the real question is
whether by varying \w{\bdz{\bV{n+2}}:\bV{n+2}\to\mC{n+1}\qVd{n+1}} within
its homotopy class (determined by \w[),]{\bdz{\bG{n+2}}} we can force the
null homotopic composite \w{d_{0} \bdz{\bV{n+2}}} to be strictly zero.
This sets the stage for constructing a higher homotopy operation in 
Section \ref{chhoeo}. First, we show how the cohomology class
\w{\beta_{n}} may be described in terms of \w[:]{d_{0}\bdz{\bV{n+2}}}
\end{mysubsection}

%
%     Proposition:  description of the existence obstruction
%
\begin{prop}\label{pexistobst}
Under the assumptions of \ref{assumptions} for \w[,]{n \geq 1}
the obstruction class \w{\beta_{n}} of \eqref{eqaqcohcl} is represented
in the sense of Corollary \ref{cco} by the map
\w{\bG{n+2} \to \pinat{n}\qVd{n+1}} induced by \w[.]{d^{V_{n+1}}_{0}
  \circ \bdz{\bV{n+2}}:\bV{n+2}\to\mZ{n}\qVd{n+1}}
\end{prop}	
	
\begin{proof}
Let \w{\qWd{n}:=\csk{n}\qVd{n+1}} be an \wwb{n-1}\sps\
for $\Lambda$ as above, and let \w{\Kd:=\piA\qWd{n}} be the corresponding
simplicial \PAa[.] We begin by constructing a weak equivalence
\w{f:\Gd\to\Po{n}\Kd=\csk{n+1}\Kd}  as follows:

Since \w{\Po{n}\Kd} is \wwb{n+1}coskeletal, it suffices to define
$f$ on \w[.]{\sk{n+1}\Gd} On the other hand, since
\w{\qWd{n}:=\csk{n}\qVd{n+1}} and
\w[,]{\tau_{n+1}\piA\qVd{n+1}=\tau_{n+1}\Gd} we see that
\w{\sk{n}\Po{n}\Kd} is isomorphic to \w[,]{\sk{n}\Gd} and thus for
simplicity we assume that $f$ is the identity through simplicial
dimension $n$. 

Since we want \w{f_{n+1}} to be a map of \PAa[s], and \w{G_{n+1}} is
free, by the Yoneda Lemma it is enough to say where \w{f_{n+1}} takes
the tautological \wwb{n+1}simplex \w[,]{\iota_{n+1}\in G_{n+1}\lin{V_{n+1}}}
 corresponding to \w[.]{\Id\in[V_{n+1},V_{n+1}]\cong G_{n+1}\lin{V_{n+1}}}
In order to describe:
$$
f_{n+1}(\iota_{n+1})~\in~(\csk{n+1}\Kd)_{n+1}\lin{V_{n+1}}~=~
\K_{n+1}\lin{V_{n+1}}~=~(\piA\csk{n}\qVd{n+1})_{n+1}\lin{V_{n+1}}~,
$$
\noindent we need a map \w{V_{n+1}\to(\csk{n}\qVd{n+1})_{n+1}=\match{n+1}\Vd} 
(the matching object for \w[):]{\Vd} that is, a strict matching
collection of \w{n+2} maps \w[.]{V_{n+1}\to V_{n}} This is provided by:
\begin{myeq}\label{eqmatch}
(d^{V_{n+1}}_{0},\,d^{V_{n+1}}_{1},\,\dotsc,d^{V_{n+1}}_{n+1})~.
\end{myeq}
\noindent Since $f$ so defined obviously commutes with the face and
degeneracy maps, this defines a map of simplicial \PAa[s] 
\w[.]{f:\Gd\to\Po{n}\piA\qWd{n}} Moreover, $f$ is necessarily a weak
equivalence, since the only non-trivial homotopy group on either 
side is in dimension $0$.

Combined with the description of the $n$-th $k$-invariant in \S
\ref{sdki}, we see that the obstruction
\w{\beta_{n}\in\HAQ{n+2}(\Lambda,\Omega^{n}\Lambda)} of
\wref{eqaqcohcl} is represented by the map of simplicial \PAa[s]
\w[.]{k_{n}\circ f:\Gd\to\widetilde{E}_{\Lambda}(\pi_{n+1}\Kd,n+2)}
This in turn is determined by the cocycle (\PAa map)
\w[,]{b\circ \mC{n+2}(f):\mC{n+2}\Gd\to\pi_{n+2}\Kd}
where $b$ is the cocycle of \wref[,]{eqbcocycle} by the naturality
in Propostition \ref{pmoore}.  By Lemma \ref{lcocw}, it suffices to
say where \w{b\circ \mC{n+2} (f)} sends the tautological \wwb{n+2}simplex
\w[,]{\bar{\iota}_{n+2}\in\mC{n+2}\Gd\lin{\bV{n+2}}\subset G_{n+2}\lin{\bV{n+2}}}
corresponding to the inclusion
\w[.]{\bar{\iota}_{n+2}:\bG{n+2}\hra G_{n+2}} Since the target is
\wwb{n+1}coskeletal, the \wwb{n+2}simplex
\w{f_{n+2}(\bar{\iota}_{n+2})} in \w{(\csk{n+1}\Kd)_{n+2}\lin{\bV{n+2}}} 
is given by the matching collection of \w{n+3} elements
in \w[:]{\K_{n+1}\lin{\bV{n+2}}}
$$
\left(f\left(d^{G_{n+2}}_{0}\circ \bar{\iota}_{n+2}\right),\,
f\left(d^{G_{n+2}}_{1}\circ \bar{\iota}_{n+2}\right),\,\dotsc,
f\left(d^{G_{n+2}}_{n+2}\circ \bar{\iota}_{n+2}\right)\,\right)~.
$$
\noindent This is simply
\w[,]{\left(f\left(\bdz{\bG{n+2}}\right),0,0,\dotsc,0\right)} 
since \w{\bar{\iota}_{n+2}:\bG{n+2}\hra G_{n+2}} lands in
\w[,]{\mC{n+2}\Gd} so \w{d^{G_{n+2}}_{i}\circ \bar{\iota}_{n+2}=0} for \w[.]{i>0}

Each element \w{f(\alpha)\in \mC{n+1}\Kd\lin{\bV{n+2}}} is
represented by a map 
\w{\tilde\alpha:\bV{n+2}\to(\csk{n}\Vd)_{n+1}=\match{n+1}\Vd}
induced by \wref{eqmatch} \wwh so it is the homotopy class of the
collection of \emph{strictly} matching maps \w[:]{\bV{n+2}\to V_{n}}
$$
\left(d^{V_{n+1}}_{0}\tilde\alpha,\,
d^{V_{n+1}}_{1}\tilde\alpha,\,d^{V_{n+1}}_{2}\tilde\alpha,\,\dotsc,
d^{V_{n+1}}_{n+1}\tilde\alpha\right)~.
$$
The class \w{\alpha=\bdz{\bG{n+2}}} is obtained by precomposing
the tautological class \w{\iota_{n+1} \in G_{n+1}\lin{V_{n+1}}} with
the homotopy class of \w[.]{\bdz{\bV{n+2}}:\bV{n+2} \to V_{n+1}}
Thus by \wref[:]{eqmatch}
\begin{align*}\label{eqdzmatch}
f(\alpha)~=&f\left(\iota_{n+1} \circ
\left[\bdz{\bV{n+2}}\right]\right)~=~
f\left(\iota_{n+1}\right) \circ \left[\bdz{\bV{n+2}}\right]~= \\
&\left[\left(d^{V_{n+1}}_{0},\,d^{V_{n+1}}_{1},\,\dotsc,
d^{V_{n+1}}_{n+1}\right)\right] \circ \left[\bdz{\bV{n+2}}\right]~= \\
&\left[\left(d^{V_{n+1}}_{0}\circ\bdz{\bV{n+2}},\,
d^{V_{n+1}}_{1}\circ\bdz{\bV{n+2}},\,\dotsc
d^{V_{n}+1}_{n+1}\circ\bdz{\bV{n+2}}\,\right)\right]~.
\end{align*}
\noindent Since we assumed that \w{\bdz{n+2}:\bV{n+2}\to V_{n+1}} lands in
\w[,]{\mC{n+1}\qVd{n+1}} this collection is equal to
\w[.]{\left[\left(d^{V_{n+1}}_{0}\circ\bdz{\bV{n+2}},0,0,\dotsc,0\right)\right]}

Since \w[,]{\qWd{n}:=\csk{n}\qVd{n+1}} we find that
\begin{myeq}\label{eqpcz}
\mC{n+1}\Wd~\cong~\mZ{n}\Vd
\end{myeq}
\noindent (see \cite[Fact 3.3]{BlaCH}). We have thus described a representing map
\w[,]{\bV{n+2}\to\mZ{n}\qVd{n+1}} as required.
\end{proof}

\begin{mysubsection}{Ladder diagrams}\label{sladders}
In order to identify the target of the map representing \w{\beta_{n}}
in Proposition \ref{pexistobst}, we must analyze the isomorphisms
\w{s_{i}} in the spiral long exact sequence \wref[.]{eqspiral} 
For this, we need the following technical tool:

Given \w{X \in \MA} and \w{\Yd \in s\M} with \w{\gamma_{m}:X\to\mZ{m}\Yd} 
a morphism in $\M$, let \w{\cone{X}} denote the cone on $X$, and
\w[,]{i:X\to\cone{X}} \w{j_{m}:\mZ{m}\Yd\to\mC{m}\Yd} 
the inclusions. If \w{g_{m}=j_{m}\gamma_{m}} is null
homotopic, then a choice of null homotopy \w{H_{m}} 
yields a commutative diagram:
$$
\xymatrix@R=25pt{
X \ar[d]^{\gamma_{m}} \ar[dr]^{g_{m}} \ar[r] & \cone{X}
\ar@{.>}[d]^{H_{m}} \ar[r] & \Sigma X \ar@{.>}[d]^{\gamma_{m-1}} \\
\mZ{m} \Yd \ar[r]_{j_{m}} & \mC{m} \Yd \ar[r]_{d_{0}} & \mZ{m-1} \Yd .
}
$$
\noindent The reason is \w{d_{0} j_{m}=0} by the definition of
\w[,]{\mZ{m} \Yd} hence it follows that 
\w[,]{d_{0} g_{m}=d_{0} j_{m} \gamma_{m}=0} so \w{d_{0}H_{m}} descends to a map
\w{\cone{X}/X \to \mZ{m-1}\Yd} and identifying \w{\cone{X}/X} with \w{\Sigma X}
produces a map \w[.]{\gamma_{m-1}:\Sigma X \to \mZ{m-1}\Yd}
\end{mysubsection}

\begin{defn}\label{dladder}
A \textit{ladder diagram} from \w{\gamma_{m}} to \w{\gamma_{k}}
is a commutative diagram of the form:
$$
\xymatrix@R=25pt{
X \ar[d]^{\gamma_{m}} \ar[dr]^{g_{m}} \ar[r] & \cone{X} \ar[d]^{H_{m}} \ar[r] &
\Sigma X \ar[r] \ar[d]^{\gamma_{m-1}} & \dotsb\Sigma^{m-k-1} X \ar[r]
\ar[d]^{\gamma_{k+1}} \ar[dr]^{g_{k+1}}	
& \cone{\Sigma^{m-k-1}X} \ar[d]^{H_{k+1}} \ar[r] & 
\ar[d]^{\gamma_{k}} \Sigma^{m-k}X \\
\mZ{m} \Yd \ar[r]_{j_{m}} & \mC{m} \Yd \ar[r]_{d_{0}} & 
\mZ{m-1} \Yd \ar[r] & \dotsb \mZ{k+1}\Yd \ar[r]_{j_{k+1}} & 
\mC{k+1} \Yd \ar[r]_{d_{0}} & \mZ{k} \Yd ~.
}
$$
\end{defn}

Consider any ladder diagram ending in:
\mydiagram[\label{eqlad}]{
W \ar[d]^{\gamma_{n}} \ar[dr]^{g_{n}} \ar[r] & \cone{W} \ar[d]^{H_{n}} \ar[r] &
\Sigma W \ar[d]^{\gamma_{n-1}} \ar[dr]^{g_{n-1}} \\
\mZ{n} \Yd \ar[r]_{j_{n}} & \mC{n} \Yd \ar[r]_{d_{0}} & \mZ{n-1} \Yd
\ar[r]_{j_{n-1}} & \mC{n-1} \Yd ~. 
}
\noindent In order to continue building it to the right, we would need
to have chosen \w{H_{n}} carefully to induce \w{\gamma_{n-1}} such that
\w{g_{n-1}=j_{n-1} \circ \gamma_{n-1}} is null homotopic.
For this purpose, we can sometimes use the following:

\begin{lemma}\label{furtherladder}
If \w{W \in \MA} and \w[,]{\pi_{n-1}\piA\Yd\lin{\Sigma W}=0} then
the null homotopy \w{H_{n}} in \eqref{eqlad} can be chosen so the map
\w{g_{n-1}} it induces is null homotopic. 
\end{lemma}

\begin{proof}
Suppose \w{H_{n}} is chosen so that the resulting \w[.]{[g_{n-1}] \neq 0}
Then by assumption the map 
\w{d_{0}:\mC{n}\piA\Yd\lin{\Sigma X}\to\mZ{n-1}\piA\Yd\lin{\Sigma X}}
is surjective.  By \wref[,]{eqcommmoor} there is then a map
\w{\alpha_{n}:\Sigma X\to\mC{n}\Yd} with 
\w[,]{d_{0}[\alpha_{n}]=-[\gamma_{n-1}]} and we choose
\w{H'_{n}=H_{n}\top\alpha_{n}} (using the notation of \cite[\S 2]{SpaS}
for the coaction of \w{[\Sigma X,Z]} on a nullhomotopy \w[).]{H:\cone{X}\to Z}
This new nullhomotopy of \w[,]{\gamma_{n}} induces a map
\w{\gamma'_{n-1}:\Sigma X \to \mZ{n-1}\Yd} as above, which in turn
yields \w[.]{g'_{n-1}} In fact,
\begin{equation*}
\begin{split}
[g'_{n-1}]~=~(j_{m-1})_{\ast}[d_{0}(H_{n}\star\alpha_{n})]~=&~
(j_{n-1})_{\ast}[\gamma_{n-1}] + (j_{n-1})_{\ast}(-[\gamma_{n-1}])\\
=&~[g_{n-1}]-[g_{n-1}]~=~0~, 
\end{split}
\end{equation*}
as required.
\end{proof}

We now exploit the usual proof of exactness for the homotopy sequence
of a fibration at the fiber position to show:

\begin{prop}\label{pkinv}
Under the assumptions of \S \ref{assumptions},
given \w{\gamma_{n}=d_{0} \bdz{\bV{n+2}}:\bV{n+2} \to \mZ{n} \qVd{n+1}}
there exists a ladder diagram from \w{\gamma_{n}} to 
\w[,]{\gamma_{0}:\Sigma^{n}\bV{n+2}\to\mZ{0}\qVd{n+1}=V_{0}}  with
adjoint whose homotopy class \w{[\tilde \gamma_{0}]} in
\w{\Omega^{n}\pi_{0}\piA\qVd{n+1}\lin{\bV{n+2}}
\cong\Omega^{n}\Lambda\lin{\bV{n+2}}}
represents the obstruction class \w[.]{\beta_{n}}
\end{prop}

\begin{proof}
First, since \w{\qVd{n+1}} is a \qps, it follows from \wref{eqnqnat}
and \wref{eqnqpia} that the homomorphisms \w{s_{i}} in the spiral long
exact sequence \wref{eqspiral} are isomorphisms for \w[.]{0\leq i\leq n}  	
We assume by downward induction on $i$ (starting with \w{i=n} from
Proposition \ref{pexistobst}) that \w{\gamma_{i}} represents
\w{\beta_{n}} in 
\w[.]{\pinat{i}\qVd{n+1}\lin{\Sigma^{n-i}\bV{n+2}}\cong
\Omega^{n-i}\pinat{i}\qVd{n+1}\lin{\bV{n+2}}} 
	
By definition, \w{s_{i}} is induced by the connecting homomorphism for
the fibration sequence \wref{eqsn} (see \S \ref{srmc}).  Thus, 
a preimage of \w{\gamma_{i}:\Sigma^{n-i}\bV{n+2}\to \mZ{i}\qVd{n+1}}
under \w{s_{i}} is obtained by choosing a null homotopy
\w{H_{i}:\cone{\Sigma^{n-i}\bV{n+2}}\to \mC{i}\qVd{n+1}} for
\w[,]{j_{i}\circ\gamma_{i}} noting that 
\w{d_{0}\circ H_{i}\rest{\Sigma^{n-i}\bV{n+2}}=0} so \w{H_{i}} 
induces a map \w{\gamma_{i-1}:\Sigma^{n-i+1}\bV{n+2}\to \mZ{i-1}\qVd{n+1}}
with \w[.]{s_{i}[\gamma_{i-1}]=[\gamma_{i}]}  Evidently, this is just extending a
ladder diagram from \w{\gamma_{n}} to \w{\gamma_{i}}
and thereby producing a ladder diagram from \w{\gamma_{n}} to \w[,]{\gamma_{i-1}}
so is possible by \wref{eqnqpia} verifying
the vanishing condition required in order to apply Lemma \ref{furtherladder}.
\end{proof}

%
%c5  Higher homotopy operations as existence obstructions
%
\sect{Higher homotopy operations as existence obstructions}
\label{chhoeo}

In \S \ref{speckinv} we saw that vanishing of the $n$-th ``algebraic''
obstruction to realizing the \PAa $\Lambda$ determines whether we can
choose a representative \w{\bdz{\bV{n+2}}} for the attaching map 
\w{\bdz{\bG{n+2}}} so that the composite
\w{d_{0}\circ\bdz{\bV{n+2}}:\bV{n+2}\to\mZ{n}\qVd{n+1}} is \emph{zero} (and not
just null-homotopic).

The same question may also be addressed using the inductive
rectification process of \S \ref{sipr}, since we can view this as
trying to find a strict representative for a (pointed) diagram in the homotopy
category.  We need the relative version of \S \ref{sdks} for this,
since we want to ensure that \w{d_{i}^{n+1} \bdz{\bV{n+1}}} remains
strictly zero for \w[,]{i>0} and want to leave the \wwb{n+1}truncation
\w{\qVd{n+1}} untouched.

Because the Dwyer-Kan-Smith \ww{\SaO}-cohomology obstructions are
difficult to compute, we prefer the more ``geometric'' approach of \S
\ref{shho}; but for this the indexing category must be a lattice. 
This is not true of (truncated) simplicial objects, because of the
degeneracy maps.  However, in our case we can work with a smaller 
indexing category, which \emph{is} a lattice, by using the CW
structure to avoid dealing with the degeneracies. We can further
reduce the complexity by careful initial choices (see Section
\ref{cmhho} below). 

\begin{remark}\label{rcw}
By \cite[Theorem 3.16]{BlaCW} any chosen CW-resolution \w{\Gd} of a
\emph{realizable} \PAa \w{\Lambda :=\piA X} (cf. \S \ref{dcw}) can be
realized by a CW-resolution \w{\Vd} of $X$ in \w[,]{s\M} with
\w{\piA\bV{n}\cong\bG{n}} and \w[,]{(\bdz{\bV{n}})_{\#}=\bdz{\bG{n}}} and
thus \w[.]{\piA\Vd\cong\Gd}

We observe also that if we want to realize a CW-resolution \w{\Gd\to\Lambda}
in \ww{s\PAAlg} \wwh or equivalently, to rectify the corresponding
simplicial object up-to-homotopy in \w[,]{s(\ho\M)} obtained by choosing
some \w{\bV{n}\in\MA} with \w{\piA\bV{n}\cong\bG{n}} for each
\w[,]{n\geq 0} with \w{\bdz{V_{n}}} uniquely determined up to homotopy by
\w{\bdz{G_{n}}} \wwh it suffices to inductively rectify the corresponding
\emph{restricted} simplicial object (in which we forget the degeneracies),
since the degeneracies can then be reconstructed from \wref[.]{eqlatch}
\end{remark}

\begin{defn}\label{ddso}
Write \w{\tD{}} for the indexing category for restricted simplicial
objects, with \w{\Obj(\tD{}):=\{\bn\}_{n=0}^{\infty}} and maps
\w{d_{0},\dotsc,d_{n}:\bn\to\bnm} satisfying the simplicial identities
\w{d_{i}d_{j}=d_{j-1}d_{i}} for \w[.]{i<j}

Similarly, \emph{augmented} restricted simplicial objects are
represented by \w[,]{\tDp{}} with an additional object \w{\bmo} and
\w[.]{\var=d_{0}:\bze\to\bmo} The truncated categories \w{\tD{n}} and
\w{\tDp{n}} are the full subcategories with objects
\w{\bn,\bnm,\dotsc,\bze} and \w[,]{\bn,\bnm,\dotsc,\bze,\bmo}
respectively. The latter is a lattice of length \w[,]{n+1} with
(weakly) terminal object \w[.]{\vf:=\bmo}

Finally, let \w{\hD{n+2}} be the category with object set
\w[,]{\{\bze,\bo,\dotsc,\bnpp\}} all the maps of \w[,]{\tD{n+1}} and
one additional map \w{\bdz{}:\bnpp\to\bnp} such that
\w{d^{n+1}_{i}\circ\bdz{}=0} for \w[.]{0\leq i\leq n+1}
\end{defn}

Thus \w[,]{\qVd{n+1}} together with \w{\bV{n+2}} and its face maps,
define a functor \w[.]{\cVd{n+2}:\hD{n+2}\to\ho\M} Moreover, we have
maps \w{d_{i}^{V_{k}}:V_{k}\to V_{k-1}}
and \w{\bdz{\bV{n+2}}:\bV{n+2}\to V_{n+1}} which, together with
\w[,]{\hVd{n+1}} ``almost'' define a functor
\w{\hVd{n+2}:\hD{n+2}\to\M} lifting \w[,]{\cVd{n+2}}
in that all (pointed) identities of \w{\hD{n+2}} hold strictly, except:
\begin{myeq}\label{eqcompzero}
d_{0}\circ\bdz{\bV{n+2}}~\sim~0
\end{myeq}
\noindent (the composite need not be strictly zero).

Rectifying \w{\cVd{n+2}} (thus turning \w{\hVd{n+2}}
into a strict functor) is equivalent to producing a full
\wwb{n+2}truncated simplicial object \w{\qVd{n+2}} realizing
\w[,]{\tau_{n+2}\Gd} by Remark \ref{stepred}.

\begin{mysubsection}{Using the Dwyer-Kan-Smith approach}
\label{sudks}
In order to apply the inductive procedure of \S \ref{sipr},
let \w{\C:=\tD{n+1}} and \w[,]{\D=\hD{n+2}} so
\w{\OO=\{\bze,\bo,\dotsc,\bnp\}} and \w[.]{\Op=\OO\cup\{\bnpp\}} The
\wwb{n+1}truncated simplicial object \w{\qVd{n+1}} provides the
 strictly commuting diagram \w[,]{\tD{n+1}\to\M} and thus the constant
\ww{\SaO} map \w[.]{X:\co{\tD{n+1}}\to\MO}

The lifting diagram of \ww{\SaOp}-categories representing \wref{eqdks}
in our case is:
\mydiagram[\label{eqdksspec}]{
\Fs\tD{n+1} \ar[d]_{r}^{\simeq} \ar[r]^{\Fs i} &
\Fs\hD{n+2} \ar[d]^{\simeq} && \\
\co{\tD{n+1}} \ar[r]^{i_{\ast}} \ar[d]_{X} &
\QQ\ar[rd]^{\hX{\ell-1}} \ar@{.>}[r]^{\hX{\ell}} & \Po{\ell}\MOp \ar[d] & \\
\MO \ar[r]_{j} & \MOp \ar[r]_{q^{n}} & \Po{\ell-1}\MOp \ar[r]^<<<<{k_{\ell-1}} &
E_{\hD{n+2}}(\pi_{\ell}\MOp,\ell+1)
}
\noindent for each \w[,]{1\leq\ell\leq n} where $\QQ$ is the pushout
\w{\Fs(\hD{n+2},\tD{n+1})}  in \w[.]{\SaOpC} The initial choice of
\w{\hX{0}:\QQ\to\Po{0}\MOp=\csk{1}\MOp} is actually equivalent to the ``lax
functor'' \w[,]{\hVd{n+2}} together with choices of nullhomotopies
\w{H_{i}:d_{i}\circ\bdz{V_{n+2}}\sim 0} \wb{0\leq i\leq n+1}
in \wref[.]{eqcompzero}
\end{mysubsection}

\begin{mysubsect}{The geometry of $\QQ$}
\label{ssq}

To analyze the  \ww{\SaOp}-maps \w[,]{\hX{k}:\QQ\to\Po{k}\MOp}
or their adjoints \w[,]{\sk{k+1}\QQ \to \MOp} we need an explicit
description of the simplices of each simplicial mapping space
\w[.]{\map_{\QQ}(\bm,\bj)} We denote the $i$-th ``internal'' face map
of each such simplicial set by \w[,]{\partial_{i}} to avoid confusion
with the ``categorical'' face maps \w{d_{i}} (morphisms of
\w[).]{\tD{n+1}} 

Recall that the free simplicial resolution
\w{\Fs\K} of a category $\K$ is constructed by iterating the  free
category monad \w{FU} (cf.\ \S \ref{srp}). We denote a $k$-simplex
\w{\sigma^{k}} of \w{\Fs\K} by a \wwb{k+1}fold 
parenthesization of a sequence of maps from $\K$, which we think of as
representing a $k$-th order homotopy betwen the vertices of $\sigma$.
The $i$-th face \w{\partial_{i}\sigma^{k}} is obtained by omitting
the \wwb{i+1}st level of parentheses, and the degeneracy
\w{s_{j}\sigma^{k}} is obtained by iterating the \wwb{j+1}st level
of parentheses \wb[.]{0\leq i,j\leq k} The categorical composition is
denoted by concatenation. The augmentation  \w{\var:\Fs\K\to\K} is
defined by  dropping all parentheses (and composing in $\K$).

Note that the mapping space \w{\map_{\QQ}(\bm,\bj)} is discrete unless
\w[,]{\bm=\bnpp} since otherwise \w{\Fs i} (at the top of
\wref[),]{eqdksspec}  and consequently \w[,]{i_{\ast}} too, are 
isomorphisms. Consequently each pair of parentheses not surrounding a
map out of \w{\bnpp} represents a homotopy in a discrete space, so
will be omitted. Thus we only need to consider nested 
parentheses having \w{\bdz{}} (the only new $0$-vertex in 
\w[)]{\QQ\setminus\Fs\tD{n+1}} in the innermost parentheses.

Therefore, we further abbreviate the standard notation by omitting all 
right parentheses, and replacing left parentheses by vertical
bars. Moreover, every string representing a simplex $\sigma$ of 
\w{\map_{\QQ}(\bnpp,\bj)} has \w{\bdz{}} as its last entry, so we can 
omit it. Thus, \w{|d_{i}|d_3|} represents \w{(d_{i}(d_3(\bdz{})))}
in the standard notation, while \w{d_{i}||d_3|=s_{0}(d_{i}|d_3|)}
is decomposable (that is, represents a composition with a degeneracy
of a zero simplex in some discrete case) because there is no
vertical bar on the extreme left.

Now any non-degenerate $k$-simplex $\sigma$ of $\QQ$ \wb{k\geq 1}
is necessarily in \w{\map_{\QQ}(\bnpp,\bj)} for some \w[.]{j\leq n-k+1} 
It can thus be written uniquely as:
$$
\sigma~=~d_{I_{0}}|d_{I_{1}}|d_{I_{2}}|\dotsb|d_{I_{k}}|d_{I_{k+1}}
$$
\noindent with a total of \w{k+1} vertical bars. Either or both of
 \w{I_{0}} and \w{I_{k+1}} can be empty.  

Any $k$-simplex 
in \w{\map_{\QQ}(\bnpp,\,\bnmkpo)} (the maximal dimension here)
will be called \emph{atomic}. In particular, the \emph{basic atomic}
$k$-simplex is: 
\begin{myeq}\label{eqatomic}
\tau_{k}~:=~|d_{0}|d_{0}|\dots|d_{0}|~.
\end{myeq}
\end{mysubsect}

\begin{defn}\label{faceguy}

For fixed \w[,]{n\geq 1} a $k$-\emph{flag} is a sequence
\w{\vrp=(i_{1},i_{2},\dotsc,i_{k})} of \w{|\vrp|:=k} integers with 
\w[.]{0\leq i_{1}<i_{2}<\dotsc<i_{k}\leq n+1} 
The collection \w{\Psi^{n+2}_{n-k+1}} of $k$-flags is in one-to-one
correspondence with the set \w[,]{\Hom_{\tD{n+1}}(\bnpp,\bnmkpo)}
where $\vrp$ represents the map 
\w{d_{\vrp}=d_{i_{1}}d_{i_{2}}\dots d_{i_{k}}\bdz{}:\bnpp\to\bnmkpo}
in \w{\hD{n+2}} (cf.\ \S \ref{ddso}). It thus determines a
$k$-dimensional simplicial set \w{\cF{\vrp}} \wwh namely, the
(pointed) component of \w{(d_\vrp)} in \w[,]{\QQ(\bnpp,\,\bnmkpo)}
which we call the \emph{flag complex} for $\vrp$. We may describe \w{\cF{\vrp}}
explicitly as follows:

A $j$-simplex $\sigma$ of \w{\cF{\vrp}} is determined by an expression
of the form 
$$
d_{\ell^{0}_{1}}\dots d_{\ell^{0}_{m^{0}}}~|~
d_{\ell^{1}_{1}}\dots d_{\ell^{1}_{m^{1}}}~|~\dotsc~|~
d_{\ell^{j+1}_{1}}\dots d_{\ell^{j+1}_{m^{j+1}}}~, 
$$
\noindent with \w{\ell^{t}_{1}<\ell^{t}_{2}<\dotsc<\ell^{t}_{m^{t}}}
and \w{j+1} vertical bars, where the composite (obtained by omitting
all bars) is \w[,]{d_\vrp} and we allow bars at either end
of the sequence, as well as repeated bars. 
The face map \w{\partial_{i}} removes the \wwb{i+1}st vertical bar
(rewriting the resulting sequence in  standard form, if necessary),
and the $i$-th degeneracy repeats the \wwb{i+1}st bar, as in $\QQ$.  
\end{defn}

\begin{example}\label{egfaces}
For \w[,]{\vrp=(2<4<5)} we have a $3$-simplex \w{|d_3|d_3|d_{2}|}
in \w[,]{\cF{\vrp}} since \w[.]{d_3d_3d_{2}=d_{2}d_4d_5} 
The $2$-simplices include \w[,]{|d_{2}d_4|d_5|} \w{d_{2}|d_4|d_5|}
and \w[.]{|d_{2}d_4|d_4|}

Moreover, \w{\partial_{1}\left[ |d_3|d_3|d_{2}|\right]=|d_3d_4|d_{2}|} and
\w[.]{\parz\left[|d_{2}d_4|d_5|\right]=d_{2}d_4|d_5|} 
\end{example}

\begin{defn}\label{dboundary}
For any flag $\vrp$ as above, the \emph{boundary} of the flag
complex \w[,]{\cF{\vrp}} denoted by \w[,]{\partial\cF{\vrp}} is
the \wwb{k-1}dimensional subsimplicial set spanned by
\w[,]{\partial_{i}\sigma} for $\sigma$ a $k$-simplex of 
\w{\cF{\vrp}} \wb[.]{0\leq i\leq k} The subcomplex of
\w{\partial\cF{\vrp}} spanned by the zero-faces \w{\partial_{0}\sigma}
is called the \emph{base complex},  
written \w[;]{\cfbase{\vrp}} and the subcomplex spanned by the
other faces \wb{1\leq i\leq k} is written \w[,]{\cftop{\vrp}}
so \w[.]{\partial\cF{\vrp} \cong \cfbase{\vrp} \cup \cftop{\vrp}}

The vertex \w{|d_{i_{1}}d_{i_{2}}\dots d_{i_{k}}} of
\w[,]{\cF{\vrp}} lying in \w[,]{\cftop{\vrp}} will be called its
\emph{cone point}, and denoted by \w{c_{\vrp}} 
(see Lemma \ref{ldualperm} below). 
\end{defn}

\begin{examples}\label{egfc}
\noindent (1)~~Figure \ref{fig1} shows the $2$-dimensional flag
complex \w{\cF{\vrp}} for \w[.]{\vrp=(0<1)} It contains the basic
atomic $2$-simplex \w{\tau_{2}} as the left $2$-simplex (cf.\
\wref[).]{eqatomic}  

%
%      Figure 1:  a 2-dimensional flag complex
%
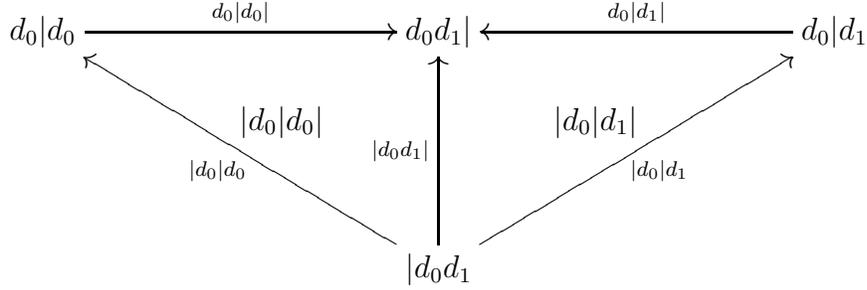
\begin{figure}[htbp]
$$
\xymatrix@R=15pt{
{d_{0}|d_{0}} \ar[rrr]^{d_{0}|d_{0}|}&&& {d_{0}d_{1}|}  &&& 
{d_{0}|d_{1}}  \ar[lll]_{d_{0}|d_{1}|} \\
&& {|d_{0}|d_{0}|}  && {|d_{0}|d_{1}|}  \\ \\
 &&&  \ar[uuulll]^{|d_{0}|d_{0}} {|d_{0}d_{1}} 
\ar[uuurrr]_{|d_{0}|d_{1}} \ar[uuu]^{|d_{0}d_{1}|}
}
$$
\caption[fig1]{The flag complex \w{\cF{\vrp}} for \w[.]{\vrp=(0<1)}}
\label{fig1}
\end{figure}

The base complex \w{\cfbase{\vrp}} consists of the top two
$1$-simplices \w{\parz\tau_{2}=d_{0}|d_{0}|} and \w[.]{d_{0}|d_{1}|}
Note that \w{\cfbase{\vrp}} is a triangulated dual
$2$-permutohedron on the set \w[,]{\{0,1\}} where
\w{d_{0}d_{1}} corresponds to \w{(0,1)} and \w{d_{0}d_{0}} corresponds
to \w{(1,0)} since the simplicial identities are involved in permuting
face maps\vsm. 

\noindent (2)~~Figure \ref{fig2} shows the atomic $3$-simplex 
\w{\sigma=|d_{0}|d_{1}|d_{0}|} of \w[:]{\map_{\QQ}(\bnpp,\bnmt)} more
precisely, we have cut open its boundary \w[,]{\partial\sigma} 
so that it can be depicted in the plane, with identifications of the
outer edges indicated by dotted arrows. The cone point \w{c_{\vrp}}
corresponds to the outer vertices \w{|d_{0}d_{1}d_{2}} (all
identified) and the central facet \w{\parz\sigma} is part of
\w[,]{\cfbase{\vrp}} while the bottom $2$-simplex in the figure
\w{\partial_{3}\sigma} forms part of the (upper) boundary
\w[.]{\cftop{\vrp}}   

%
%      Figure 2:  a 3-simplex
%
\begin{figure}[htbp]
\begin{picture}(350,240)(0,0)
%
%      left top vertex
%
\put(50,210){\circle*{5}}
\put(20,215){{\scriptsize $|d_{0}d_{1}d_{2}$}}
%
%       Top edge
%
\put(50,210){\line(1,0){260}}
\put(100,215){{\footnotesize $|d_{0}d_{1}d_{2}|$}}
\put(225,215){{\footnotesize $|d_{0}d_{1}d_{2}|$}}
\bezier{40}(135,225)(172,255)(225,225)
\put(135,225){\vector(-3,-1){3}}
\put(225,225){\vector(3,-1){3}}
%
%       Left edge
%
\put(50,210){\line(2,-3){130}}
\put(50,150){\footnotesize $|d_{0}|d_{0}d_{2}$}
\put(115,50){\footnotesize $|d_{0}|d_{0}d_{2}$}
\bezier{50}(45,143)(20,75)(110,54)
\put(45,140){\vector(1,3){3}}
\put(110,54){\vector(3,-1){3}}
%
%       middle top vertex
%
\put(180,210){\circle*{3}}
\put(165,215){{\scriptsize $d_{0}d_{1}d_{2}|$}}
%
%       right top vertex
%
\put(310,210){\circle*{5}}
\put(312,215){{\scriptsize $|d_{0}d_{1}d_{2}$}}
%
%       right edge
%
\put(310,210){\line(-2,-3){130}}
\put(278,150){\footnotesize $|d_{0}d_{1}|d_{0}$}
\put(208,50){\footnotesize $|d_{0}d_{1}|d_{0}$}
\bezier{50}(310,135)(325,85)(248,53)
\put(310,135){\vector(-1,3){3}}
\put(248,53){\vector(-3,-1){3}}
%
%         middle left vertex
%
\put(115,112){\circle*{3}}
\put(80,110){{\scriptsize $d_{0}|d_{0}d_{2}$}}
%
%         middle right top edge
%
\put(115,112){\line(2,3){65}}
\put(213,160){\footnotesize $d_{0}d_{1}|d_{0}|$}
%
%         central edge
%
\put(115,112){\line(1,0){130}}
\put(160,115){{\footnotesize $d_{0}|d_{1}|d_{0}$}}
%
%         middle right vertex
%
\put(245,112){\circle*{3}}
\put(249,110){{\scriptsize $d_{0}d_{1}|d_{0}$}}
%
%         middle left top edge
%
\put(245,112){\line(-2,3){65}}
\put(110,160){\footnotesize $d_{0}|d_{0}d_{2}|$}
%
%          bottom vertex
%
\put(180,15){\circle*{5}}
\put(165,3){{\scriptsize $|d_{0}d_{1}d_{2}$}}
%
%          facets
%
\put(114,199){\oval(21,13)}
\put(105,195){$\partial_{2}\sigma$}
\put(90,180){\fbox{\footnotesize $|d_{0}|d_{0}d_{2}|$}}
\put(244,199){\oval(21,13)}
\put(235,195){$\partial_{1}\sigma$}
\put(220,180){\fbox{\footnotesize $|d_{0}d_{1}|d_{0}|$}}
\put(179,159){\oval(21,13)}
\put(170,155){$\partial_{0}\sigma$}
\put(155,140){\fbox{\footnotesize $d_{0}|d_{1}|d_{0}|$}}
\put(155,80){\fbox{\footnotesize $|d_{0}|d_{1}|d_{0}$}}
\put(179,66){\oval(21,13)}
\put(170,63){$\partial_{3}\sigma$}
\end{picture}
\caption[fig2]{The boundary \w{\partial\sigma} of the $3$-simplex
\w{\sigma=|d_{0}|d_{1}|d_{0}|}}
\label{fig2}
\end{figure}
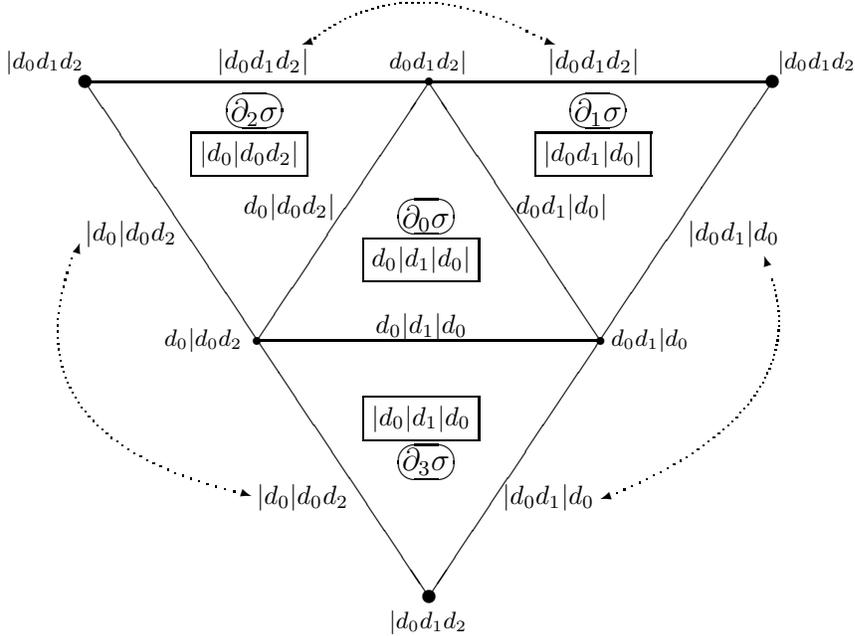

\noindent (3)~~Figure \ref{fig3} depicts the base complex  \w{\cfbase{\vrp}}  (a
triangulated dual $3$-permu\-to\-hedron) for the first flag of length
$3$ \w{\vrp=(0<1<2)} in \w[.]{\map_{\QQ}(\bnpp,\bnmt)}

%
%      Figure 3:  The triangulated 3-permutahedron
%
\begin{figure}[htbp]
\begin{picture}(350,345)(30,10)
%
%       central vertex
%
\put(210,180){\circle*{5}}
\put(216,177){{\scriptsize $d_{0}d_{1}d_{2}|$}}
%
%       left lower vertex
%
\put(50,100){\circle*{5}}
\put(18,95){{\scriptsize $d_{0}|d_{1}d_{2}$}}
%
%       LU diagonal
%
\put(50,100){\line(2,1){320}}
\put(95,145){{\footnotesize $d_{0}|d_{1}d_{2}|$}}
\put(297,215){{\footnotesize $d_{0}d_{1}|d_{0}|$}}
%\put(290,220){\vector(-3,-1){5}}
%
%       left edge
%
\put(50,100){\line(0,1){160}}
\put(12,175){{\footnotesize $d_{0}|d_{1}|d_{1}$}}
%
%      left upper vertex
%
\put(50,260){\circle*{5}}
\put(18,255){{\scriptsize $d_{0}d_{1}|d_{1}$}}
%
%       UL diagonal
%
\put(50,260){\line(2,-1){320}}
\put(125,225){{\footnotesize $d_{0}d_{1}|d_{1}|$}}
\put(245,135){{\footnotesize $d_{0}|d_{0}d_{2}|$}}
%
%       left upper edge
%
\put(50,260){\line(2,1){160}}
\put(91,300){{\footnotesize $d_{0}|d_{0}|d_{1}$}}
%
%      top vertex
%
\put(210,340){\circle*{5}}
\put(195,347){{\scriptsize $d_{0}|d_{0}d_{1}$}}
%
%       upper vertical diagonal
%
\put(210,340){\line(0,-1){320}}
\put(213,235){{\footnotesize $d_{0}|d_{0}|d_{1}$}}
%
%       lower vertical diagonal
%
\put(172,70){{\footnotesize $d_{0}d_{1}|d_{2}|$}}
%
%       right upper edge
%
\put(210,340){\line(2,-1){160}}
\put(295,300){{\footnotesize $d_{0}|d_{0}|d_{0}$}}
%
%      right upper vertex
%
\put(370,260){\circle*{5}}
\put(375,255){{\scriptsize $d_{0}d_{1}|d_{0}$}}
%
%       right edge
%
\put(370,260){\line(0,-1){160}}
\put(373,175){{\footnotesize $d_{0}|d_{1}|d_{0}$}}
%
%      right lower vertex
%
\put(370,100){\circle*{5}}
\put(375,95){{\scriptsize $d_{0}|d_{0}d_{2}$}}
%
%      right lower edge
%
\put(370,100){\line(-2,-1){160}}
\put(287,50){{\footnotesize $d_{0}|d_{0}|d_{2}$}}
%
%      bottom vertex
%
\put(210,20){\circle*{5}}
\put(197,10){{\scriptsize $d_{0}d_{1}|d_{2}$}}
%
%      left lower edge
%
\put(210,20){\line(-2,1){160}}
\put(95,50){{\footnotesize $d_{0}|d_{1}|d_{2}$}}
%
%          facets
%
\put(100,180){\fbox{\footnotesize $d_{0}|d_{1}|d_{1}|$}}
\put(120,260){\fbox{\footnotesize $d_{0}|d_{0}|d_{1}|$}}
\put(250,260){\fbox{\footnotesize $d_{0}|d_{0}|d_{0}|$}}
\put(280,180){\fbox{\footnotesize $d_{0}|d_{1}|d_{0}|$}}
\put(250,100){\fbox{\footnotesize $d_{0}|d_{0}|d_{2}|$}}
\put(120,100){\fbox{\footnotesize $d_{0}|d_{1}|d_{2}|$}}
\end{picture}
\caption[fig3]{\w{\cfbase{\vrp} \subset \map_{\QQ}(\bnpp,\bnmt)}
for \ $\vrp=(0<1<2)$}
\label{fig3}
\end{figure}
\end{examples}

\begin{lemma}\label{ldualperm}
For any flag $\vrp$ with \w[,]{|\vrp|=k} the base complex
\w{\cfbase{\vrp}} is isomorphic to a triangulated 
dual $k$-permutohedron (a \wwb{k-1}dimensional convex polytope), 
and \w{\cF{\vrp}} is the combinatorial cone on its zero face
\w[,]{\cfbase{\vrp}} with cone point \w[.]{c_{\vrp}} 
\end{lemma}

\begin{proof}
The \wwb{k-1}simplices of the base complex \w{\cfbase{\vrp}} 
may be listed by decomposing \w{d_\vrp} in all possible atomic
(mostly non-standard) forms, so as a composite of $k$ face maps with vertical bars 
inserted to the right of each map (but not at the left end). 
If we let the standard
representation (in ascending order) \w{d_{i_{1}}|d_{i_{2}}|\dots|d_{i_{k}}|}
correspond to the identity permutation, we have a faithful, effective,
and transitive action of the symmetric group \w{\Sigma_{k}} on the 
\wwb{k-1}simplices of \w[,]{\cfbase{\vrp}} in which any adjacent
transposition \w{(j,j+1)} takes 
\w{d_{\ell^{1}}|\dots|d_{\ell^j}|d_{\ell^{j+1}}|\dots|d_{\ell^{k}}|}
to 
\w[,]{d_{\ell^{1}}|\dots|d_{(\ell')^j}|d_{(\ell')^{j+1}}|\dots|d_{\ell^{k}}|}
by applying the simplicial identity 
\w[.]{d_{\ell^j}\circ d_{\ell^{j+1}}=d_{(\ell')^j}\circ d_{(\ell')^{j+1}}}
This shows that the base complex \w{\cfbase{\vrp}} is indeed the dual of the
triangulated $k$-permutohedron.
	
To see that \w[,]{\cF{\vrp}=\cone{\cfbase{\vrp}}} note that if we
compose all the factors in any representation
\w{|d_{i_{j_{1}}}|\dots|\dots|d_{i_{j_{k}}}} as above, we obtain the
same map \w[,]{d_{\vrp}} which implies that the initial vertex of each
$k$-simplex $\sigma$ of \w{\cF{\vrp}} is precisely the cone point \w[,]{c_{\vrp}}
explicitly \w[.]{|d_{i_{1}}\dots d_{i_{k}}} 
\end{proof}

\begin{remark}\label{rdualperm}
The terminal vertex of each simplex of \w{\cfbase{\vrp}} is 
\w[,]{d_{i_{1}}\dots d_{i_{k}}|} so the base complex is again the cone
on its own boundary.
\end{remark}

%
%      Corollary:  Flag complexes = balls
%
\begin{cor}\label{cdualperm}
For any flag $\vrp$ with \w[,]{|\vrp|=k} the boundary
\w{\partial\cF{\vrp}} of the flag complex is a combinatorial
\wwb{k-1}sphere. 
\end{cor}

\begin{proof}
The permutohedron \w{\Pe{k}} on $k$ elements is a \wwb{k-1}polytope
\wh that is, a convex \wwb{k-1}dimensional polyhedron in
\ww{\bR^{k-1}} (cf.\ \cite{GRosA}) so its dual \w{\cfbase{\vrp}} is 
also a \wwb{(k-1)}polytope (see \cite[\S 3.4]{GrunC}). Thus the cone
\w{\cF{\vrp}=\cone{\cfbase{\vrp}}} is a $k$-polytope, and its boundary
\w{\partial\cF{\vrp}} is combinatorially equivalent to a \wwb{k-1}sphere. 
\end{proof}

We can now describe the decomposition of the mapping spaces in $\QQ$:

\begin{lemma}\label{lqdec}
The space \w{\map_{\QQ}(\bm,\bj)} is empty if \w[,]{m<j} discrete for
\w[,]{m\neq n+2} and for \w{0 \leq k \leq n+1}
$$
\map_{\QQ}(\bnpp,\bnmkpo)~\cong~\bigvee_{\vrp\in\Psi^{n+2}_{n-k+1}}~\cF{\vrp}~,
$$
\noindent where the vertex \w{|d_{i_{1}}d_{i_{2}}\dots d_{i_{k}}} is
chosen as basepoint in each flag complex.
\end{lemma}

\begin{proof}
In the pushout  \eqref{eqdksspec} 
defining $\QQ$, for entries with \w[,]{m \neq n+2} the top map is an isomorphism, 
hence the bottom map is as well, so \w{\map_{\QQ}(\bm,\bj)} is discrete.

If \w[,]{m > j} by construction, the pointed mapping space 
\w{\map_{\QQ}(\bm,\bj)} decomposes into a wedge
of natural summands corresponding to distinct maps in \w{\tD{n+1}} \wwh
that is, iterated face maps \w[.]{\bm\to\bj} However, these summands are just
the flag complexes, since the composite
obtained by omitting the bars from any decomposition of $\vrp$ always
equals \w[.]{d_{i_{1}}d_{i_{2}}\dots d_{i_{k}}}
\end{proof}

\begin{remark}\label{rfirstflag}
Note that the basic atomic $k$-simplex 
\w{\tau_{k}=|d_{0}|d_{0}|\dots|d_{0}|} (cf.\ \S \ref{ssq}) is a top
dimensional simplex of \w{\cF{0<1<2<\dots<k-1}} in each case, 
while \w{\parz\tau_{k}} decomposes (in the simplicial enrichment of
$\QQ$) as \w{\tau_{k-1}} followed by (the degeneracy of) the 
$0$-simplex \w[.]{d_{0}:\bjp\to\bj} 
We will see later, that with appropriate initial choices, we can send
any top dimensional simplex other than $\tau_k$ to zero, and one only really
needs to extend over this single simplex, imposing this one
relation, at each stage of the induction.

Given a flag \w[,]{\vrp=\{0\leq i_{1}<i_{2}< \dots <i_{k}\leq n+1\}}
write \w[.]{\vrp^{\hat j}:=
\{0\leq i_{1}<i_{2}<\dots i_{j-1}<i_{j+1}\dots<i_{k}\leq n+1\}} 
For \w{j<\ell} write 
\w[.]{\vrp^{\hat j, \hat \ell}:=\{0\leq i_{1}< \dots i_{j-1}<i_{j+1}\dots
i_{\ell-1}<i_{\ell+1}\dots<i_{k}\leq n+1\}}
Let \w{d_{i_{j}-j+1} \circ \cF{\vrp^{\hat j}}} denote the
subcomplex of \w{\cfbase{\vrp}} spanned by simplices of the form
\w[,]{d_{i_{j}-j+1}|\sigma} where \w{\sigma} represents a simplex of
\w[.]{\cF{\vrp^{\hat j}}}
\end{remark}

\begin{lemma}\label{basedecomp}
Each component \w{\cF{\vrp}} of the mapping space
\w{\map_{\QQ}(\bnpp,\bnmkpo)} has
\w[,]{\parz \cF{\vrp}~\cong~\bigcup_{j} d_{i_{j}-j+1}\circ
  \cF{\vrp^{\hat j}}}
where each \wwb{k-2}simplex not in the boundary is shared
by only two \wwb{k-1}simplices, and if \w[:]{j<\ell}
$$
d_{i_{j} -j+1} \circ \cF{\vrp^{\hat j}}
~\bigcap~d_{i_{\ell} -\ell+1} \circ \cF{\vrp^{\hat \ell}}~=~
d_{i_{j} -j+1} d_{i_{\ell} -\ell+1}\circ\cF{\vrp^{\hat j,\hat \ell}}~.
$$	
\end{lemma}

\begin{proof}
This is a straightforward calculation, based on the simplicial identity
$$
d_{i_{1}} d_{i_{2}}\dots d_{i_{k}}~=~
d_{i_{j}-(j-1)}d_{i_{1}}\dots d_{i_{j-1}} d_{i_{j+1}}\dots d_{i_{k}}~,
$$
\noindent since the face map \w{d_{i_{j}}} is moved forward past
\w{j-1} others, whose indices are always lower by assumption.
Permutations naturally break up according to those which move a fixed
term to the front, with each such piece a copy of a permutation group
on a set with one less element.  The same applies a second time to get
the intersection statement.
\end{proof}

\begin{mysubsection}{The inductive procedure}
\label{sipq}
Let \w{\QQ[\bnpp,\bj]} denote the full subsimplicial category of $\QQ$
which only contains the objects \w[.]{\bnpp,\bnp,\dots,\bj}
Note that any simplicial functor \w{\hX{k}:\QQ[\bnpp,\bnmkpo]\to\MOp}
(compatible with \w[)]{X:\tD{n+1} \to \MO} extends uniquely to a
simplicial functor \w{:\parz\QQ[\bnpp,\bnmk] \to\MOp}
(where the face is only taken in the last mapping space) by
Lemma \ref{basedecomp} (for each $\vrp$) and extends
by zero to \w{\cftop{\vrp}} (the rest of \w[).]{\partial\cF{\vrp}}
Together with \w[,]{\hX{k}} this yields a simplicial functor
\w{\tX^{k}:\partial\QQ[\bnpp,\bnmk] \to \MOp} (where again the
boundary is only taken in the last mapping space).

For any flag with \w[,]{|\vrp|=k+1} note that
\w{\tX^{k}\rest{\partial \cF{\vrp}}} sends \w{\cftop{\vrp}}
to the zero map. Since the target in \w{\MOp} is assumed to be a Kan
complex, we can instead consider the induced map \w{f_{\vrp}} from the
quotient 
\w[,]{\partial\cF{\vrp}/\cftop{\vrp}~\cong~\cfbase{\vrp}/\partial\cfbase{\vrp}}
which is a  $k$-sphere by Lemma \ref{ldualperm}.  Moreover, the adjoint
map \w{\tilde{f}_{\vrp}:\Sigma^{k}\bV{n+2} \to V_{n-k}} is null
homotopic precisely when \w{f_{\vrp}} is such,
or equivalently, when \w{\tX^{k}\rest{\partial \cF{\vrp}}} has a filler
to all of \w[.]{\cF{\vrp}}
\end{mysubsection}

%
%           Proposition: inductive obstruction to extensions
%
\begin{prop}\label{pnullextend}
A simplicial functor \w{\hX{k}:\QQ[\bnpp,\bnmkpo]\to\MOp}
(compatible with a fibrant \w[)]{X:\tD{n+1} \to \MO} extends to a
simplicial functor  \w{\hX{k+1}:\QQ[\bnpp,\bnmk] \to \MOp} if
and only if  for each flag of length \w{k+1} the induced map
\w{f_{\vrp}} is null homotopic.
\end{prop}

\begin{proof}
If there is an extension, the fact that the full \w{\cF{\vrp}} serves
as a cone on its boundary $k$-sphere means the extension serves as a
null homotopy of the restriction to \w[,]{\partial \cF{\vrp}}
thereby implying that \w{f_{\vrp}} is also null homotopic.
	
Conversely, if \w{f_{\vrp}} is null homotopic, given the choice of a
null homotopy  $H$ for \w{\widetilde X^{k}\rest{\partial\cF{\vrp}}}
and an $i$-simplex
\w{\sigma \in \cF{\vrp} \setminus\partial\cF{\vrp}} (with \w{i=k}
or \w[),]{k+1} $H$ determines an $i$-simplex \w{\hX{k+1}(\sigma)\in
\map_\M(\widetilde X^{k}(\bnpp),\widetilde X^{k}(\bnmk))}
since the target is a Kan complex.  Note that any such $i$-simplex
$\sigma$ is indecomposable, so \w{\hX{k+1}} so defined (and
extending \w[)]{\hX{k}} is indeed a simplicial functor.
\end{proof}

\begin{defn}\label{dhhoa}
Given a flag \w[,]{\vrp=(0 \leq i_{1}< \dots i_{k+1} \leq n+1)} with
corresponding map 
\w[,]{d_{\vrp}=d_{i_{1}}\dots d_{i_{k+1}}\bdz{}:\bnpp\to\bnmk}
by Corollary \ref{cdualperm} the boundary \w{\partial\cF{\vrp}} of the
flag complex is a simplicial  $k$-sphere. Therefore, the adjoint of
\w{f_{\vrp}=\widetilde X^{k}\rest{\partial\cF{\vrp}}:\partial K_{\vrp}\to
\map_{\M}(\bV{n+2},V_{n-k})}
may be thought of as a map
\w{\tilde{f}_{\vrp}:\Sigma^{k}\bV{n+2}\to V_{n-k}} (after identifying 
\w{\cftop{\vrp}} with the cone point \w[).]{c_{\vrp}} 
We define the \wwb{k+1}\emph{st order higher homotopy operation}
obstruction to realizing $\Lambda$ to be the subset:
$$
\llrr{\Psi^{n+2}_{n-k}}~\subseteq~
\left[\bigvee_{\vrp\in \Psi^{n+2}_{n-k}}~\Sigma^{k}\bV{n+2},~V_{n-k}\right]~,
$$
\noindent consisting of all homotopy classes:
\begin{myeq}\label{eqahhob}
\lra{f^{n+2}_{n-k}}~:=~
\left[\bigvee_{\vrp\in \Psi^{n+2}_{n-k}}~\tilde{f}_{\vrp}\right]~\in~
\left[\bigvee_{\vrp\in \Psi^{n+2}_{n-k}}~\Sigma^{k}\bV{n+2},~V_{n-k}\right]
\end{myeq}
\noindent obtained by varying the inductively defined choice of
\w[.]{\hX{k}} Each such class \wref{eqahhob} is thus  a
\emph{value}, in the sense of  \S \ref{shho}, of the \wwb{k+1}st
order higher homotopy operation \w[.]{\llrr{\Psi^{n+2}_{n-k}}}
\end{defn}

\begin{thm}\label{thhoa}
Under the assumptions of \S \ref{sudks}, the homotopy class
\w{\lra{f^{n+2}_{n-k}}} of \wref{eqahhob} vanishes if and only if the
restriction of the lifting \w{\hX{k}} in Diagram \wref{eqdksspec}
to \w{\QQ[\bnpp,\bnmk]} exists.
\end{thm}

\begin{proof}
This follows by induction from Proposition \ref{pnullextend}.
\end{proof}

%
%      Corollary:  HHOs as obstructions to rectification
%
\begin{cor}\label{chhoa}
The last ($n$-th order) higher homotopy  operation
\w{\llrr{\Psi^{n+2}_{0}}} is the (final) obstruction to extending the
\wwb{n+1}truncated simplicial object \w{\qVd{n+1}}  to a
rectification of \w[,]{\tVd{n+2}} and thus to a realization of
\w[.]{\tau_{n+2}\Gd}
\end{cor}

\begin{proof}
The induction of \S \ref{sipq} (and Proposition \ref{pnullextend}) is 
different \emph{prima facie} from that of \S \ref{sipr}, since
we enlarge the indexing categories \w{\hD{n+2}\rest{[\bnpp,\bj]}} 
at each stage. However, this is no longer true at the last stage, 
when \w[,]{k=n} so our last obstruction is for the full extension to
\w[,]{\MOp} as in \S \ref{sipr}.
\end{proof}

%
%c6   Minimal higher homotopy operations
%
\sect{Minimal higher homotopy operations}
\label{cmhho}

We would like to relate the higher homotopy operation obstructions of
Section \ref{chhoeo} to the cohomological obstructions of Section 
\ref{caqco}. Evidently, these two obstructions do not take values in the same
groups, so they can not be identified \emph{per se}. In order to compare
them, we define a homomorphism between the target groups, as follows:

\begin{mysubsection}{The correspondence homomorphism}
\label{sch}
By adjointness, there is a natural isomorphism
\w{[\Sigma^{n}\bV{n+2},X] \cong [\bV{n+2},\Omega^{n}X]}
and by \w{\bV{n+2} \in \MA}, we have a natural isomorphism
\w[.]{[\bV{n+2},Y] \cong \Hom_{\PAAlg}(\piA\bV{n+2},\piA Y)}
However, \w{\Omega^{n}\piA X := \piA \Omega^{n} X} so the combination gives
a natural isomorphism
$$
[\Sigma^{n}\bV{n+2},V_{0}]\cong\Hom_{\PAAlg}(\piA\bV{n+2},\Omega^{n}\piA
V_{0})~.
$$
\noindent Next, post-composition with the looped augmentation map
\w{\epsilon:G_{0}=\piA V_{0} \to \Lambda} induces a (surjective) homomorphism
$$
\Hom_{\PAAlg}(\bG{n+2},\Omega^{n}\piA V_{0})\to
\Hom_{\PAAlg}(\bG{n+2},\Omega^{n}\Lambda)~.
$$
\noindent If we identify \w{\bG{n+2}} with \w[,]{\piA\bV{n+2}} we get
a homomorphism
$$
[\Sigma^{n}\bV{n+2},V_{0}]\to\Hom_{\PAAlg}(\bG{n+2},\Omega^{n}\Lambda)~.
$$
\noindent Finally, by Corollary \ref{cco} there is a homomorphism
$$
\Hom_{\PAAlg}(\bG{n+2},\Omega^{n}\Lambda)\to
\HAQ{n+2}(\Lambda;\Omega^{n}\Lambda)
$$
\noindent as well. Combining these maps and identifications yields:
\begin{myeq}\label{eqcorhom}
\Phi:[\Sigma^{n}\bV{n+2},V_{0}]~\to~\HAQ{n+2}(\Lambda;\Omega^{n}\Lambda)
\end{myeq}
\end{mysubsection}

\begin{defn}\label{chomom}
In the setting of \S \ref{assumptions}, the $n$-th
\emph{correspondence homomorphism} is the map
$$
\wPh{n}:\bigoplus_{\vrp\in\Psi^{n+2}_0}\left[\Sigma^{n}\bV{n+2},V_{0}\right]
~\longrightarrow~\HAQ{n+2}(\Lambda;\Omega^{n}\Lambda)
$$
\noindent obtained by adding up the homomorphisms $\Phi$ of
\wref[,]{eqcorhom} whose target is an abelian group.
\end{defn}

The correspondence homomorphism is hard to evaluate, in
general. However, there is a special class of values of the higher
homotopy operation \w{\llrr{\Psi^{n+2}_{0}}} for which this evaluation
is possible:

\begin{defn}\label{dmhho}
A value
\w{\lra{f}\in[\bigvee_{\vrp\in\Psi^{n+2}_{j}}\,\Sigma^{n-j}\bV{n+2},\,V_{j}]}
of the higher homotopy  operation \w{\llrr{\Psi^{n+2}_{j}}} defined
in \S \ref{dhhoa} is called \emph{minimal} if it is represented by a map
\w[,]{f=\bigvee_{\vrp\in\Psi^{n+2}_{j}}~\tilde{f}_{\vrp}} as in
\wref[,]{eqahhob} for which \w{\tilde{f}_{\vrp}} is a constant map
(that is, takes values in degenerate $0$-simplices) for
all but the particular flag \w{\vrp_{j}:=(0<1<2<\dots<n-j)} in
\w{\Psi^{n+2}_{j}} (corresponding to \w[),]{d_0d_0\dots\bdz{}} and the map
\w{f_{\vrp_{j}}:\partial \cF{\vrp_{0}}\to\map_{\M}(\bV{n+2},V_{j})} is
constant on all but the basic atomic $k$-simplex
\w{\tau_{k}=|d_{0}|\dots|d_{0}|} of \w{\cF{\vrp_{j}}} (\S \ref{ssq}) 
for each \w[.]{1\leq k\leq n-j} 
\end{defn}

\begin{remark}\label{rmhho}
More generally, we could replace \w{\vrp_{j}} by another map
\w[,]{\vrp'\in\Psi^{n+2}_{j}} and simply require that \w{f_{\vrp'}} be
constant on all but one $k$-simplex \w{\sigma_{k}} of \w{\cF{\vrp'}}
for each \w[.]{1\leq k\leq n-j}  Note that for minimal cases,
\w{\lra{f}\in[\bigvee_{\vrp\in\Psi^{n+2}_{j}}\,\Sigma^{n-j}\bV{n+2},\,V_{j}]} 
is completely determined by the homotopy class of
\w{\tilde{f}\rest{\partial\sigma_{j}}} corresponding to a single map
\w[.]{\Sigma^{n-j}\bV{n+2} \to V_{j}} 
\end{remark}

\begin{defn}\label{deqladder}
A ladder diagram from 
\w{\gamma_{n}=d_{0}\bdz{\bV{n+2}}:\bV{n+2}\to \mZ{n} \qVd{n+1}} to 
\w{\gamma_{j}:\Sigma^{n-j}\bV{n+2}\to\mZ{j}\qVd{n+1}}
(cf.\ \S \ref{sladders}) is said to be \emph{equivalent} to 
a ladder diagram from the same \w{\gamma_{n}} to 
\w{\gamma'_{j}:\Sigma^{n-j}\bV{n+2}\to\mZ{j}\qVd{n+1}} 
if \w[.]{\gamma_{j}\sim\gamma'_{j}}
\end{defn}

%
%       Proposition:  ladder diagrams and minimal values
%
\begin{prop}\label{pladmin}
There is a bijection between equivalence classes of ladder diagrams from 
\w{\gamma_{n}=d_{0}\bdz{\bV{n+2}}:\bV{n+2} \to \mZ{n} \qVd{n+1}} to 
\w{\gamma_{j}:\Sigma^{n-j}\bV{n+2} \to\mZ{j}\qVd{n+1}}
and minimal values of the higher homotopy
operation \w[.]{\lra{f_{\gamma_{j}}}\in \llrr{\Psi^{n+2}_{j}}}
Moreover, if \w{\gamma_j \sim 0} then \w{\lra{f_{\gamma_{j}}}}
vanishes (and conversely for the appropriate minimal value).
\end{prop}

\begin{proof}
Given such a ladder diagram, we inductively define suitable
simplicial functors \w{\hX{k}:\sk{k}\QQ[\bnpp,\bnmk] \to \MOp}
extending the given \w[,]{X:\tD{n+1} \to \MO} in the notation of
\wref[.]{eqdksspec}
To do so, we only need to specify \w{\hX{k}} on \w[,]{\tau_{k}} with all
other simplices of \w{\QQ[\bnpp,\bnmk] \setminus \Fs\tD{n+1}}
sent to zero.  The compatibility condition reduces to
\w{\parz\hX{k}(\tau_{k})=d_{0}^{V_{n-k+2}}\hX{k-1}(\tau_{k-1})} and 
\w{\partial_{i}\hX{k}(\tau_{k})=0} for \w[.]{i>0}

Given the map \w{g_{i}:W \to \mC{i}\Yd} in a ladder diagram, let
\w[,]{g'_{i}:=j_{i}'\circ g_{i}} where \w{j_{i}':\mC{i}\Yd\hra Y_{i}} 
is the inclusion.
	
To begin, we define \w{\hX{0}:\sk{0}\QQ[\bnpp,\bn] \to \MOp} by mapping 
\w{\parz\tau_{1}=d_{0}|} to \w[.]{g'_{n}\in \MOp(\bV{n+2},V_{n})_{0}}
Recall that \w{\left(\Delta[1]\otimes\bV{n+2}\right)/
\left(\Lambda^{1}_{0}\otimes\bV{n+2}\right)}
provides a model for the cone \w[,]{\cone{\bV{n+2}}} where
\w{\Lambda^{n}_{k}\subseteq\partial\Delta[n]} is the horn omitting the
$k$-th face. Hence, the map in the ladder diagram
\w{H_{n}:\cone{\bV{n+2}}\to \mC{n}\Vd \subset V_{n}} defines a 
map \w{\widehat{H_{n}}:\Delta[1] \otimes\bV{n+2}\to V_{n}} whose
restriction to \w{\Lambda^{1}_{0}\otimes\bV{n+2}} is zero, and whose
restriction to \w{\partial_{0}\Delta[1]\otimes\bV{n+2}} is \w[.]{g'_{n}} 
The adjoint \w{\widetilde{H_{n}}:\Delta[1]\to\MOp(\bV{n+2},V_{n})}
of \w{\widehat{H_{n}}} restricts to the zero map on the horn 
\w[.]{\Lambda^{1}_{0}}  We can therefore extend \w{\hX{0}} to
a simplicial functor  \w{\sk{1}\QQ[\bnpp,\bn]\to\MOp} by sending \w{\tau_{1}} to
\w[,]{\widetilde{H_{n}} \in \MOp(\bV{n+2},V_{n})_{1}} with
\w{d_{1}\widetilde{H_{n}}=0} and \w{d_{0} \widetilde{H_{n}}=\tilde{g}'_{n}} 
by construction. Since \w[,]{d_{0}^{V_{n}} g'_{n}=0} \w{H_{n}} induces
a map  
$$
g'_{n-1}:\Sigma\bV{n+2}\cong\left(\Delta[1]\otimes\bV{n+2}\right)/
\left(\partial\Delta[1]\otimes\bV{n+2}\right)\to \mC{n-1}\Vd \subset V_{n-1}~,
$$
\noindent and we define \w{\hX{1}:\sk{1}\QQ[\bnpp,\mathbf{n-1}]\to\MOp}
extending the previous choices by sending \w{\parz \tau_{2}} to 
\w[.]{\tilde{g}'_{n-1} \in \MOp(\bV{n+2},V_{n-1})_{1}}
	
At the $k$-th stage, assume we have defined
\w{\hX{k}:\sk{k}\QQ[\bnpp,\bnmk]\to\MOp} sending \w{\parz\tau_{k+1}} 
to \w[.]{\tilde{g}'_{n-k}\in\MOp(\bV{n+2},V_{n-k})_{k}} 
Note that \w{\left(\Delta[k+1]\otimes\bV{n+2}\right)/
\left(\Lambda^{k+1}_{0}\otimes\bV{n+2}\right)} is a model for 
\w[,]{\cone{\Sigma^{k}\bV{n+2}}} so \w{H_{n-k}} defines a map 
\w{\widetilde{H_{n-k}}:\Delta[k+1]\to\MOp(\bV{n+2},V_{n-k})} 
whose restriction to \w{\Lambda^{k+1}_{0}} is zero.  Viewed as a
\wwb{k+1}simplex in the mapping space, this means that 
\w{d_{i}\widetilde{H_{n-k}}=0} for \w[,]{i>0} while 
\w[.]{d_{0} \widetilde{H_{n-k}}=\tilde g'_{n-k}}
We may therefore extend \w{\hX{k}:\sk{k}\QQ[n+2,n-k] \to \MOp} to the
\wwb{k+1}skeleton of \w{\QQ[n+2,n-k]} by mapping \w{\tau_{k+1}} to
\w[.]{\widetilde{H_{n-k}}} Since \w[,]{d_{0}^{V_{n-k}} g'_{n-k}=0} 
\w{H_{n-k}} induces a map 
$$
g'_{n-k-1}:\Sigma^{n-k}\bV{n+2}\cong\left(\Delta[k+1]\otimes\bV{n+2}\right)/
\left(\partial\Delta[k+1]\otimes\bV{n+2}\right)
\to \mC{n-k-1}\Vd \subset V_{n-k-1}~.
$$
\noindent We define \w{\hX{k+1}:\sk{k+1}\QQ[\bnpp,\mathbf{n-k-1}]\to\MOp} 
extending the previous choices by sending \w{\parz \tau_{k+2}} to 
\w[.]{\tilde{g}'_{n-k-1} \in \MOp(\bV{n+2},V_{n-k-1})_{k+1}}

Conversely, given \w{\hX{j}:\sk{j}\QQ[\bnpp,\mathbf{n-j}] \to \MOp}
representing a minimal value of
\w[,]{[\bigvee_{\vrp\in\Psi^{n+2}_{j}}\,\Sigma^{n-j}\bV{n+2},\,V_{j}]} 
we define the corresponding ladder diagram by setting
\w[,]{\widetilde{H_{m}}=\hX{j}(\tau_{n-m+1})} and
\w[.]{\tilde{g}'_{m}=\hX{j}(\parz\tau_{n-m+1})} Note that the adjoint
\w{H_{m}} factors through \w[,]{\mC{m}\Vd \subset V_m} since
\w{\partial_{i}\tau_{n-m+1}=0} for \w[,]{i >0} while the adjoint
\w{g'_{m}} factors through \w{\mZ{m}\Vd\subset V_{m}} since 
\w{\partial_{i}\parz \tau_{n-m+1}=0} for all $i$.
\end{proof}

\begin{cor}\label{cladmin}
Given a ladder diagram from \w{\gamma_{n}} down to \w[,]{\gamma_{j}} as in 
Proposition \ref{pladmin}, the segment of the simplicial diagram from
dimension \w{\bnpp} down to $\bj$ is $\infty$-homotopy commutative, so
it can be rectified.
\end{cor}

Combining Proposition \ref{pladmin} with Proposition \ref{pkinv}
yields:

%
%       Corollary: minimal values exist
%
\begin{cor}\label{calwaysmin}
Under the assumptions of \ref{assumptions}, the last higher operation
\w{\llrr{\Psi^{n+2}_{0}}} has a minimal value.  As a consequence, the
operation is non-empty (well-defined) and vanishes if any minimal
value vanishes. 
\end{cor}

\begin{remark}\label{longtoda}
Note that minimal values of the higher operation are values of long
Toda brackets of the form:
$$
\xymatrix@R=15pt{
&&&&&&\\
\bV{n+2} \ar[r]^<<<<<{\bdz{}} \ar@/^{3.3pc}/[rr]^{0} &
C_{n+1}\Vd \ar[r]^<<<<{d^{n+1}_{0}} \ar@{=>}[u]_{H}
\ar@/_{3.3pc}/[rr]_{0} & C_{n}\Vd \ar@{=>}[d]_{0} \ar[r]^<<<<{d^{n}_{0}} &
C_{n-1}\Vd  \ar[r]^<<<<<{d^{n-1}_{0}} & \dotsc \ar[r]^{d^{2}_{0}} & 
C_{1}\Vd \ar[r]^{d^{1}_{0}} & V_{0}~, \\
&&&&&&
}
$$
\noindent with linear indexing category, in which all but the first
composite is strictly $0$.
\end{remark}

%
%      Theorem: HH and AQ Obstructions correspond
%
\begin{thm}\label{thhaq}
In the situation of \S \ref{assumptions}, the correspondence homomorphism
\w{\wPh{n}} maps a minimal value of \w{\llrr{\Psi^{n+2}_0}} to the
Andr\'{e}-Quillen obstruction \w{\beta_{n}} to realizing
$\Lambda$.
\end{thm}

\begin{proof}
This follows from Propositions \ref{pkinv} and \ref{pladmin} with \w[.]{j=0}
\end{proof}

%
%      Corollary:  correspondence of the two systems of obstructions
%
\begin{cor}\label{cvanish}
The vanishing of any minimal value of \w{\llrr{\Psi^{n+2}_0}} implies the
vanishing of \w[.]{\beta_{n}}
Conversely, given an \wwb{n+1}\sps\ \w{\qWd{n+1}} for $\Lambda$,
as in \S \ref{sceo}, there is an \wwb{n+1}truncated
simplicial object \w{\qVd{n+1}} realizing \w[,]{\tau_{n+1}\Gd}
and if the cohomology obstruction \w{\beta_{n+1}} associated to
\w{\qWd{n+1}} vanishes, so does \w{\llrr{\Psi^{n+2}_0}} for \w[.]{\qVd{n+1}}
\end{cor}

Note the (necessary) shift in indexing of the two series of obstructions,
because of the different things they measure: the $k$-invariant \w{\beta_{n}}
is the obstruction \eqref{eqaqcohcl} to obtaining an $n$-th semi-Postnikov section for
$\Lambda$, extending a given \wwb{n-1}st semi-Postnikov section, 
while the higher homotopy operation \w{\llrr{\Psi^{n+2}_0}}
is the obstruction (Corollary \ref{chhoa}) to constructing
\w{\qVd{n+2}} realizing the \wwb{n+2}truncation of a given resolution
\w{\Gd} for $\Lambda$.

%
%c7   Difference obstructions
%
\sect{Difference obstructions}
\label{cdo}

The next question arising in the inductive procedure for realizing a
\PAa $\Lambda$ described in \S \ref{assumptions}  is that 
of distinguishing between different extensions of a given
\wwb{n-1}\sps\ \w{\qWd{n}} to an $n$-\sps\ \w[.]{\qWd{n+1}} 

Recall that if $\Lambda$ is realizable, then any \w{X\in\M} with
\w{\piA X\cong\Lambda} has a free resolution \w{\Wd\xra{\simeq}\co{X}}
in the \ww{E^{2}}-model category \w{s\M} 
(by \cite{StoV}), with \w{J\Wd} $\A$-equivalent to $X$ (\S
\ref{setmc}), and \w{\piA\Wd} a free \PAa resolution of $\Lambda$.
Since $J$ preserves weak equivalences, classifying realizations (in
\w[)]{s\M} of free simplicial resolutions of $\Lambda$ subsumes (and
refines) the classification of all realizations of $\Lambda$ (in $\M$)
up to $\A$-equivalence. 
However, every free simplicial \PAa \w{\Gd} has a CW basis,
and every CW resolution \w{\Gd\to\Lambda} can be realized as a
resolution \w{\Wd\to X} in \w[.]{s\M} Thus we can in fact apply the
inductive procedure described in \S \ref{assumptions}, starting with a
specific CW basis for \w[.]{\Gd}

Once more, we have two methods of constructing the difference
obstructions for inductively distinguishing between realizations of
such a CW resolution \w[:]{\Gd\to\Lambda} in terms of
Andr\'{e}-Quillen cohomology classes, and in terms of higher homotopy
operations.

\begin{mysubsection}{Cohomology difference obstructions}\label{scdo}
Again, there are a number of equivalent descriptions of the difference 
obstructions in cohomology, and we give one based on \cite{BDGoeR},
but stated in terms of a fixed simplicial \PAa resolution
\w{\Gd\to\Lambda} with given CW basis \w[:]{(\bG{i})_{i=0}^{\infty}}

Assume as in \S \ref{assumptions} that we have chosen a realization 
\w{\qVd{n+1}} for \w[,]{\tau_{n+1}\Gd} which has two different
extensions \w{\qVd{a}} and \w[,]{\qVd{b}} both realizing
\w[.]{\tau_{n+2}\Gd} Since we are not concerned now with the existence
problem, we may assume that \w{\qVd{a}} and \w{\qVd{b}} are
\wwb{n+2}coskeletal objects (that is, \ww{\Po{n+1}}-simplicial objects)
in \w[,]{s\M}  with
\w{\Wd:=\csk{n+1}\qVd{a}=\csk{n+1}\qVd{b}=\csk{n+1}\qVd{n+1}} 
as their common $n$-th Postnikov section. In particular, they are
both $n$-\qps s.

The question is whether \w{\qVd{a}} and \w{\qVd{b}} are weakly
equivalent (relative to \w[),]{\Wd} that is, whether there is a map
$\vrp$ fitting into a commuting diagram of vertical fibration
sequences in \w[,]{s\M/B\Lambda} with horizontal weak equivalences: 
\mydiagram[\label{eqpostdiag}]{
\bE{\Omega^{n+1}\Lambda}{n+1} \ar[d] \ar[rr]^{\vrp_{\ast}} && 
\bE{\Omega^{n+1}\Lambda}{n+1} \ar[d] \\
\qVd{a} \ar[d]_{p\q{a}_{n+1}} \ar[rr]^{\vrp} &&
\qVd{b} \ar[d]_{p\q{b}_{n+1}} \\
\Wd \ar[d]_{k\q{a}_{n}} \ar[rr]^{=} && \Wd \ar[d]_{k\q{b}_{n}} \\
\bE{\Omega^{n+1}\Lambda}{n+2} \ar[rr]^{\bar{W}\vrp_{\ast}} && 
\bE{\Omega^{n+1}\Lambda}{n+2}
}
\noindent where \w{p\q{t}_{n+1}:\qVd{t}\to\Wd} are the structure maps
in the Postnikov towers, and  
\w{k\q{t}_{n}:\Wd=\Po{n}\qVd{t}\to\bE{\Omega^{n+1}\Lambda}{n+2}}
are the (functorial) $k$-invariants for \w{\qVd{t}} \wb[.]{t=a,b}

By \cite[Prop.~8.7]{BDGoeR} (or \cite[Prop.~5.3]{BJTurR}), for any
$\Lambda$-module $\K$ there is a natural isomorphism 
$$
\left[\Wd,\,E_{\Lambda}^{s\M}(\K,n)\right]_{s\M/B\Lambda}~\to~
\left[\piA\Wd,\tEL{\K}{n}\right]_{s\PAAlg/\Lambda}~\cong~
\HAQ{n}(\piA\Wd/\Lambda;\K)
$$
\noindent for \w[.]{n\geq 2}  

Therefore, the $k$-invariants \w{k\q{t}_{n}} are determined by the
induced maps of simplicial \PAa[s]  
\w{(k\q{t}_{n})_{\#}:\piA\Wd\to\tEL{\Omega^{n+1}\Lambda}{n+2}}
(where both source and target are Eilenberg-Mac~Lane objects in
\w[,]{s\PAAlg} by \S \ref{sceo}). Since the vertical fibration
sequences of \wref{eqpostdiag} induce long exact sequences in \w{\piA}
by \cite[Lemma 5.11]{BJTurR}, we see 
that \w{(k\q{t}_{n})_{\#}} is an isomorphism. Thus the choice of the
$k$-invariants is determined solely by the Eilenberg-Mac~Lane
structure on \w[,]{\piA\Wd} which is determined in turn by the choice
of sections \w{s\q{a},s\q{b}:\tPo{0}\piA\Wd\simeq\tBL\to\piA\Wd}
(where \w{\tPo{i}} is the $i$-th Postnikov section in \w[).]{s\PAAlg}

Thus the \emph{cohomology difference obstruction} for \w{\qVd{a}} and
\w{\qVd{b}} (relative to \w[)]{\Po{n}\qVd{a}=\Po{n}\qVd{b}=\Wd} is
defined to be
\begin{myeq}\label{eqaqcdo}
\delta_{n}~:=~[s\q{a}]-[s\q{b}]~\in~\HAQ{n+2}(\Lambda,\Omega^{n+1}\Lambda)~.
\end{myeq}
\noindent See \cite[\S 8 and Proposition 9.12]{BDGoeR} and 
\cite[Theorem~5.7(c)]{BJTurR}.
\end{mysubsection}

%
%     Proposition:  description of the difference obstruction
%
\begin{prop}\label{pdiffobst}
Under the assumptions of \ref{assumptions} for \w[,]{n \geq 1} with
the two attaching maps
\w{\bdz{t}:\bV{n+2}\to\mZ{n+1}\qVd{t}=\mZ{n+1}\Wd} \wb[,]{t=a,b}
the obstruction class \w{\delta_{n}} is represented
in the sense of Corollary \ref{cco} by the map 
\w{\bG{n+2}\to\pinat{n+1}\qVd{t}} induced by 
\w[.]{\odel:=\bdz{a}\cdot(\bdz{b})^{-1}:\bV{n+2}\to\mZ{n+1}\Wd} 
\end{prop}	

\begin{remark}\label{rdiffobst}
Note that even though \w{\Wd=\Po{n}\qVd{t}} is only
\wwb{n+1}coskeletal, it is an $n$-\qps\ for $\Lambda$, and in either
extension we have \w[.]{\mZ{n+1}\qVd{t}=\mZ{n+1}\Wd} Thus we have a
canonical identification \w[,]{\pinat{n+1}\qVd{a}\cong\pinat{n+1}\qVd{b}}
induced by the identity on \w{\mZ{n+1}} (cf.\ \wref[).]{eqhurewicz}
Thus the difference \w{\bdz{a}\cdot(\bdz{b})^{-1}} makes sense; 
it is defined using the the homotopy cogroup structure on \w[.]{\bV{n+2}}
\end{remark}	
	
\begin{proof}
Note that \w{\Gd} is a cofibrant model for \w[,]{\tBL} while the
sections \w{s\q{t}} evidently factor though
\w{\tPo{n+1}\piA\qVd{t}\simeq\tPo{n+1}\Gd} \wb[,]{t=a,b} and are thus
induced in \w{\piA} by the Postnikov structure maps
\w{p\q{t}_{n+1}:\qVd{t}\to\Wd} in \w[.]{s\M} The cohomology class
\w{\delta_{n}} is thus represented by the difference 
\w{\left(p\q{a}_{n+1}\right)_{\#}-\left(p\q{b}_{n+1}\right)_{\#}} (now taking values in
abelian groups) mapping \w{\tPo{n+1}\Gd} to \w[.]{\piA\Wd} 
Since by assumption the maps \w{p\q{t}_{n+1}} are
the identity through simplicial dimension \w[,]{n+1} and their source
is \wwb{n+2}coskeletal, the map $\odel$ is determined by what each
map \w{p\q{t}_{n+1}} does to \wwb{n+2}simplices. As in the proof
of Proposition \ref{pexistobst}, \w{\left(p\q{t}_{n+1}\right)_{\#}} is determined
by its value on the tautological class
\w[,]{\iota_{n+2}\in G_{n+2}\lin{\bV{n+2}}} which maps to the 
class represented by the matching collection 
\w[,]{\left(\bdz{\bV{n+2}},0,\dotsc,0\right)\in\Hom_{\M}(\bV{n+2},M_{n+2}\qVd{t})} 
corresponding to \w[.]{\bdz{\bV{n+2}}:\bV{n+2}\to\mZ{n+1}\qVd{t}=\mZ{n+1}\Wd} 
\end{proof}

\begin{mysubsection}{The representing cocycles}\label{srcc}
The cohomology class \w{\delta_{n}} evidently vanishes if the cocycle 
representing it does \wh that is, if the attaching maps 
\w{\bdz{t}:\bV{n+2}\to\mZ{n+1}\Wd} for \w{\qVd{t}} \wb{t=a,b} are
homotopic \wh though the contrary need not hold.
Note that if we post-compose the maps \w{\bdz{t}} with the inclusion 
\w[,]{j_{n+1}:\mZ{n+1}\Wd\hra\mC{n+1}\Wd} the resulting maps
\w{j_{n+1}\circ\bdz{t}:\bV{n+2}\to\mC{n+1}\Wd} both represent 
\w[,]{\bdz{G_{n+2}}:\bG{n+2}\to C_{n+1}\Gd} by \wref[,]{eqcommmoor} so
that we again face a situation similar to that of \S \ref{sladders},
where we want to lift a nullhomotopy for
\w{j_{n+1}\circ\odel:\bV{n+2}\to\mC{n+1}\Wd} to a nullhomotopy for $\odel$:
$$
\xymatrix@R=15pt{
\bV{n+2} \ar[d]_{\odel} \ar[rr]^{i} && 
\cone{\bV{n+2}} \ar@{.>}[lld]^{H} \ar[d]^{H'} \\
\mZ{n+1}\Wd \ar[rr]_{j_{n+1}} && \mC{n+1}\Wd ~.
}
$$
\noindent If we can do so, \w{\qVd{a}} and \w{\qVd{b}} are weakly equivalent
(relative to \w[).]{\Wd}

We can therefore compare the Andr\'{e}-Quillen difference obstruction
of \cite{BDGoeR} and \cite{BJTurR} (described above) with the
construction of \cite{BlaAI}, as follows:

In the notation of \cite[\S 4]{BlaAI}, for each \w{t=a,b} the attaching map 
\w{\bdz{t}:\bV{n+2}\to\mZ{n+1}\qVd{t}=\mZ{n+1}\Wd} is determined by a map
\w[,]{\lambda:\bG{n+2}\cong\piA\bV{n+2}\to\piA\mZ{n+1}\qVd{t}} fitting into the 
commuting diagram:
\mydiagram[\label{eqliftdo}]{
\piA\bV{n+2} \ar@{-->}[rrr]^{\lambda} \ar[d]_<<<<{\cong}  & & & 
\piA\mZ{n+1}\Vd \ar@{->>}[d]^{(j_{n+1})_{\#}} \\
\bG{n+2} \ar[rr]_{\bdz{\bG{n+2}}} & & \mZ{n+1}\Gd \ar[r]_<<<<{\cong} &
\mZ{n+1}(\piA \qVd{t})
}
\noindent The map \w{(j_{n+1})_{\#}} is surjective by the commutative
diagram before \cite[\S 4.12]{BlaCH}, which also shows that there is a
short exact sequence:
$$
\xymatrix@R=15pt{
0~\to~\Omega\pinat{n}\qVd{t}~\hra~\piA\mZ{n+1}\qVd{t}~\xra{(j_{n+1})_{\#}}~
Z_{n+1}(\piA\qVd{t})~\to~0~.
}
$$
\noindent Thus the choices for the lift $\lambda$, and thus the
difference $\odel$ between \w{\bdz{a}} and \w[,]{\bdz{b}} are in fact
parametrized by \w[.]{\Omega^{n+1}\Lambda}  

We have thus shown that the cohomology obstructions of
\cite[Theorem~4.18]{BlaAI}, described there in terms of the lift
$\lambda$, may be identified with \w{\delta_{n}} of \wref[.]{eqaqcdo}
\end{mysubsection}

\begin{mysubsection}{The difference higher homotopy operation}\label{sdhho}
The problem of lifting the nullhomotopy in \wref{eqliftdo} can be
stated in terms of a ladder diagram as in \S \ref{sladders}, namely:
$$
\xymatrix@R=15pt{
\bV{n+2} \ar[d]^{\odel} \ar[dr]^{g_{m}} \ar[r] & 
\cone{\bV{n+2}} \ar[d]^{H_{m}} \ar[r] &
\Sigma\bV{n+2} \ar[d]^{\gamma_{m-1}} \ar[r] & \dotsb\\
\mZ{n+1}\Wd \ar[r]_{j_{m}} & \mC{n+1}\Wd \ar[r]_{d_{0}} & 
\mZ{n}\Wd \ar[r] & \dotsb~.
}
$$
\noindent This in turn can be recast as a special instance of 
rectifying a suitable homotopy commutative diagram 
(compare Proposition \ref{pladmin}).  

However, to avoid describing this diagram explicitly, we instead
define a certain \wwb{n+2}truncated simplicial object \w{\Yd} and an
attaching map \w{\bdz{\bY{n+3}}:\bY{n+3} \to \mZ{n+2}\Yd} which allows
us to use the definitions of Section \ref{chhoeo} verbatim, in the
setting of \S \ref{scdo}: 

We begin with the \wwb{n+1}truncation 
\w[,]{\tau_{n+1}\Yd:=\tau_{n+1}\Wd} and set 
\w{Y_{n+2}:=\mZ{n+1}\Wd} with \w{d_{0}^{n+2}} the inclusion
\w{i_{n+1}:\mZ{n+1}\Yd\hra Y_{n+1}} and \w{d_{i}^{n+2}=0} for \w[.]{i>0}  
This indeed constitutes an \wwb{n+2}truncated simplicial object, since
\w{d_{i}^{n+1}d_{j}^{n+2}=0} for any \w[.]{0\leq i,j\leq n+2} 
Now set \w{\bY{n+3}:=\bV{n+2}} and let \w[,]{\bdz{\bY{n+3}}=\odel} 
which lands in \w{\mZ{n+1}\qVd{n+1}=Y_{n+2}=\mC{n+2}\Yd} by
construction.  However, since \w{\mZ{n+2}\Yd=0} (since \w{d_{0}^{n+2}}
is monic) it follows that \w{\odel=0} if and only if it factors
through \w[.]{\mZ{n+2}\Yd}  Note that
\w{d^{n+2}_{0}\circ\bdz{\bY{n+3}}=i_{n+1}\circ\odel} is null-homotopic,
since \w{i_{n+1}\circ\bdz{a}} and \w{i_{n+1}\circ\bdz{b}} both realize 
\w[,]{\bdz{G_{n+2}}:\bG{n+2}\to C_{n+1}\Gd} by \wref[.]{eqcommmoor}

Thus the \wwb{n+2}truncated simplicial object \w[,]{\tau_{n+2}\Yd}
together with \w{\bY{n+3}} and \w[,]{\bdz{\bY{n+3}}} satisfies the
assumptions of \S \ref{assumptions}, and we may therefore apply the
constructions of Section \ref{chhoeo} to make the following:
\end{mysubsection}

\begin{defn}\label{shiftdiff}
In the setting of \S \ref{scdo}, the \emph{higher homotopy operation
difference obstruction} associated to \w{\bdz{a}} and \w{\bdz{b}}
is defined to be the subset of 
\w{[\Sigma^{n+1}\bY{n+3},Y_{0}]=[\Sigma^{n+1}\bV{n+2},V_{0}]}
associated as in \S \ref{dhhoa} to \w[.]{\Yd} We denote this subset
by \w[.]{\llrr{\bdz{a},\bdz{b}}}
\end{defn}

Note that the source is suspended once further than in the existence case.
Thus, the correspondence homomorphism yields a map 
\w[,]{\bG{n+2}\to \Omega^{n+1}\Lambda} which represents a cohomology class in
\w[,]{\HAQ{n+2}(\Lambda,\Omega^{n+1}\Lambda)} with coefficient
module determined by the target in Proposition \ref{pmoore}.

\begin{prop}\label{diffsuff}
In the setting of \S \ref{scdo}, the difference obstruction
\w{\llrr{\bdz{a},\bdz{b}}} always has a minimal value, so it is well
defined (and non-empty). It vanishes if and only if
\w[.]{\qVd{a}\simeq \qVd{b}}  
\end{prop}

\begin{proof}
Corollary \ref{calwaysmin} applies here, too, so
\w{\llrr{\bdz{a},\bdz{b}}} has a minimal value. Combining Proposition
\ref{pladmin} with Proposition \ref{pkinv} for \w[,]{\Yd} we see that
vanishing of some value implies that \w{\overline\delta} can be
chosen to factor through \w[.]{\mZ{n+2}\Yd=0}  Consequently
\w[,]{[\overline \delta]=0} so \w[.]{\bdz{a} \sim \bdz{b}} Hence
\w[.]{\qVd{a} \sim \qVd{b}} 
\end{proof}

The proof of Theorem \ref{thhaq} transfers word for word to our
setting to show:
%
%       Theorem:   correspondence of difference obstructions
%
\begin{thm}\label{tdiffobst}
The correspondence homomorphism \w{\wPh{n}(\Yd)}
maps a minimal value of \w{\llrr{\bdz{a},\bdz{b}}} to the
Andr\'{e}-Quillen difference obstruction between the two
realizations of $\Lambda$.
\end{thm}

\begin{remark}\label{rdiagsphere}
The rectification problem of \S \ref{sdhho} is somewhat
unsatisfactory, in that the truncated simplicial object \w{\Yd} which
we are trying to rectify does not consist solely of ``wedges of spheres''
(objects in \w[)]{\MA} in each simplicial dimension, so that 
(unlike the existence obstructions of \S \ref{dhhoa}) the higher
homotopy operation of \S \ref{shiftdiff} cannot be described in terms of 
\w[.]{\PAAlg} 

However, \w{s\M} is a pointed simplicial category in the sense of
\cite[II, \S 1]{QuiH}, and for any \w{A\in\M} we have an object
\w{S^{n}\wedge A:=(\bS{n}\otimes A)/(\{\pt\}\otimes A)} in \w{s\M}
(having $A$ in dimension $n$ and $\ast$ below), with a natural bijection between
\w{\map_{s\M}(S^{n+1}\wedge A,\,\Wd)} and \w[.]{\map_{\M}(A,\,\mZ{n+1}\Wd)}  
Therefore, \w{\odel\sim\ast} in Proposition \ref{pdiffobst} \wh or equivalently, 
\w{\bdz{a}\sim\bdz{b}:\bV{n+2}\to\mZ{n+1}\Wd} \wwh if and only if the diagram:
\mydiagram[\label{eqtwomaps}]{
S^{n+1}\wedge\bV{n+2} \ar@/^4mm/[rrr]^{\tbd{a}} \ar@/_4mm/[rrr]_{\tbd{b}} &&& \Wd
}
\noindent homotopy commutes (i.e., the corresponding maps \w{\tbd{a}} and
\w{\tbd{b}} are homotopic in \w[).]{s\M}  Thus the vanishing of the higher homotopy
operation \w{\llrr{\bdz{a},\bdz{b}}} is equivalent to rectifying the
diagram \wref[,]{eqtwomaps} whose entries are all in \w[.]{\MA} 
\end{remark}	

\begin{mysubsect}[\label{sgch}]{The geometric correspondence homomorphism}

In the setting of \S \ref{scdo}, assume that we have changed the two
maps \w{\tbd{a}} and \w{\tbd{b}} in \wref{eqtwomaps} into 
cofibrations, so that we have cofibration sequences:
$$
S^{n+1}\wedge\bV{n+2}~\xra{\tbd{t}}~\Wd~\to\sk{n+2}\qVd{t}
\hspace*{18mm} (t=a,b)
$$
\noindent in \w{s\M} (cf.\ \cite[I, \S 2]{QuiH}). 

If \w[,]{\tbd{a}\sim\tbd{b}} we have the solid
homotopy-commutative diagram with vertical weak equivalences:
\mydiagram[\label{eqcofiber}]{
S^{n+1}\wedge\bV{n+2}~\ar@{>->}[rr]^-{\tbd{a}} \ar[d]^{=} &&\Wd \ar[rr]^{i^{a}}
\ar[d]_{=} && \sk{n+2}\qVd{a} \ar@{.>}[d]_{\simeq}^{\phi} \\
S^{n+1}\wedge\bV{n+2}~\ar@{>->}[rr]^-{\tbd{b}} &&\Wd \ar[rr]^{i^{b}}
&& \sk{n+2}\qVd{b} ~,
}
\noindent inducing the dotted weak equivalence $\phi$.

In general, since homotopy colimits preserve weak equivalences and
commute with each other, applying \w{\hocolim} to the top row of 
\wref{eqcofiber} yields a horizontal homotopy cofibration sequence:
$$
\xymatrix@R=15pt{
\Sigma^{n+1}\bV{n+2}~\ar[rr]^-{\hocolim \tbd{a}} 
\ar[rrd]_{\hocolim(i^{b}\tbd{a})}  && 
\hocolim \Wd \ar[rr]^{\hocolim i^{a}}
\ar[d]^{\hocolim i^{b}} && \hocolim\sk{n+2}\qVd{a} \ar@{.>}[lld]^{\exists?g} \\
&&\hocolim \sk{n+2}\qVd{b}~, &&
}
$$
\noindent because \w{\hocolim(S^{n+1}\wedge\bV{n+2})\simeq\Sigma^{n+1}\bV{n+2}} 
by construction.

Note that  $g$ exists (making the diagram commute up to homotopy) 
if and only if the class 
\w{\Theta_{n+2}\q{b}:=\left[\hocolim(i^{b}\tbd{a})\right]} is zero. A
sufficient (but not necessary) condition for this to 
happen is that the left square in \wref{eqcofiber} commutes up to
homotopy (in which case \w{g:=\hocolim \phi} is a weak equivalence). 

Therefore, we can think of \w{\Theta_{n+2}\q{b}} in 
\w{\left[\Sigma^{n+1}\bV{n+2},\,\hocolim\sk{n+2}\qVd{b}\right]} 
as the obstruction to extending the weak equivalence 
\w{\sk{n+1}\qVd{a}\xra{=}\sk{n+1}\Wd\xra{=}\sk{n+1}\qVd{b}} to a weak
equivalence \w[.]{\sk{n+2}\qVd{a}\to\sk{n+2}\qVd{b}} 

Now assume that we have two full simplicial realizations \w{\qVd{a}} and
\w{\qVd{b}} of \w{\Gd\to\Lambda} in \w[,]{s\M} and that 
the two associated realizations \w{X\q{t}:=J\qVd{t}} of $\Lambda$ in
$\M$ (cf.\ \S \ref{setmc}) are given by \w[.]{\hocolim\qVd{t}}
In this case, we can think of the filtrations
\w{\F_{n}X\q{t}:=\hocolim\sk{n}\qVd{t}} \wb{n\geq 0} as successive
approximations to the objects \w[,]{X\q{t}} where each inclusion
\w{\sk{n}\qVd{t}\hra\qVd{t}} induces a map
\w[.]{j\q{t}_{n}:\F_{n}X\q{t}\to X\q{t}} Moreover, the resulting
sequence of obstructions are just traditional ``higher Toda brackets''
(cf.\ \cite{GWalkL}) appearing in trying to extend the inclusion of
\w{\hocolim\sk{0}\qVd{a}=W_{0}} into \w[.]{\hocolim\sk{0}\qVd{b}} 
There is also an ``inverse'' obstruction \w{\Theta_{n+2}\q{a}} in
\w[.]{\left[\Sigma^{n+1}\bV{n+2},\,\hocolim\sk{n+2}\qVd{a}\right]} 
Note that one vanishes if and only if the other does. 

Considering \w{\Theta_{n+2}\q{b}} as an element in 
\w[,]{(\Omega^{n+1}\piA\F_{n+2}X\q{b})\lin{\bV{n+2}}} 
and applying \w{(j\q{t}_{n})_{\#}} to \w{\Theta_{n+2}\q{b}} yields
a class in:
$$
(\Omega^{n+1}\piA X\q{b})\lin{\bV{n+2}}~=\ 
(\Omega^{n+1}\Lambda)\lin{\bV{n+2}}~
\cong~\Hom_{\PAAlg}(\bG{n+2},\,\Omega^{n+1}\Lambda)~.
$$
\noindent The procedure defined above should therefore be thought of
as a ``geometric version'' of the correspondence homomorphism
\w{\wPh{n}(\Yd)} of \S \ref{chomom} applied to a minimal value of
\w[:]{\llrr{\bdz{a},\bdz{b}}}  

To see this, we specialize to the case where \w{\M=\Sa} and 
\w[.]{\A=\{S^{1}\}} In this case the simplicial sets
\w{\F_{n}X\q{t}:=\hocolim\sk{n}\qVd{t}} are indeed successive 
approximations to \w[,]{X\q{t}} since from \wref{eqnqpia} and the
  Bousfield-Friedlander spectral sequence (cf.\ \cite[Theorem B.5]{BFrieH}), 
we see that \w{\pis\F_{n}X\q{t}} agrees with $\Lambda$ through
dimension $n$ (and in fact much more is true: see \cite[\S 10]{BBlaC}).

Moreover, we can use the description of the $E^{n+1}$-term of the
spectral sequence for \w{\sk{n+1}\qVd{b}} in \cite[\S 3]{BBlaS}
(combined with \cite[\S 8.4]{DKStB}) to see that the minimal value for
the higher homotopy operation \w{\llrr{\bdz{a},\bdz{b}}} is
\emph{defined} precisely when the differential \w{d^{n+1}} vanishes on
the element in \w{E^{n+1}_{n+2,\ast}} represented by
\w[.]{\odel\in(\piA\mC{n+1})\lin{\bV{n+2}}=E^{1}_{n+2,\ast}} 
In this case \w{[\odel]} lifts to \w[,]{\mC{0}\Wd=W_{0}} and
post-composing with the structure map
\w{W_{0}\to\hocolim\sk{n+1}\qVd{b}} yields (one value for)  
\w{\Theta_{n+2}\q{b}} as constructed above.
\end{mysubsect}

%
%c8  Rational homotopy theory section
%
\sect{An application to rational homotopy}
\label{crht}
Let \w{\M=\LQ} be the category of reduced differential graded Lie
algebras (DGLs) over $\bQ$ \wh a model for the rational homotopy
theory of simply-connected pointed spaces (cf.\ \cite[\S 4]{QuiR}).
If we let our collection $\A$ of homotopy cogroup objects consist of
the rational DGL $2$-sphere \w{S^{2}_{\bQ}} in \w[,]{\LQ} we see that
a \PAa is just a connected graded Lie algebra over $\bQ$. Note the
shift in dimension, due to the fact that we use \w{\pis(\Omega X;\bQ)}
as the homotopy \PAa of \w[,]{X_{\bQ}} so we have Samelson products,
which respect the grading of \w[,]{\pis(\Omega X;\bQ)} rather than
Whitehead products. 

By Hilton's theorem (cf.\ \cite[Theorem A]{HilH}), if $W$ is a wedge of
rational spheres of dimension \w[,]{\geq 2} 
\w{\Gamma:=\pis(\Omega W;\bQ)} is a free Lie algebra, so $\Gamma$ is
(intrinsically) \emph{coformal} \wh that is, $\Gamma$, equipped with
zero differentials, is itself a minimal DGL model for \w[.]{W_{\bQ}}
Moreover, no higher $\A$-homotopy operations exist in 
\w[,]{\pis(\Omega W;\bQ)} because no non-trivial rational homotopies
exist for maps between (fibrant and cofibrant) DGLs with zero differential.

Thus any free simplicial \PAa resolution \w{\Gd\to\Lambda}
of a rational \Pa $\Lambda$ is canonically realizable by a (strict)
simplicial DGL \w[,]{\Wd\in s\LQ} whose geometric realization
\w{\|\Wd\|} is the coformal realization of $\Lambda$ (unique up to
weak equivalence in $\M$). 

Since this \w{\Wd} is coformal in each simplicial dimension, all
(higher) homotopies used to define \emph{all} values of the
existence obstruction \w{\llrr{\Psi^{n+2}_0}} based on the
\wwb{n+1}truncation of \w[,]{\Wd} for any \w[,]{n\geq 1} necessarily
vanish. In particular, this will hold for any of the minimal values 
(which exist by Corollary \ref{calwaysmin}). In light of
Theorem \ref{thhaq}, this implies: 

%
%      Proposition: Obstructions to realiazability of rational Pi-algebras 
%
\begin{prop}\label{pratpa}
If \w[,]{\M=\LQ} \w[,]{\A:=\{S^{2}_{\bQ}\}} and
\w[,]{\Lambda\in\PAAlg} all the Andr\'{e}-Quillen obstructions
\w{\beta_{n}} to realizing $\Lambda$ vanish (for one branch of the
inductive procedure in \S \ref{sipr}).
\end{prop}

Of course, by Corollaries \ref{chhoa} and \ref{cvanish}, this implies in turn
Quillen's corollary to \cite[Theorem 1]{QuiR}, stating that any
simply-connected rational \Pa $\Lambda$ is realizable. 

\begin{remark}\label{rdegcof}
Note that if we replace \w{\Wd} by a weakly equivalent Reedy fibrant
simplicial object \w[,]{\Wd'} the differentials in each DGL \w{W_{i}}
need no longer vanish; so we cannot apply the above argument (for the
vanishing of the higher homotopy operations) to calculating
the cocycle representing the cohomology obstruction directly.
\end{remark}


\begin{thebibliography}{ABCDE}
%
\bibitem[Ad]{AdamsN}
J.F.~Adams,
"On the non-existence of elements of Hopf invariant one," \hsm
Annals of Math. \textbf{72} (1960), pp.~20-104.
%
\bibitem[An]{AndrM}
M.~Andr\'{e},
\textit{M{\'{e}}thode Simpliciale en Alg{\`{e}}bre Homologique et
Alg{\`{e}}bre Commutative},\hsm
Springer-\-Verlag \textit{Lec.\ Notes Math.} \textbf{32},
Berlin-\-New York, 1967.
%
\bibitem[Ba]{BauA}
H.-J. Baues,  
\textit{The algebra of secondary cohomology operations}, 
Prog.\ in Math.\ \textbf{247}, Birkh{\"{a}}user Verlag, Basel, 2006.
%
\bibitem[BB1]{BBlaC}
H.-J.~Baues \& D.~Blanc
``Comparing cohomology obstructions'',\hsm
\textit{J.\ Pure \& Appl.\ Alg.} \textbf{215} (2011), pp.~1420-1439.
%
\bibitem[BB2]{BBlaS}
H.-J.~Baues \& D.~Blanc,
``Stems and Spectral Sequences'',\hsm
\textit{Alg.\ \& Geom.\ Top.}, \textbf{10} (2010), pp.~2061-2078.
%
\bibitem[BJ]{BJiblSL}
H.-J.~Baues \& M.A.~Jibladze,
``Suspension and loop objects in theories and cohomology'',\hsm
\textit{Georgian Math.\ J.} \textbf{8} (2001), pp.~697-712.
%
\bibitem[Be]{BecT}
J.M.~Beck,
``Triples, algebras and cohomology'',
\textit{Repr.\ Theory Appl.\ Cats.} \textbf{2} (2003), pp.~1-59.
%
\bibitem[Bl1]{BlaHH}
D.~Blanc,
``Higher homotopy operations and the realizability of homotopy groups'',\hsm
\textit{Proc.\ London Math.\ Soc.} \textbf{70} (1995), pp.~214-240.
%
\bibitem[Bl2]{BlaAI}
D.~Blanc,
``Algebraic invariants for homotopy types'',
\textit{Math.\ Proc.\ Camb.\ Phil.\ Soc.} \textbf{127} (1999), pp.~497-523.
%
\bibitem[Bl3]{BlaCW}
D.~Blanc,
``CW simplicial resolutions of spaces, with an application to loop spaces'',\hsm
\textit{Top.\ \& Appl.} \textbf{100} (2000), pp.~151-175.
%
\bibitem[Bl4]{BlaCH}
D.~Blanc,
``Comparing homotopy categories'',\hsm
\textit{J.~$K$-Theory} \textbf{2} (2008), pp.~169-205.
%
\bibitem[Bl5]{BlaQ}
D.~Blanc,
``Generalized Andr\'{e}-Quillen Cohomology'',\hsm
\textit{J.~Homotopy \& Rel.\ Structures} \textbf{3} (2008), pp.~161-191.
%
\bibitem[BC]{BChachP}
D.~Blanc \& W.~Chach{\'{o}}lski,
``Pointed higher homotopy operations'',\hsm
preprint, 2011.
%
\bibitem[BDG]{BDGoeR}
D.~Blanc, W.G.~Dywer, \& P.G.~Goerss,
``The realization space of a $\Pi$-algebra: a moduli problem in algebraic
topology'',\hsm
\textit{Topology} \textbf{43} (2004), pp.~857-892.
%
\bibitem[BJT1]{BJTurR}
D.~Blanc, M.W.~Johnson, \& J.M.~Turner,
``On realizing diagrams of $\Pi$-algebras'',\hsm
\textit{Alg.\ \& Geom.\ Top.} \textbf{6} (2006), pp.~763-807.
%
\bibitem[BJT2]{BJTurH}
D.~Blanc, M.W.~Johnson, \& J.M.~Turner,
``Higher homotopy operations and cohomology'',\hsm
\textit{J.\ $K$-Theory} \textbf{5} (2010), pp.~167-200.
%
\bibitem[BM]{BMarkH}
D.~Blanc \& M.~Markl,
``Higher homotopy operations'',\hsm
\noindent{Math.\ Zeit.} \textbf{345} (2003), pp.~1-29.
%
\bibitem[BV]{BVogHI}
J.M.~Boardman \& R.M.~Vogt,
\textit{Homotopy Invariant Algebraic Structures on Topological Spaces},\hsm
Springer-\-Verlag \textit{Lec.\ Notes Math.} \textbf{347},
Berlin-\-New York, 1973.
%
\bibitem[Bor]{BorcH2}
F.~Borceux,
\textit{Handbook of Categorical Algebra, Vol.\ 2: Categories and Structures},\hsm
Encyc.\ Math.\ \& its Appl.\ \textbf{51},
Cambridge U.\ Press, Cambridge, UK, 1994.
%
\bibitem[Bou]{BouC}
A.K.~Bousfield,
``Cosimplicial resolutions and homotopy spectral sequences in model
categories,''\hsm
\textit{Geometry and Topology} \textbf{7} (2003), pp.~1001-1053.
%
\bibitem[BF]{BFrieH}
A.K.~Bousfield \& E.M.~Friedlander,
``Homotopy theory of $\Gamma$-spaces, spectra, and bisimplicial sets'',\hsm
in M.G.~Barratt \& M.E.~Mahowald, eds.,
\textit{Geometric Applications of Homotopy Theory, II}
Springer-\-Verlag \textit{Lec.\ Notes Math.} \textbf{658},
Berlin-\-New York, 1978, pp.~80-130
%
\bibitem[D]{DolH}
A.~Dold,
``Homology of symmetric products and other functors of complexes'',\hsm
\textit{Ann.\ Math.\ (2)} \textbf{68} (1958), pp.~54-80.
%
\bibitem[DK1]{DKanS}
W.G.~Dwyer \& D.M.~Kan,
``Simplicial localizations of categories'',\hsm
\textit{J.\ Pure \& Appl.\ Alg.} \textbf{17} (1980), No.~3, pp.~267-284.
%
\bibitem[DK2]{DKanF}
W.G.~Dwyer \& D.M.~Kan,
``Function complexes in homotopical algebra'',\hsm
\textit{Topology} \textbf{19} (1980), pp.~427-440.
%
\bibitem[DKSm1]{DKSmO}
W.G.~Dwyer, D.M.~Kan, \& J.H.~Smith,
``An obstruction theory for simplicial categories'',\hsm
\textit{Proc.\ Kon.\ Ned.\ Akad.\ Wet.\ - Ind.\ Math.} \textbf{89} (1986)
No.~2, pp.~153-161.
%
\bibitem[DKSm2]{DKSmH}
W.G.~Dwyer, D.M.~Kan, \& J.H.~Smith,
``Homotopy commutative diagrams and their realizations'',\hsm
\textit{J.\ Pure \& Appl.\ Alg.}, \textbf{57} (1989), No.~1, pp.~5-24.
%
\bibitem[DKSt1]{DKStE}
W.G.~Dwyer, D.M.~Kan, \& C.R.~Stover,
``An $E^{2}$ model category structure for pointed simplicial spaces'',\hsm
\textit{J.\ Pure \& Appl.\ Alg.} \textbf{90} (1993), No.~2, pp.~137-152.
%
\bibitem[DKSt2]{DKStB}
W.G.~Dwyer, D.M.~Kan, \& C.R.~Stover,
``The bigraded homotopy groups \w{\pi_{i,j}X} for a pointed simplicial
space $X$'',\hsm
\textit{J.\ Pure Appl.\ Alg.} \textbf{103} (1995), pp.~167-188.
%
\bibitem[E]{EhreET}
C.~Ehresmann,
``Esquisses et types des structures alg{\'{e}}briques'',\hsm
\textit{Bul. Inst. Politehn. Ia\c{s}i (N.S.)} \textbf{14} (1968), pp.~1-14.
%
\bibitem[G]{Goe}
P.~Goerss, 
``(Pre-)sheaves of Ring Spectra over the Moduli Stack of 
Formal Group Laws'' 
in J.~Greenlees, ed.,\textit{Axiomatic, enriched, and motivic homotopy theory}, 
Kluwer Academic Publishers (2004), 101-131.
%
\bibitem[GH]{GHopM}
P.~Goerss \& M.~Hopkins,
"Moduli spaces of commutative ring spectra," \hsm
\textit{Structured Ring Spectra}, London Math. Soc. Lecture Note Ser., 
\textbf{315}, Cambridge Univ. Press, Cambridge, 2004, pp.~151-200.
%
\bibitem[GR]{GRosA}
G.T.~Guilbaud \& P.~Rosenstiehl,
``Analyse algebrique d'un scrutin'',\hsm
\textit{Math.\ Sci.\ Humaines} \textbf{4} (1960), pp.~9-33.
%
\bibitem[Gr]{GrunC}
B.~Gr\"{u}nbaum,
\textit{Convex polytopes},\hsm
Springer-\-Verlag \textit{Grad.\ Texts Math.} \textbf{221}, second ed.,
Berlin-\-New York, 2003.
%
\bibitem[Ha]{HarpS}
J.R.~Harper, 
\textit{Secondary cohomology operations,}
Grad.\ Stud.\ Math.\ \textbf{49},  AMS, Providence, RI, 2002.
%
\bibitem[Hil]{HilH}
P.J.~Hilton,
 ``On the homotopy groups of the union of spheres'',\hsm
\textit{J.\ Lond.\ Math.\ Soc.} \textbf{30} (1955) pp.~154-172.
%
\bibitem[Hir]{PHirM}
P.S.~Hirschhorn,
\textit{Model Categories and their Localizations},\hsm
Math.\ Surveys \& Monographs \textbf{99}, AMS, Providence, RI, 2002.
%
\bibitem[Hop]{HopfT}
H.~Hopf,
``\"{U}ber die Topologie der Gruppen-Mannigfaltkeiten und ihre 
Verallgemeinerungen'',\hsm
\textit{Ann.\ Math.\ (2)} \textbf{42} (1941), pp.\ 22-52.
%
\bibitem[Hov]{HovM}
M.A.~Hovey,
\textit{Model Categories},\hsm
Math.\ Surveys \& Monographs \textbf{63}, AMS, Providence, RI, 1998.
%
\bibitem[J]{JardB}
J.F.~Jardine,
``Bousfield's $E\sb{2}$ Model Theory for Simplicial Objects,''\hsm
in P.G. Goerss \& S.B. Priddy, eds.,
\textit{Homotopy Theory: Relations with Algebraic Geometry, Group
Cohomology, and Algebraic $K$-Theory},
Contemp.\ Math. \textbf{346}, AMS, Providence, RI 2004, pp.~305-319.
%
\bibitem[K]{KuhnTR}
N.J.~Kuhn,
``On topologically realizing modules over the {Steenrod} algebra'',\hsm
\textit{Ann. Math., Ser. 2} \textbf{141} (1995), No.\ 2, pp.\ 321-347.
%
\bibitem[Q1]{QuiH}
D.G.~Quillen,
\textit{Homotopical Algebra},\hsm
Springer-\-Verlag \textit{Lec.\ Notes Math.} \textbf{20},
Berlin-\-New York, 1963.
%
\bibitem[Q2]{QuiR}
D.G.~Quillen,
``Rational homotopy theory'',\hsm
\textit{Ann.\ Math.} \textbf{90} (1969), pp.~205-295.
%
\bibitem[Q3]{QuiC}
D.G.~Quillen,
``On the (co-)homology of commutative rings'',\hsm
\textit{Applications of Categorical Algebra}, \ Proc.\ Symp.\ Pure Math.\
\textbf{17}, AMS, Providence, RI, 1970, pp.~65-87.
%
\bibitem[Sc]{SchwU}
L.~Schwartz,
\textit{Unstable Modules over the Steenrod Algebra and Sullivan's Fixed Point 
Set Conjecture},\hsm
U. Chicago Press, Chicago-\-London, 1994.
%
\bibitem[Sp]{SpaS}
E.H.~Spanier,
``Secondary operations on mappings and cohomology'',\hsm
\textit{Ann.\ Math.\ (2)} \textbf{75} (1962) No.\ 2, pp.\ 260-282.
%
\bibitem[Ste]{SteCA}
N.E.~Steenrod,
``The cohomology algebra of a space'',\hsm
\textit{Ens. Math.} \textbf{7} (1961), pp.\ 153-178.
%
\bibitem[Sto]{StoV}
C.R.~Stover,
``A Van Kampen spectral sequence for higher homotopy groups'',\hsm
\textit{Topology} \textbf{29} (1990), pp.~9-26.
%
\bibitem[Ta]{TanrH}
D.~Tanr\'{e},
\textit{Homotopie Rationelle: Mod\`{e}les de Chen, Quillen, Sullivan},\hsm
Springer-\-Verlag \textit{Lec.\ Notes Math.} \textbf{1025},
Berlin-\-New York, 1983.
%
\bibitem[To]{TodaC}
H.~Toda,
\textit{Composition methods in the homotopy groups of spheres},
Adv.\ in Math.\ Study \textbf{49}, Princeton U.\ Press, Princeton, 1962.
%
\bibitem[W]{GWalkL}
G.~Walker,
``Long Toda brackets'',\hsm
in \textit{Proceedings of the Advanced Study Institute on Algebraic 
Topology, Vol. III (Aarhus, 1970)}, Aarhus, 1970, pp.~612-631.
%
\end{thebibliography}
\end{document}